\documentclass[final]{amsart}

\usepackage{url}
\usepackage{hyperref}
\usepackage[notref,notcite]{showkeys}
\usepackage{mypreamble}

\newtheorem{thm}[subsubsection]{Theorem}
\newtheorem{lem}[subsubsection]{Lemma}
\newtheorem{prop}[subsubsection]{Proposition}
\newtheorem{cor}[subsubsection]{Corollary}

\theoremstyle{remark}
\newtheorem{rem}[subsubsection]{Remark}

\theoremstyle{definition}
\newtheorem{defn}[subsubsection]{Definition}
\newtheorem{eg}[subsubsection]{Example}

\nc\C{\eC}
\nc\eGr{\eG r}
\nc\IC{\mathrm{IC}}
\nop\Roots{Roots}
\nop\Ch{Ch}
\nop\CT{CT}
\nop\Eis{Eis}
\nop\Sat{Sat}
\nop\Ran{Ran}
\nop\coBar{CoBar}
\nop\MC{MC}
\nop\HN{HN}
\nop\Asymp{Asymp}
\nop\Act{Act}
\nop\VinBun{VinBun}
\nop\Hecke{Hecke}
\nop\HNSupp{HNSupp}
\nop\comp{comp}
\nop\conv{conv}
\nop\rank{rank}
\nc\pos{\mathrm{pos}}
\nc\Iw{\mathrm{Iw}}
\nc\ur{\mathrm{ur}}
\nc\enh{\mathrm{enh}}
\nc\adj{\mathrm{adj}}
\nc\disj{\mathrm{disj}}
\nc\naive{\mathrm{naive}}
\nc\strict{\mathrm{strict}}
\nc\ssc{\mathrm{ssc}}
\nc\true{\mathrm{true}}
\nc\ab{\mathrm{ab}}
\nc\abst{\mathrm{ab,st}}
\nop\Sph{\mathsf{Sph}}
\nc\psc{{ps\text{-}c}}
\nc\oo[1]{\overset\circ{#1}}

\DeclareMathOperator*{\tbt}{\tilde \boxtimes}

\DeclareMathOperator*{\cotimes}{\square}

\numberwithin{equation}{section}
\title[On an invariant bilinear form via asymptotics]{On an invariant bilinear form on the space of automorphic forms via asymptotics}

\author{Jonathan Wang}

\begin{document} 

\begin{abstract}
% text version:
% This article concerns the study of a new invariant bilinear form $\mathcal B$ on the space of automorphic forms of a split reductive group $G$ over a function field. We define $\mathcal B$ using the asymptotics maps from Bezrukavnikov-Kazhdan and Sakellaridis-Venkatesh, which involve the geometry of the wonderful compactification of $G$. We show that $\mathcal B$ is naturally related to miraculous duality in the geometric Langlands program through the functions-sheaves dictionary. In the proof,  we highlight the connection between the classical non-Archimedean Gindikin-Karpelevich formula and certain factorization algebras acting on geometric Eisenstein series. We then give another definition of $\mathcal B$ using the constant term operator and the inverse of the standard intertwining operator. The form $\mathcal B$ defines an invertible operator $L$ from the space of compactly supported automorphic forms to a new space of "pseudo-compactly" supported automorphic forms. We give a formula for $L^{-1}$ in terms of pseudo-Eisenstein series and constant term operators which suggests that $L^{-1}$ is an analog of the Aubert-Zelevinsky involution.
%
%MSC: 11F70, 22E50
%
This article concerns the study of a new invariant bilinear form $\eB$ 
on the space of automorphic forms of a split reductive group $G$ over a function field. 
We define $\eB$ using the asymptotics maps from \cite{BK, SV}, which involve the geometry of the wonderful compactification of $G$. 
We show that $\eB$ is naturally related to miraculous duality in the geometric Langlands program through the functions--sheaves dictionary. 
In the proof,  we highlight the connection between the classical non-Archimedean Gindikin--Karpelevich formula and certain factorization algebras acting on geometric Eisenstein series. 
We then give another definition of $\eB$ using the constant term operator and the inverse of the standard intertwining operator. 
The form $\eB$ defines an invertible operator $L$ from the space of compactly supported automorphic forms to a new space of ``pseudo-compactly'' supported automorphic forms. 
We give a formula for $L^{-1}$ in terms of pseudo-Eisenstein series and constant term operators which suggests that $L^{-1}$ is an analog of the Aubert--Zelevinsky involution.
\end{abstract}

\maketitle

\tableofcontents

\section{Introduction} 

\subsection{The goal of this paper} In \cite{DW}, an invariant
symmetric bilinear form $\eB$ is defined on the space of 
automorphic forms for $\on{SL}(2)$ over any global field. 
The goal of this paper is to generalize the definition of $\eB$ and the corresponding theory to any split reductive group $G$ over a function field. 
The study of $\eB$ is motivated by works \cite{DG:CG, G:miraculous} on the geometric Langlands program. There is also a significant connection between $\eB$ and the theory of Eisenstein series, as evidenced by \cite{DW}.

\medskip

Let $G$ be a split reductive group over $\bbF_q$.
Let $X$ be a geometrically connected smooth projective curve over a finite field $\bbF_q$, and let $F$ be the field of rational functions on $X$. 
Let $\bbA$ denote the adele ring of $F$. 

For any place $v$ of $F$, the completion of $F$ with respect to $v$ 
will be denoted $F_v$. Let $\fo_v$ denote the ring of integers of $F_v$, with 
residue field $\bbF_{q_v}$. We denote the standard maximal compact subgroup of $G(F_v)$
by $K_v$. Set $K := \prod_v K_v$; this is a maximal compact subgroup of $G(\bbA)$.

We fix a field $E$ of characteristic $0$. Unless specified otherwise, all functions will take values in $E$. 

%\subsubsection{} 
\label{sss:A_c}
Let $\eA$ denote the space of $K$-finite $C^\infty$ functions on $G(\bbA)/G(F)$.
Let $\eA_c \subset \eA$ denote the subspace of functions with compact support.

\medskip

In this paper we define and study a $G(\bbA)$-invariant
symmetric bilinear form $\eB$ on $\eA_c$. 
(The definition of $\eB$ is given in \S\ref{ss:defB}.) 
Fix a Haar measure on $G(\bbA)$. 
The form $\eB$ is defined as an alternating sum of invariant bilinear
forms $\eB_P$ on $\eA_c$, where the sum ranges over the conjugacy classes
of parabolic subgroups of $G$. When $P=G$, the form $\eB_G$ is the
naive pairing 
\[ \eB_\naive(f_1,f_2) = \int_{G(\bbA)/G(F)} f_1(x) f_2(x) dx, \quad\quad 
f_1,f_2 \in \eA_c.\] 
The definition of $\eB_P$ was suggested by Y.~Sakellaridis in a private communication, and it uses the \emph{local} asymptotics maps constructed in \cite{BK,SV} using the geometry of the De Concini--Procesi wonderful compactification of $G$. 
The asymptotics map is defined in the more general setting of harmonic analysis on spherical varieties in \cite{SV}. In \cite{BK}, the asymptotics map is used to give a geometric proof of Bernstein's theorem on second adjointness. 
It is also shown (\cite[Theorem 7.6]{BK}) that the asymptotics map is inverse to the standard (long) intertwining operator in the classical representation theory of $\fp$-adic groups. 
Using this relationship, one sees that the computation of the asymptotics of the characteristic function of $K_v$ (cf.~\cite[\S 6]{Sak}) goes back to the classical non-Archimedean Gindikin--Karpelevich formula due to \cite{L:Euler, MacDonald}. 

\medskip

In order to study the form $\eB$, we consider certain subspaces $\C_{P,\pm}$ 
of the space of smooth $K$-finite functions on $G(\bbA)/M(F)U(\bbA)$, where $P = MU$
is a standard parabolic subgroup with Levi subgroup $M$ and unipotent radical $U$.
The spaces $\C_{P,\pm}$ may be of independent interest as they are defined with 
respect to the same rational cones and support conditions as in the definition of Arthur's truncation operator (cf.~the definition of $\what \tau_P$ in \cite[\S 6]{Arthur}).
In Proposition~\ref{prop:I^-1}, we prove that the standard intertwining operator 
extends to an isomorphism $R_P : \C_{P^-,+} \to \C_{P,-}$.

\begin{rem}
We only consider the function field case in this paper, but the reader may check that the definition of $\eB$ on $K$-invariant automorphic forms extends to the number field case using the Archimedean Gindikin--Karpelevich formula. 
We hope to define $\eB$ on the whole space $\eA_c$ for an arbitrary global field $F$ by better understanding the local Archimedean intertwining operator in the future.
\end{rem}

\subsection{Motivation from geometric Langlands} 
Let us explain the motivation for the existence of $\eB$ from the geometric Langlands program. Here we assume that the field $E$ equals $\wbar\bbQ_\ell$ for a prime $\ell$ coprime to the characteristic of $F$.

\subsubsection{A remarkable $\ell$-adic complex on $\Bun_G \xt \Bun_G$}
Let $\Bun_G$ denote the stack of $G$-bundles on $X$. 
Let $\Delta : \Bun_G \to \Bun_G \xt \Bun_G$ be the diagonal morphism. 
We have the $\ell$-adic complex $\Delta_*(\wbar\bbQ_\ell)$ on $\Bun_G \xt \Bun_G$. 

This complex is the $\ell$-adic analog of the complex of D-modules $\Delta_! \omega_{\Bun_G}$, which plays a crucial role in the theory of \emph{miraculous duality} on $\Bun_G$, which was developed in \cite[\S 4.5]{DG:CG} and \cite{G:miraculous}. Assume for the moment that $X$ is over a ground field $k$ of characteristic $0$. Then miraculous duality gives an equivalence between the DG category of (complexes of) D-modules on $\Bun_G$ and its Lurie dual. 
Very roughly, the equivalence is defined as the functor 
\[ \on{Ps-Id}_{\Bun_G,!} : \on{D-mod}(\Bun_G)_{\on{co}} \to \on{D-mod}(\Bun_G) \]
given by the kernel $\Delta_!\omega_{\Bun_G}$ (whereas the identity functor is given by the kernel $\Delta_*\omega_{\Bun_G}$). The fact that this functor is an equivalence is a highly nontrivial theorem \cite[Theorem 0.2.4]{G:miraculous}. 

\subsubsection{The function $b$} 
Given $G$-bundles $\eF^1_G, \eF^2_G \in \Bun_G(\bbF_q)$, let 
$b(\eF^1_G, \eF^2_G)$ denote the trace of the geometric Frobenius 
acting on the $*$-stalk of the complex $\Delta_*(\wbar\bbQ_\ell)$ over the point
$(\eF^1_G, \eF^2_G) \in (\Bun_G \xt \Bun_G)(\wbar\bbF_q)$.
Using results of \cite{Schieder:gen}, we deduce a formula 
for $b$ in terms of the asymptotics maps (see Theorem~\ref{thm:b}). 

\subsubsection{Relation between $\eB$ and $b$} 
The quotient $K \bs G(\bbA) / G(F)$ identifies with $\abs{\Bun_G(\bbF_q)}$, the set of isomorphism classes of $G$-bundles on $X$. So the function $b$ can be considered as a function on $(G(\bbA)/G(F)) \xt (G(\bbA)/G(F))$. The following theorem is one of our main results. The proof is given in \S\ref{ss:geom}.

\begin{thm}  \label{thm:B=b}
Let $E = \wbar \bbQ_\ell$ for $\ell$ coprime to the characteristic of $F$.
Normalize the Haar measure on $G(\bbA)$ so that $K$ has measure $1$. 
Then for any $f_1, f_2 \in \eA_c^K$, one has 
\begin{equation}
\eB(f_1,f_2) = \int_{(G \xt G)(\bbA) / (G \xt G)(F)} b(g_1,g_2) f_1(g_1) f_2(g_2) dg_1 dg_2. 
\end{equation}
\end{thm}

By non-degeneracy of the naive pairing $\eB_\naive$, defining the bilinear form $\eB$ is
equivalent to defining an operator $L : \eA_c \to \eA$ such that 
\[ \eB(f_1,f_2) = \eB_\naive(Lf_1, f_2), \quad\quad f_1,f_2\in \eA_c. \]
Theorem~\ref{thm:B=b} implies that the miraculous duality functor $\on{Ps-Id}_{\Bun_G,!}$ 
is a D-module analog of the operator $q^{-\dim \Bun_G} L$ via the functions--sheaves 
dictionary.

\subsection{Analog of the Aubert--Zelevinsky involution}
In \S\ref{ss:A_psc}, we define a subspace $\eA_\psc \subset \eA$ of ``pseudo-compactly'' supported functions using the constant term operators and the spaces $\C_{P,+}$. 
We prove that the operator $L$ above sends $\eA_c$ to $\eA_\psc$, and the operator $L:\eA_c \to \eA_\psc$ 
is an isomorphism (Theorem~\ref{thm:L^-1}). 
The invertibility of $L$ may be considered as a function-theoretic analog of 
the main result (Theorem 0.2.4) of \cite{G:miraculous}. 

Moreover, we give an explicit formula 
\[ L^{-1} f = \sum_P (-1)^{\dim Z(M)} (\Eis_P \circ \CT_P)(f), \quad\quad f \in \eA_\psc, \] 
for the inverse, where $\Eis_P, \CT_P$ denote respectively, the (pseudo-)Eisenstein operator and constant term operator. 
By considering $\Eis_P,\CT_P$ as global analogs of the parabolic induction functor and Jacquet functor, respectively, in the theory of smooth representations of a $\fp$-adic group, one can view 
the formula for $L^{-1}$ as an analog of the formula for the Aubert--Zelevinsky involution
on the Grothendieck group of smooth representations of finite length (see Remark~\ref{rem:Aubert}). 
This involution was first defined
and studied for $G = \on{GL}(n)$ by Zelevsinky \cite{Zelevsinky} and later for general reductive groups by Aubert \cite{Aubert}. On Iwahori fixed vectors, it also corresponds to the Iwahori--Matsumoto involution (cf.~\cite{KatoS}). 
There is an analogous involution for representations of a finite Chevalley group, often called the Alvis--Curtis involution, which was studied earlier in \cite{Alvis,Curtis}. 

The Aubert--Zelevinsky involution can be studied at the level of complexes. 
Such complexes were considered in \cite{DL} for representations of a finite Chevalley group.
For every smooth representation $M$ one can form a complex 
\[ 0 \to M \to \bigoplus_P i^G_P r^G_P (M) \to \dotsb \to i^G_B r^G_B (M) \to 0 \]
where $i^G_P, r^G_P$ denote, respectively, the parabolic induction and Jacquet functors, and the sum in the $i$-th term runs over standard parabolic subgroups of corank $i$ in $G$. 
We call this complex the Deligne--Lusztig complex associated to $M$ and denote it by 
$\on{DL}(M)$. 
Aubert showed that for an irreducible module $M$, the complex $\on{DL}(M)$ has
cohomology in only one degree, which implies that the Aubert--Zelevsinky involution sends irreducible modules to irreducible modules (up to a sign). 
A new proof of this result was recently given in \cite{BBK:duality} using asymptotics maps and the geometry of the wonderful compactification of $G$.

\subsection{Structure of the paper}
\subsubsection{General remark} In the main body of the article we work with \emph{classical} functions on $G(F_v)$ and $G(\bbA)/G(F)$. 
These are, however, heavily motivated by geometric definitions and results appearing in the \emph{geometric} Langlands program. We review the relevant geometry in Appendices \ref{s:appendixHecke}--\ref{sect:DLV}. 

\subsubsection{The main body of the paper}

In Section \ref{s:local}, we study the asymptotics map and its relation to the intertwining operator over a local non-Archimedean field. In order to elucidate the support conditions of various functions, we give a combinatorial description of the bounded subsets of the boundary degenerations of $G$.
\medskip

In Section \ref{s:Kinvariants}, we compute the asymptotics of the characteristic function of $K_v$ by reducing to the non-Archimedean Gindikin--Karpelevich formula using intertwining operators on $K_v$-invariants. To do so, we extend the classical Satake isomorphism to an isomorphism between certain completed Hecke algebras. 
\medskip

In Section \ref{s:B}, we define the bilinear form $\eB$. 
After giving a geometric interpretation of the restriction of $\eB$
to $\eA_c^K$, we prove Theorem~\ref{thm:B=b}. 

\medskip

The definition of $\eB$ we give differs slightly from the definition given in \cite{DW} for $G = \on{SL}(2)$. In our definition, we use local asymptotics (which is essentially equivalent to local inverse intertwining operators) and then apply a local-to-global procedure. 

\medskip

In Section \ref{s:global}, we provide an alternate definition of $\eB$, which directly generalizes the one in \cite{DW}. 
For a parabolic subgroup $P$ with Levi factor $M$, we define subspaces $\C_{P,\pm}$ of the space of $K$-finite $C^\infty$ functions on $G(\bbA)/M(F)U(\bbA)$.
The definitions are such that the constant term operator $\CT_P$ (whose definition we recall) sends $\eA_c$ to $\C_{P,-}$. The intertwining operator $R_P$ (which is of local nature) is defined as a map $\C_{P^-,+} \to \C_{P,-}$, and we show that $R_P$ is an isomorphism. Let $\brac{\, ,}$ denote the natural 
pairing between functions in $\C_{P^-}$ (when convergent). We prove the following in \S\ref{ss:altdef}: 
\begin{thm} \label{thm:global}
For any $f_1,f_2 \in \eA_c$, one has 
\begin{equation}
\eB(f_1,f_2) = \sum_P (-1)^{\dim Z(M)} \brac{ R_P^{-1} \CT_P(f_1), \CT_{P^-}(f_2) },
\end{equation}
where the sum ranges over conjugacy classes of parabolic subgroups of $G$. 
\end{thm}

\medskip

In Section \ref{s:L}, we use Theorem~\ref{thm:global} to define the operator $L :\eA_c \to \eA$
and the subspace $\eA_\psc \subset \eA$ of ``pseudo-compactly'' supported functions. 
We show that $L$ sends $\eA_c$ to $\eA_\psc$, and in Theorem~\ref{thm:L^-1} we prove that
the operator $L : \eA_c \to \eA_\psc$ is invertible. We give a formula~\eqref{e:Linverse}
for $L^{-1}$, which is in fact simpler than the formula for $L$. 
This formula may be viewed as an analog of the definition of the Aubert--Zelevinsky involution. 

\subsubsection{Appendices \ref{s:appendixHecke}--\ref{sect:DLV}}
In Appendix \ref{s:appendixHecke}, we consider the global model
for the formal arc space of a group embedding into an algebraic monoid. 
This model was also used in \cite[\S 2]{BNY}. We realize the global model
as a substack of a symmetrized version of the Hecke stack. 
We give a bound on the difference of the Harder--Narasimhan coweights of 
the two bundles corresponding to a point of the Hecke stack (Lemma~\ref{lem:heckeHN}). In this article, we are primarily interested in the stack $\eH^+_M$ attached to the monoid $\wbar M$, defined as the closure of $M$ in the affine closure of $G/U$, where $P$ is a parabolic subgroup with Levi factor $M$.
The stack $\eH^+_M$ is a graded Ran version (in the sense of \cite{Whatacts}) of the closed substack of the Hecke stack studied in \cite{BG} and \cite[\S 1.8]{BFGM}. 
The Hecke stack is a twisted product of $\Bun_M$ and 
the Beilinson--Drinfeld (factorizable) affine Grassmannian, and it is more convenient to use the latter to talk about factorization properties.
We briefly review the relevant notation and properties of the factorizable affine Grassmannian -- we use a symmetrized version that does not explicitly mention the Ran space. 

\medskip

In Appendix \ref{s:factorization}, we review the definition of the factorization algebras on the affine Grassmannian introduced in \cite{BG2, Whatacts} that act on geometric Eisenstein series. The main goal of this Appendix is to highlight the connection (via Grothendieck's functions--sheaves dictionary) between certain measures (related to unramified intertwining operators) appearing in the classical non-Archimedean Gindikin--Karpelevich formula 
and Gaitsgory's factorization algebras (see Proposition~\ref{prop:geom-nu}, Lemma~\ref{lem:f_Omega}). 
From this perspective, we point out how the main theorem of \cite{BFGM} may be interpreted as a categorical or geometric version of (Langlands' interpretation of) the Gindikin--Karpelevich formula. 

\medskip

In Appendix \ref{sect:DLV}, we study the compactification of the diagonal morphism of $\Bun_G$ using the results of \cite{Schieder:gen}. The compactification $\wbar\Bun_G$ we define is slightly different from the one found in the literature. We review the definition and relevant properties of the Drinfeld--Lafforgue--Vinberg degeneration of $\Bun_G$. In particular we highlight the connection between the geometric Bernstein asymptotics studied in \emph{loc.~cit.}~and Gaitsgory's factorization algebras to deduce that for an arbitrary parabolic subgroup, the geometric Bernstein asymptotics corresponds to the classical asymptotics of the characteristic function of $K$ via the functions--sheaves dictionary.

\subsection{Conventions}

Throughout the paper, $G$ will be a connected split reductive group over $\bbF_q$.  
Fix a split torus $T \subset G$ and a Borel $B$ containing $T$. 
Let $W$ be the Weyl group of $T$.
Let $\check \Lambda$ (resp.~$\Lambda$) denote the weight (resp.~coweight) lattice of $T$.

The monoid of dominant weights (resp., coweights) will be denoted
by $\check\Lambda^+_G$ (resp., by $\Lambda_G^+$).
The set of vertices of the Dynkin diagram of $G$ will be denoted by $\Gamma_G$; for
each $i \in \Gamma_G$ there corresponds a simple coroot $\alpha_i$
and a simple root $\check \alpha_i$. 
The set of coroots (resp.~positive coroots) will be denoted by $\Phi_G$ (resp.~$\Phi_G^+$) and the positive span of $\Phi_G^+$ inside $\Lambda$ by $\Lambda^\pos_G$.
Let $\check\Delta_G$ (resp.~$\check\Phi_G^+,\,\check\Phi_G^-,\,\check\Phi_G$) denote the simple (resp.~positive, negative, all) roots of $G$.
By $2\check\rho \in \check\Lambda$ (resp.~$2\rho \in \Lambda$) we will denote the sum of 
the positive roots (coroots) of $G$ and by $w_0$ the longest element in the Weyl group of $G$.
For $\lambda,\mu \in \Lambda$ we will write that $\lambda \ge \mu$
if $\lambda - \mu \in \Lambda^\pos_G$, and similarly for $\check\Lambda_G^\pos$.
%Let $\check T$ be the complex Langlands dual group of $T$.

We will only consider parabolic subgroups that contain $T$.
Let $P$ be a standard\footnote{Recall that in any conjugacy class of parabolic subgroups of $G$, there is exactly one standard parabolic subgroup.}
 parabolic subgroup, i.e., $P$ contains $B$.
Then the Levi quotient can be canonically realized as a subgroup 
$M \subset P$. We have $P = M U$ where $U$ is the unipotent radical of $P$. 
There is a unique parabolic $P^-$ such that
$P \cap P^- = M$. 
To $M$ there corresponds a subdiagram $\Gamma_M \subset \Gamma_G$, coroots $\Phi_M \subset \Phi_G$, and positive coroots $\Phi^+_M \subset \Phi^+_G$.
We will denote by $\Lambda^+_M \supset \Lambda^+_G$, $\Lambda^\pos_M \subset
\Lambda_G^\pos, 2\check\rho_M \in \check \Lambda$, $\ge_M$, etc. the
corresponding objects for $M$.

Let $\Rep(G)$ denote the abelian category of finite-dimensional $G$-modules.

Given two $G$-spaces $Y,Z$ such that the diagonal action of $G$ on $Y \xt Z$ is free, we let $Y \xt^G Z$ denote the quotient of $Y \xt Z$ by the diagonal $G$-action.

For a scheme or stack $\eY$, we let $D(\eY)$ denote the DG category of bounded constructible $\wbar\bbQ_\ell$-sheaves on $\eY$. 
We will use `sheaf' to mean a complex of sheaves. 
All functors between sheaves are derived functors. 
When $\eY$ is a stack over $\spec \bbF_q$, we assume that $\ell$ is coprime to $q$. Choose a square root of $q$ in $\wbar\bbQ_\ell$ once and for all. 
The intersection cohomology sheaves are normalized so that they are pure of weight $0$. In other words, for a smooth $\bbF_q$-stack $\eY$ of dimension $n$, $\IC_\eY \cong (\wbar\bbQ_\ell(\frac 1 2)[1])^{\ot n}$.

\subsection{Acknowledgments}
This research is partially supported by the Department of Defense (DoD) through the National Defense Science and Engineering Graduate Fellowship (NDSEG) Program.
I am extremely grateful to my doctoral advisor V.~Drinfeld for his continual guidance and support throughout this project. 
I thank R.~Bezrukavnikov and Y.~Sakellaridis for many helpful discussions about their works. I also thank D.~Gaitsgory and S.~Raskin for explaining much about factorization algebras and the related derived algebraic geometry to me. I thank S.~Schieder for sharing his results with me and for many conversations about $\VinBun_G$.

\section{Local intertwining operators and asymptotics} \label{s:local}
In this section, we work over a non-Archimedean local field $F_v$,
and $G$ is a connected split reductive group over $F_v$. 
The subscript $v$ is only present to keep notation consistent throughout this article -- the presence of a global field is not assumed, and the characteristic of $F_v$ is arbitrary (and possibly zero). 

\medskip

Let $\abs{}_v$ denote the absolute value on $F_v$, let $\fo_v$ denote the
ring of integers of $F_v$, and let $q_v$ be the cardinality of the residue field.
We will use $G,\, P,\, \bbX_P$, etc.~to also denote
the topological groups/spaces of $F_v$-points of the corresponding algebraic groups or varieties, 
e.g., $G = G(F_v),\, P = P(F_v),\, \bbX_P = \bbX_P(F_v)$.
Let $K = K_v$ denote the standard maximal compact subgroup
of $G$, and $K_M$ denotes the standard maximal compact subgroup of $M$. 

\medskip

In \S\ref{ss:bd}--\ref{ss:combinatorics}, we define the space $\bbX_P$ and describe how to consider bounded subsets of $G/U$ and $\bbX_P$ 
in terms of subsets of the lattice $\Lambda$. 
In \S\ref{ss:CXMY}--\ref{ss:asymp}, we review some definitions and results from \cite{BK} to introduce the local asymptotics map $\Asymp_P$, which is ``essentially the same'' as the inverse of the standard intertwining operator. 
We observe that $\Asymp_{P}$ is determined by a generalized function $\xi_{P}$ on $\bbX_P$. 
In \S\ref{ss:CPv}--\ref{ss:R_Pv^-1}, we give a formula for the inverse of the intertwining operator in terms of $\xi_{P}$. 

\subsection{Bounded sets} \label{ss:bd}
Let $\bbX$ be a quasi-affine variety over $F_v$ (i.e., there exists
a locally closed embedding of $\bbX$ into a finite dimensional affine space). 
We say that a subset $S \subset \bbX(F_v)$ is \emph{bounded} if the following
equivalent conditions are satisfied:

(i) for any regular function $f \in F_v[\bbX]:=\Gamma(\bbX,\eO_\bbX)$, the function $\abs{f}_v$ is bounded on $S$,

(ii) for any locally closed embedding (in the sense of algebraic geometry) 
of $\bbX$ into an affine space, the image of $S$ is
bounded (with respect to the norm induced by the absolute value on $F_v$),

(iii) for any open embedding (in the sense of algebraic geometry) of $X$ into 
an affine variety, the image of $S$ is
relatively compact (for the ``usual'' topology induced by the topology on $F_v$).

\subsubsection{}
Recall that an $F_v$-scheme $\bbX$ is \emph{strongly quasi-affine} if 
the canonical morphism \[ \bbX \to \spec F_v[\bbX] \] is an open embedding
and $F_v[\bbX]$ is a finitely generated $F_v$-algebra.
For a strongly quasi-affine variety $\bbX$, 
in condition (iii) it suffices to consider only the open embedding $\bbX \into \spec F_v[\bbX]$.

\subsection{The strongly quasi-affine varieties $G/U$ and $\bbX_P$}
Fix a standard parabolic subgroup $P \subset G$ with Levi subgroup $M$ and unipotent radical $U$.  

\subsubsection{}
The quotient varieties $G/U$ and $G/U^-$ are strongly quasi-affine by
\cite{Grosshans}. Let $\wbar{G/U}:= \spec F_v[G/U]$ and $\wbar{G/U^-} := \spec F_v[G/U^-]$ denote the affine closures.

\smallskip

We review the definition of the variety $\bbX_P$ introduced in \cite[\S 2.2.1]{BK} below.

\subsubsection{} \label{sss:X_P}
Define the \emph{boundary degeneration} 
\[ \bbX_P := (G \xt G)/(P \xt_M P^-) = (G/U \xt G/U^-)/M, \] 
where $M$ acts diagonally on the right. 
Recall (cf.~\cite[Proposition~2.4.4]{DG:CT}) that $\bbX_P$ is a quasi-affine variety; let $\wbar \bbX_P := \spec F_v[\bbX_P]$ denote the affine closure.
By \cite{Grosshans}, $F_v[G/U \xt G/U^-]$ is finitely generated. Therefore Hilbert's theorem on invariants implies that $F_v[\bbX_P] = F_v[G/U \xt G/U^-]^M$ is finitely generated (i.e., $\bbX_P$ is strongly quasi-affine).
Thus a subset $S \subset \bbX_P$ is bounded if and only if 
$f(S) \subset F_v$ is bounded for every $f \in F_v[\bbX_P]$.

\subsection{Combinatorial setup} \label{ss:combinatorics}

We give a combinatorial description of bounded subsets of $\bbX_P$ in Proposition~\ref{prop:Xbound} below. 

\subsubsection{} 
By the Cartan decomposition, $K_M \bs M / K_M = (T/K_T) / W_M$. 
We have an isomorphism
\[ \ord_T : T/K_T \to \Lambda \] 
sending $\lambda(x) \mapsto \lambda \ot (-\log_{q_v} \abs{x}_v)$ where $\lambda \in \Lambda,\, x \in F^\times_v$. 
This induces an isomorphism 
\begin{equation}\label{e:ordM}
    \ord_M : K_M \bs M / K_M \to \Lambda^+_M. 
\end{equation}

\subsubsection{}\label{sect:KorbX}
By the Iwasawa decomposition, $G = K \cdot P = K \cdot P^-$. 
Therefore \eqref{e:ordM} induces the projections
\begin{equation}\label{e:ordMG}
    \ord_M : G/U \to K \bs (G/U)/K_M = K_M \bs M / K_M = \Lambda^+_M
\end{equation}
and $\ord_M : G/U^- \to \Lambda^+_M$.
We have a left $G\xt G$-action on $\bbX_P$.
Using \eqref{e:ordM} again, we also define the projection 
\begin{equation}\label{e:ordMX} 
    \ord_M : \bbX_P \to (K \xt K)\bs \bbX_P = K_M \bs M / K_M = \Lambda^+_M, 
\end{equation}
where the first equality sends $(m_1, m_2) \mapsto m_1^{-1} m_2$ when $m_1,m_2\in M$.
%Analogously, we have a map $\ord_M : \bbY_P \to \Lambda^+_M$ 
%where $(m_1,m_2)$ identifies with $m_1^{-1} m_2$ when $m_1,m_2\in M$. 

\begin{lem} \label{lem:ordineq}
Let $g_1 \in G/U$ and $g_2 \in G/U^-$.
Consider the image of $(g_1,g_2)$ in $\bbX_P$. Then 
\[ w_0^M \ord_M(g_1,g_2) \le_M \ord_M(g_2) - \ord_M(g_1) \le_M \ord_M(g_1,g_2). \]
\end{lem}
\begin{proof}
Let $\lambda_1 = \ord_M(g_1),\, \lambda_2 = \ord_M(g_2)$, and
$\theta = \ord_M(g_1,g_2)$. It follows from the definitions that
$\theta(\varpi_v) \in K_M \lambda_1(\varpi_v)^{-1} K_M \lambda_2(\varpi_v) K_M$,
where $\varpi_v \in \fo_v$ is a uniformizer.
This is equivalent to $\lambda_2(\varpi_v) \in K_M \lambda_1(\varpi_v) K_M \theta(\varpi_v) K_M$. The usual properties of the (spherical) Hecke algebra imply that $\lambda_2 \le_M \lambda_1 + \theta$. 
Similarly, we also have $\lambda_1(\varpi_v) \in K_M \lambda_2(\varpi_v) K_M \theta(\varpi_v)^{-1} K_M$, which implies that $\lambda_1 \le_M \lambda_2 - w_0^M \theta$.
\end{proof}

\subsubsection{} \label{sss:bddbelow}
Let $\Lambda^\bbQ := \bbQ \ot_\bbZ \Lambda$. 
Let $\Lambda^{\pos,\bbQ}_G \subset \Lambda^\bbQ$ denote the rational cone 
corresponding to $\Lambda^\pos_G$.
We define the rational ordering $\le_G^\bbQ$ by
$\mu\le_G^\bbQ \lambda$ if and only if $\lambda- \mu\in \Lambda^{\pos,\bbQ}_G$. 

Let $\check\Lambda^{+,\bbQ}_G \subset \check\Lambda^\bbQ := \bbQ \ot_\bbZ \check\Lambda$ denote the rational cone corresponding to $\check\Lambda^+_G$.

\medskip

We say that a subset $S \subset \Lambda^\bbQ$ is \emph{bounded below} 
(with respect to $\le_G^\bbQ$) 
if the following equivalent conditions are satisfied:

(i) For any $\check\lambda \in \check\Lambda_G^+$, the subset 
$\check\lambda(S) \subset \bbQ$ is bounded below.

(ii) There exists a subset $S_0 \subset \Lambda^\bbQ$ with compact
closure in $\bbR \ot_\bbZ \Lambda$ such that $S \subset S_0 + \Lambda^{\pos,\bbQ}_G$. 

\noindent Define $S \subset \Lambda^\bbQ_G$ to be \emph{bounded above} if $-S$ is bounded below.

\begin{prop} \label{prop:Xbound}
	A subset $S \subset \bbX_P$ is bounded if and only if $\ord_M(S) \subset \Lambda^+_M$ is bounded above.
\end{prop}

\begin{proof}
Consider the embedding $T \into \bbX_P : t \mapsto (t,1)$ and let $\wbar T$ denote
the closure of $T$ in $\wbar\bbX_P$. 
Let $S_T \subset T$ denote the preimage of $(K \xt K)\cdot S \subset \bbX_P$
under the previous embedding. Then $S$ is bounded if and only if $S_T$ is bounded in $\wbar T$. 
Note that $S_T$ is $W_M$-stable, and $-\ord_T(S_T) = W_M \cdot \ord_M(S)$. 
It is shown in \cite[Corollary 4.1.5]{Wa}
that $F_v[\wbar T] \subset F_v[T]$ has a basis formed by the characters in $W_M \cdot \Lambda^+_G$. 
For a weight $\check\lambda\in \check\Lambda^+_G$ and $t \in T$, 
we have $-\log_{q_v} \abs{\check\lambda(t)}_v = \brac{\check\lambda, \ord_T(t)}$.
Therefore $S_T$ is bounded in $\wbar T$ if and only if $-\ord_T(S_T)$ is bounded above. Since $\ord_M(S) \subset \Lambda^+_M$, we conclude that
$W_M\cdot \ord_M(S)$ is bounded above if and only if $\ord_M(S)$ is bounded above. \end{proof}

\subsubsection{The rational cone $\Lambda^{\pos,\bbQ}_{U}$} \label{sss:LambdaU}
We introduce the rational cone $\Lambda^{\pos,\bbQ}_{U}$, which is used throughout this article, and review some of its properties, which are proved in \cite[\S 3.1.1]{Wa}. 

\medskip

Let $\Lambda^\pos_{U} \subset \Lambda$ denote the 
non-negative integral span of the positive coroots of $G$ that are not coroots of $M$.
The submonoid $\Lambda^\pos_{U}$ is stable under the actions of $W_M$. 
Let $\Lambda^{\pos,\bbQ}_{U}$ denote the corresponding rational cone.

\begin{lem} \label{lem:wthull}
	Let $\lambda,\lambda' \in \Lambda^+_M$ with $\lambda \le_M \lambda'$. 
If $\lambda' \in \Lambda^\pos_{U}$, then $\lambda \in \Lambda^\pos_{U}$.
\end{lem}

\begin{lem} \label{lem:posU}
The subset $\Lambda^\pos_{U}\subset\Lambda$ is equal to the intersection of $w(\Lambda^\pos_G)$
for all $w \in W_M$. Consequently, 
$\Lambda^\pos_{U} \cap (-\Lambda^+_M) = \Lambda^\pos_G \cap (-\Lambda^+_M)$. 
\end{lem}

\begin{rem}
Lemma~\ref{lem:posU} implies that $\Lambda^\pos_{U} \cap \Lambda^+_M
= w_0^M(\Lambda^\pos_G) \cap \Lambda^+_M$. This submonoid 
of $\Lambda^+_M$ is denoted by $\Lambda^+_{M,G}$ in \cite[\S 6.2.2, Proposition 6.2.3]{BG}.
\end{rem}

\begin{lem} \label{lem:RennerMbar}
The submonoid $W_M \cdot \check\Lambda^+_G \subset \check\Lambda$ is dual to 
$\Lambda^\pos_{U}$, i.e.,
\begin{equation} \label{eqn:sgpdual}
	W_M \cdot \check\Lambda^+_G = 
\{ \check\lambda \in \check\Lambda \mid 
\brac{\check\lambda, \mu} \ge 0 \emph{ for all } \mu \in \Lambda^\pos_{U} \}. 
\end{equation}
\end{lem}

\begin{rem}  \label{rem:barXO}
Let $\mathsf{X_P}$ denote an $\fo_v$-model of $\bbX_P$, and set $\mathsf{\wbar X_P} := \spec \Gamma(\mathsf{X_P}, \eO_{\mathsf{X_P}})$. Then $\mathsf{\wbar X_P} \xt_{\spec \fo_v} \spec F_v = \wbar \bbX_P$, and $\mathsf{\wbar X_P}(\fo_v)$ is a $K \xt K$-stable subset of $\wbar \bbX_P(F_v)$. 
The proof of Proposition~\ref{prop:Xbound} shows that 
\[ \mathsf{\wbar X_P}(\fo_v) \cap \bbX_P(F_v)  \subset \ord_M^{-1}((-\Lambda^{\pos,\bbQ}_{U}) \cap \Lambda^+_M), \]
where $\Lambda^{\pos,\bbQ}_{U}$ is the dual cone of $W_M \cdot \check\Lambda^{+,\bbQ}_G$
by Lemma~\ref{lem:RennerMbar}.
\end{rem}

\subsubsection{}
We recall the definition of the Langlands retraction 
$\fL : \Lambda^\bbQ \to \Lambda^{+,\bbQ}_G$, which goes back to \cite{L:geom}. It is defined as follows: for $\lambda \in \Lambda^\bbQ$, let $\fL(\lambda)$ be the least element\footnote{The existence of the least element is not obvious; it was proved by R.~P.~Langlands in \cite[\S 4]{L:geom}.} 
 in the set $\{ \theta \in \Lambda^{+,\bbQ}_G \mid \lambda \le_G^\bbQ \theta \}$ in the sense of the $\le_G^\bbQ$ ordering. 
We refer the reader to \cite{D:L} for further properties of the Langlands retraction.

Let $\Lambda^{+,\bbQ}_M \subset \Lambda^\bbQ$ denote the rational cone corresponding to $\Lambda^+_M$.
\begin{lem} \label{lem:coLretract}
	Let $\lambda \in \Lambda^{+,\bbQ}_M$. Then $\fL(\lambda) -\lambda\in \Lambda^{\pos,\bbQ}_{U}
	\cap (-\Lambda^{+,\bbQ}_M)$.
\end{lem}
\begin{proof}
Recall from \cite[Proposition 2.1]{D:L} that $\fL$ is piecewise linear, with linearity domains $C_J$ indexed by subsets $J \subset \Gamma_G$:  
Let $V_J^\perp = \{ \lambda \in \Lambda^\bbQ \mid \brac{\check\alpha_j, \lambda}=0,\, j\in J\}$. 
Then $C_J$ is the closed convex cone generated by $-\alpha_j, j \in J$ and 
$V_J^\perp \cap \Lambda^{+,\bbQ}_G$.

Suppose that $\lambda \in \Lambda^{+,\bbQ}_M$ lies in $C_J$. 
Then $\lambda - \fL(\lambda)$ belongs to the closed convex cone
generated by $-\alpha_j$ for $j \in J$, and $\fL(\lambda) \in V_J^\perp$ by \cite[Lemma 2.3]{D:L}. Therefore $\brac{\check\alpha_j, \lambda - \fL(\lambda)} = \brac{\check\alpha_j, \lambda} \ge 0$ for $j \in \Gamma_M \cap J$. 
Since $\brac{\check\alpha_i, \alpha_j} < 0$ for $i\in \Gamma_G - J,\, j\in J$, 
we also have $\brac{\check\alpha_i,\lambda - \fL(\lambda)} \ge 0$ for $i\in \Gamma_G - J$.
Hence $\lambda - \fL(\lambda) \in (-\Lambda^{\pos,\bbQ}_G) \cap \Lambda^{+,\bbQ}_M$. 
By Lemma~\ref{lem:posU}, we have the equality $(-\Lambda^{\pos,\bbQ}_G) \cap \Lambda^{+,\bbQ}_M= (-\Lambda^{\pos,\bbQ}_{U})\cap \Lambda^{+,\bbQ}_M$.
\end{proof}

Let $\Lambda^\bbR := \bbR \ot_\bbZ \Lambda$ and let $\Lambda^{+,\bbR}_G, \Lambda^{\pos,\bbR}_G$ denote the real cones corresponding to $\Lambda^+_G, \Lambda^\pos_G$.

\begin{cor} \label{cor:C-U}
A subset $S \subset \Lambda^{+,\bbQ}_M$ is bounded above 
if and only if there exists 
a compact subset $S_0 \subset \Lambda^{+,\bbR}_G$ such that 
$S$ is contained in the set 
$\{ \theta - \mu \mid \theta \in S_0,\, \mu \in \Lambda^{\pos,\bbQ}_{U}\}$.
\end{cor}
\begin{proof}
Suppose $S \subset \Lambda^{+,\bbQ}_M$ is bounded above. 
Then there exists a compact subset $S_1 \subset \Lambda^\bbR$ such
that $S$ is contained in $\{ \theta - \mu \mid \theta \in S_1,\, \mu \in \Lambda^{\pos,\bbQ}_G \}$. 
Let $S_0$ denote the closure of $\fL(S)$ in $\Lambda^\bbR$. 
Then $S_0$ is contained in $\Lambda^{+,\bbR}_G \cap \{ \theta - \mu \mid \theta \in S_1,\, \mu \in \Lambda^{\pos,\bbR}_G \}$, which is a compact set. 
Lemma~\ref{lem:coLretract} implies that $S$ is contained in 
$\{ \theta - \mu \mid \theta \in S_0,\, \mu \in \Lambda^{\pos,\bbQ}_{U} \}$.
The other direction is evident.
\end{proof}

\subsubsection{} \label{sss:GembedX}
The closed embedding $G \xt P^- \into G \xt G$ induces a closed embedding 
$G/U \into \bbX_P$ sending $g_1 \mapsto (g_1, 1)$. 
By \cite[Corollary 4.1.5]{Wa}, this embedding extends to a closed embedding of affine closures $\wbar{G/U} \into \wbar \bbX_P$. 
Similarly, the closed embedding $G/U^- \into \bbX_P : g_2 \mapsto (1,g_2)$
extends to a closed embedding $\wbar{G/U^-} \into \wbar\bbX_P$. 
Using these embeddings, we deduce the combinatorial description
for bounded subsets of $G/U$ and $G/U^-$ from Proposition~\ref{prop:Xbound}: 

\begin{prop} \label{prop:GPbound}
(i) A subset $S \subset G/U$ is bounded if and only if there exists a finite subset $S_0 \subset \Lambda$ such that $\ord_M(S) \subset S_0 + \Lambda^{\pos,\bbQ}_{U}$.

(ii) A subset $S \subset G/U^-$ is bounded if and only if there exists a finite subset $S_0 \subset \Lambda$ such that $\ord_M(S) \subset \{ \theta - \mu \mid \theta \in S_0,\, \mu \in \Lambda^{\pos,\bbQ}_{U} \}$. 
\end{prop}
\begin{proof}
Let $g_1 \in G/U$. Then $\ord_M(g_1) = -w_0^M \cdot \ord_M(g_1,1)$.
Using the closed embedding $\wbar{G/U} \into \wbar\bbX_P$, we deduce
(i) from Proposition~\ref{prop:Xbound} and Corollary~\ref{cor:C-U}. 
For $g_2 \in G/U^-$, we have $\ord_M(g_2) = \ord_M(1,g_2)$, so we can similarly deduce (ii) using the closed embedding $\wbar{G/U^-} \into \wbar\bbX_P$.
\end{proof}

\subsection{The space $\C_b(\bbX_P)$}\label{ss:CXMY} 

We review the definitions of the space $\C_b(\bbX_P)$ from \cite{BK} in the context of our combinatorial setup. 

\subsubsection{} 

Let $S^*(\bbX_P)$ denote the space of distributions on $\bbX_P$. 
Using our fixed choice of Haar measures, we identify distributions and generalized functions on $\bbX_P$. 
Given a generalized function $\xi \in S^*(\bbX_P)$, one can define a map
$T_\xi : C^\infty_c(G/U) \to C^\infty(G/U^-)$ % really S(G/U) \to S^*(G/U^-)
by the formula
\begin{equation} \label{e:T_xi}
    T_\xi(\varphi)(g_2) = \int_{G/U} \varphi(g_1)\xi(g_1,g_2) dg_1, \quad\quad \varphi \in C^\infty_c(G/U),\, g_2 \in G/U^-. 
\end{equation}

Let $\C(\bbX_P)$ denote the space of $K\xt K$-finite $C^\infty$ functions on $\bbX_P$. 

Let $\C_b(\bbX_P) \subset \C(\bbX_P)$ denote the subspace of functions with bounded support. Proposition~\ref{prop:Xbound} implies that $\C_b(\bbX_P)$ is the set of functions $\xi \in \C(\bbX_P)$ such that $\ord_M(\supp \xi)$ is bounded above.

\subsubsection{}
We say that a generalized function $\xi \in S^*(\bbX_P)$ has \emph{essentially bounded support} if the convolution of $\xi$ with 
any element of $C^\infty_c(G) \ot C^\infty_c(G)$ has bounded support. 
Let $S^*_b(\bbX_P)^G$ denote the space of generalized function with essentially bounded support that are invariant under the diagonal $G$-action on $\bbX_P$.

\subsection{The spaces $\C_{P,\pm}$} \label{ss:CPv}

Let $\Lambda_{G,P}^\bbQ = \Lambda_{M/[M,M]}^\bbQ$. This vector space is the quotient of $\Lambda^\bbQ$ by the subspace spanned by the coroots of $M$. 
For $\lambda \in \Lambda^\bbQ$, 
let $[\lambda]_P$ denote the projection of $\lambda$ to $\Lambda^\bbQ_{G,P}$.
We define the map 
\[ \deg_P : G/U \to \Lambda^\bbQ_{G,P} \]
by $\deg_P(x) = [\ord_M(x)]_P$. 

Let $\Lambda^{\pos,\bbQ}_{G,P}$ denote the image of $\Lambda^{\pos,\bbQ}_G$ (equivalently $\Lambda^{\pos,\bbQ}_U$) under the projection $\Lambda^\bbQ \to \Lambda^\bbQ_{G,P}$. 

\subsubsection{} Let $\C_P$ denote the space of $K$-finite $C^\infty$ 
functions on $G/U$. 
Let $\C_{P,c} \subset \C_P$ stand for the subspace of compactly supported functions.

Let $\C_{P,+} \subset \C_P$ denote the set of all functions $\varphi \in \C_P$ 
such that $\deg_P(\supp \varphi)$ is contained in $S_0 + \Lambda^{\pos,\bbQ}_{G,P}$ for some finite subset $S_0 \subset \Lambda^\bbQ_{G,P}$.
Similarly, 
let $\C_{P,-} \subset \C_P$ denote the set of all $\varphi \in \C_P$ such 
that $-\deg_P(\supp \varphi)$ is contained in $S_0 + \Lambda^{\pos,\bbQ}_{G,P}$ for some finite set $S_0$. 

One similarly defines the spaces $\C_{P^-,\pm} \subset \C_{P^-}$. 
We emphasize that $\C_{P^-,+}$ is defined with respect to the cone $-\Lambda^{\pos,\bbQ}_{G,P}$. So $\C_{P^-,\pm}$ is the
set of all $\varphi\in \C_{P^-}$ such that $\mp \deg_{P^-}(\supp\varphi)$
is contained in $S_0 + \Lambda^{\pos,\bbQ}_{G,P}$ for some finite set $S_0$.

\begin{lem} \label{lem:Txi}
Let $\xi \in S^*_b(\bbX_P)^G$ be a generalized function with essentially bounded support. 
Then formula \eqref{e:T_xi} defines a map $T_\xi : \C_{P,-} \to \C_{P^-,+}$. 

%(ii) If $S_0\subset \Lambda$ is a finite subset such that $\ord_M(\supp\xi) \subset \{ \theta - \mu \mid \theta \in S_0,\, \mu \in \Lambda^{\pos,\bbQ}_G \}$, then for any $\varphi \in \C_{P,-,v}$, 
%\[ \ord_M(\supp T_\xi(\varphi)) \subset \{ \lambda + \theta- \mu \mid \lambda\in\ord_M(\supp \varphi),\, \theta \in S_0,\,   \mu \in \Lambda^{\pos,\bbQ}_G \}. \]
\end{lem}
\begin{proof}
Let $\varphi \in \C_{P,-}$. Since $\varphi$ is $K$-finite, there exists a compact open subgroup $K' \subset K$ such that $\varphi$ is $K'$-invariant. 
Let $\delta_{K'} \in C^\infty_c(G)$ equal $\frac 1{\on{mes}(K')}$ times
the characteristic function of $K'$. Then 
\[ \xi' :=(\delta_{K'} \ot 1) * \xi = (1\ot \delta_{K'})*\xi \in C^\infty_b(\bbX_P) \] 
is a smooth function with bounded support, and it suffices to show that $T_{\xi'}(\varphi) = T_\xi(\varphi)$ is well-defined 
and belongs to $\C_{P^-,+}$. 

Fix $g_2 \in G/U^-$. 
For any $g_1 \in G/U$, Lemma~\ref{lem:ordineq} gives the inequality 
$\ord_M(g_1) \le_M \ord_M(g_2) - w_0^M \ord_M(g_1,g_2)$. 
Then Corollary~\ref{cor:C-U} implies that there is a finite subset
$S_0 \subset \Lambda$ such that if $\xi'(g_1,g_2)\ne 0$, then 
$\ord_M(g_1) \subset S_0 + w_0^M \Lambda^{\pos,\bbQ}_G$. 
From this combinatorial description, we deduce that the 
function sending $g_1\in G/U$ to $\varphi(g_1)\xi'(g_1,g_2)$ is compactly supported. Therefore $T_{\xi'}(\varphi)$ is well-defined. 

Moreover, if $T_{\xi'}(\varphi)(g_2) \ne 0$, then there must exist 
$g_1\in G/U$ such that $g_1 \in \supp \varphi$ and $(g_1,g_2) \in \supp \xi' \subset\bbX_P$. 
Observe that $\deg_{P^-}(g_2) = \deg_P(g_1) + [\ord_M(g_1,g_2)]_P$ in $\Lambda^\bbQ_{G,P}$. 
Since $\xi'$ has bounded support, 
$[\ord_M(\supp \xi')]_P$ is contained in $-\Lambda^{\pos,\bbQ}_{G,P} + S_1$ for 
a finite set $S_1$. By definition of $\C_{P,-}$, 
we deduce that $\deg_{P^-}(g_2)$ must lie in $-\Lambda^{\pos,\bbQ}_{G,P} + S_2$
for some finite set $S_2$. Thus $T_{\xi'}(\varphi) \in \C_{P^-,+}$.
\end{proof}

\subsection{Intertwining operator} \label{ss:R_Pv}
Define the intertwining operator $R_P : C^\infty_c(G/U^-) \to C^\infty(G/U)$ 
by the formula
\begin{equation} \label{e:IG}
R_P(\varphi)(g) = \int_{U} \varphi(gu) du, \quad\quad g \in G.
\end{equation}

Let $\bbX_P^-$ denote the space $(G/U^- \xt G/U)/M$. 
Any generalized function $\eta \in S^*(\bbX_P^-)$ defines a map 
$T_\eta : C^\infty_c(G/U^-) \to C^\infty(G/U)$ as in formula \eqref{e:T_xi}. 
Let $\eta_P \in S^*(\bbX_P^-)$ denote the generalized function 
such that $R_P = T_{\eta_P}$, i.e., 
\begin{equation} \label{e:R=T}
    \int_U \varphi(g_2 u)du = \int_{G/U^-} \varphi(g_1) \eta_P(g_1,g_2) dg_1, \quad\quad \varphi \in C^\infty_c(G/U^-),\, g_2\in G.
\end{equation} 

Define the projection 
\[ \ord_M : \bbX_P^- \to (K \xt K)\bs \bbX^-_P = \Lambda^+_M \]
by sending $(m_1,m_2) \mapsto \ord_M(m_1^{-1} m_2)$ for $m_1,m_2\in M$.

\begin{lem} \label{lem:suppeta}
The subset $\ord_M(\supp \eta_P)$ is contained in $(-\Lambda^{\pos,\bbQ}_U) \cap \Lambda^+_M$. In particular, $\ord_M(\supp \eta_P)$ is bounded above.
\end{lem}
\begin{proof}
Suppose that $(k_1 a, k_2) \in \supp \eta_P$ for $k_1, k_2\in K$ and $a \in T$ with $\ord_T(a) \in \Lambda^+_M$. 
Then by definition of $\eta_P$, there exists $u \in U$ such that
$u \in k a U^-$ for $k = k_2^{-1} k_1$. 

%The proof of Lemma~\ref{lem:suppm(u)} will rely on Lemma~\ref{lem:Kscalar} below.
Fix a dominant weight $\check\lambda \in \check\Lambda^+_G$.
Let $\Delta(\check\lambda)$ denote the Weyl $G$-module with highest weight 
$\check\lambda$. 
This $F_v$-vector space is the extension of scalars
of a free $\fo_v$-module, and the latter determines a $K$-invariant norm $\abs{}_v$ on $\Delta(\check\lambda)$. 
%If $F_v$ is Archimedean, there exists a positive definite $K$-invariant Euclidean/Hermitian
%form $(\,,)$ on $\Delta(\check\lambda)$, which also induces a $K$-invariant norm $\abs{}_v$ on $\Delta(\check\lambda)$. 
We give $\Delta(\check\lambda)^*$ the dual norm.

Let $\phi \in \Delta(\check\lambda)^*$ be a norm $1$ weight vector
of weight $-w_0^M \check\lambda$. 
Since $-w_0^M \check\lambda \ge_M -\check\lambda$ and all weights of $\Delta(\check\lambda)^*$
are $\ge_G -\check\lambda$, we observe that $\phi$ is $U^-$-invariant.
By orthogonality of weight spaces, there exists a weight vector $\xi \in \Delta(\check\lambda)$ of weight $w_0^M \check\lambda$ and norm $1$ such that 
$\brac{\phi, \xi} = 1$.
Note that $\xi$ is automatically $U$-invariant. We have the inequality 
\[ 0=\log_{q_v} \abs{\brac{u \cdot \phi, \xi}}_v = \brac{w_0^M \check\lambda, \ord_T(a)} +
\log_{q_v} \abs{\brac{k \cdot \phi, \xi}}_v \le \brac{w_0^M \check\lambda, \ord_T(a)}. \]
Since $w \check\lambda \ge_M w_0^M \check\lambda$ for any $w\in W_M$
and $\ord_T(a) \in \Lambda^+_M$, we conclude that
$\ord_T(a) \in w \Lambda^{\pos,\bbQ}_G$ for all $w\in W_M$. 
Lemma~\ref{lem:posU} implies that $\ord_T(a) \in \Lambda^{\pos,\bbQ}_{U}$.
Since $\ord_M(k_1a, k_2) = -w_0^M \ord_T(a)$, we are done.
\end{proof}

The following result is \cite[Corollary 7.4, Proposition 7.5(a)]{BK}.
We give a proof using our combinatorial description of bounded subsets of $\bbX_P$.

\begin{prop} \label{prop:I:C}
Formula \eqref{e:IG} defines an operator $R_P : \C_{P^-,+} \to \C_{P,-}$.
\end{prop}

\begin{proof} The proof is exactly the same as the proof of Lemma~\ref{lem:Txi}. We repeat the argument for completeness: 
Let $\varphi \in \C_{P^-,+}$. Fix $g_2 \in G/U$. 
For any $g_1 \in G/U^-$, Lemma~\ref{lem:ordineq} gives the inequality 
$\ord_M(g_1) \le_M \ord_M(g_2) - w_0^M \ord_M(g_1,g_2)$. 
Since $\ord_M(\supp \eta_P) \subset -\Lambda^{\pos,\bbQ}_U$, the inequality 
implies that if $\xi'(g_1,g_2)\ne 0$, then 
$\ord_M(g_1) = \ord_M(g_2) + \mu$ for some $\mu\in w_0^M \Lambda^{\pos,\bbQ}_G$. 
From this combinatorial description, we deduce that the 
function sending $g_1\in G/U^-$ to $\varphi(g_1)\eta_P(g_1,g_2)$ is compactly supported. Therefore \eqref{e:R=T} is well-defined.

Moreover, if $R_P(\varphi)(g_2) \ne 0$, then there must exist 
$g_1\in G/U^-$ such that $g_1 \in \supp \varphi$ and $(g_1,g_2) \in \supp \eta_P \subset\bbX^-_P$. 
Observe that $\deg_P(g_2) = \deg_{P^-}(g_1) + [\ord_M(g_1,g_2)]_P$ in $\Lambda^\bbQ_{G,P}$. 
Lemma~\ref{lem:suppeta} shows that 
$[\ord_M(\supp\eta_P)]_P$ is contained in $-\Lambda^{\pos,\bbQ}_{G,P}$. By definition of $\C_{P^-,+}$, 
we deduce that $\deg_P(g_2)$ must lie in $-\Lambda^{\pos,\bbQ}_{G,P} + S_0$
for some finite set $S_0$. Thus $R_P(\varphi) \in \C_{P,-}$.
\end{proof}

%Consider the composition
%\begin{equation}\label{eqn:m(u)}
%U \into G \to G/U^-, 
%\end{equation}
%where the first inclusion is the subgroup embedding and the second map is the natural projection. Note that \eqref{eqn:m(u)} is a closed embedding. 
%

\subsection{Local asymptotics map}\label{ss:asymp}
The work \cite{BK} gives a geometric proof of the second adjointness
between parabolic induction and restriction (Jacquet) functors
by defining the Bernstein map $\mathsf{B} : C^\infty_c(\bbX_P) \to C^\infty_c(G)$. If $F_v$ has characteristic $0$, this map is 
the ``asymptotics'' map constructed in \cite{SV} in the more general 
setting of spherical varieties. 
The dual of $\mathsf{B}$ gives a map 
\[ \Asymp_{P} : S^*(G) \to S^*(\bbX_P), \]
where $S^*(G), S^*(\bbX_P)$ are the spaces of distributions 
on $G, \bbX_P$, respectively. 

Fix the Haar measures on $G, M, U$ so that 
$K, K_M, K \cap U$ have measure $1$. Using these measures, we 
identify distributions and generalized functions on $G, \bbX_P$.

\subsubsection{} The Bernstein map $\mathsf{B}$ is $G \xt G$-equivariant,
and hence so is $\Asymp_{P}$. Therefore $\Asymp_{P}$ preserves 
$K \xt K$-finiteness, and \cite[Proposition 7.1]{BK} shows that 
it restricts to a map 
\[ \Asymp_{P} : C^\infty_c(G) \to \C_b(\bbX_P). \]

\subsubsection{}
Let $\delta_g \in S^*(G)$ denote the delta (generalized) function 
at $g \in G$. Set 
\begin{equation} \label{e:xi_Pv}
    \xi_{P} := \Asymp_{P}(\delta_1) \in S^*(\bbX_P),  
\end{equation}
which we consider as a generalized function on $\bbX_P$. 
Let $f_1,f_2 \in C^\infty_c(G)$ and set $f_2^\vee(g) = f_2(g^{-1})$. 
Then $G \xt G$-equivariance of $\Asymp_{P,v}$ implies that 
\begin{equation}\label{e:xiAsymp}
    (f_1,f_2) * \xi_P = \Asymp_{P}(f_1 * f_2^\vee), 
\end{equation}
where $*$ denotes convolution with respect to the $G\xt G$-action on $\bbX_P$ (resp.~the usual convolution on $G$). In particular, $\xi_{P}$ has \emph{essentially bounded
support} in the sense that the convolution of $\xi_{P}$ with 
any element of $C^\infty_c(G)$ has bounded support. 

\smallskip

Note that $\xi_{P}$ depends on the choice of Haar measure on $G$.
%Asymp does not depend on Haar measure on G if the G \times G invariant measures on G and \bbX_P are compatible (see [\S 4.2]{SV}). 

\subsubsection{}
We have the following relationship between the asymptotics map and the intertwining operator:
\begin{thm}[{\cite[Theorem 7.6]{BK}}] \label{thm:IAsymp}
Let $\varphi \in C^\infty_c(G/U)$. We have an equality 
\[ \varphi = (R_P \circ T_{\xi_P} )(\varphi). \]
In particular, the integral defining $(R_P \circ T_{\xi_P})(\varphi)$ converges.
\end{thm}

The theorem is stated for the case $P=B$ in \cite{BK}, but one can check that the proof generalizes to the case of an arbitrary parabolic subgroup.

\subsection{Invertibility of $R_P$} \label{ss:R_Pv^-1}
We use Theorem~\ref{thm:IAsymp} to show that $R_{P}$ is an isomorphism in Proposition~\ref{prop:R_Pv^-1} below. 

\subsubsection{} 
Let $\wtilde\C_{P,-}$ denote the smooth part of the linear dual representation $\C_{P,-}^*$. 
Let \[ R_P^* : \wtilde\C_{P,-} \to \wtilde\C_{P^-,+}\] denote the linear 
dual of the operator $R_P$. 
Using the measure on $G/U$ determined by the fixed Haar measures on $G$ and $U$, we
identify $\wtilde \C_{P,-}$ with a subspace of $\C_P$ containing $\C_{P,c}$. 
Similarly, we consider 
$\wtilde \C_{P^-,+} \subset \C_{P^-}$. 

Let $R_{P^-} : \C_{P,+} \to \C_{P^-,-}$ denote the intertwining operator with respect to the opposite parabolic $P^-$ (i.e., we integrate over $U^-$ in formula \eqref{e:IG}). 

\begin{lem}\label{lem:Iadjoint}
We have an equality $R_P^* = R_{P^-} : \C_{P,c} \to \C_{P^-}$. 
\end{lem}
\begin{proof}
Let $\tilde\varphi, \varphi \in \C_{P,c}$.
Then
\begin{multline*}
    \brac{R_P^*(\tilde\varphi), \varphi} = 
\int_{G/U} \tilde\varphi(g) \int_{U} \varphi(gu) du dg
=  \int_G \tilde\varphi(g) \varphi(g) dg \\ 
= \int_{G/U^-} \int_{U^-} \tilde\varphi(g \bar u) \varphi(g) d\bar u dg
= \brac{R_{P^-}(\tilde\varphi),\varphi},
\end{multline*}
where all the integrals are finite.
This proves the lemma.
\end{proof}

\subsubsection{}
Let $\xi_{P} \in S^*(\bbX_P)$ be the generalized function defined in \eqref{e:xi_Pv}. Note that $\xi_{P} \in S^*_b(\bbX_P)^G$, so Lemma~\ref{lem:Txi} defines a map $T_{\xi_P} : \C_{P,-} \to \C_{P^-,+}$. 
Theorem~\ref{thm:IAsymp} has the following reformulation:

\begin{lem} \label{lem:IT=1}
We have the equality $R_P \circ T_{\xi_{P}} = \id$ on $\C_{P,-}$.
\end{lem}

\begin{proof}
Let $\varphi \in \C_{P,-}$. Since $\varphi$ is $K$-finite, there exists a compact open subgroup $K' \subset K$ such that $\varphi$ is $K'$-invariant.
The proof of Lemma~\ref{lem:Txi} shows that 
\[ \ord_M(\supp T_{\xi_{P}}(\varphi)) \subset S:=\{ \lambda +\theta - \mu \mid 
\lambda \in \ord_M(\supp \varphi),\, \theta \in S_0,\,  \mu \in \Lambda^{\pos,\bbQ}_G \},\] where $S_0 \subset \Lambda$ is a finite subset depending
only on $K'$ and not $\varphi$. 
The proof of Proposition~\ref{prop:I:C} shows that $\ord_M(\supp R_P(T_{\xi_P}(\varphi)))$ is also contained in $S$.
Therefore we deduce that it suffices to prove that $\varphi = R_P(T_{\xi_P}(\varphi))$ for compactly supported $\varphi \in \C_{P,c}$, which is Theorem~\ref{thm:IAsymp}. 
\end{proof}

\begin{prop}  \label{prop:R_Pv^-1}
The map $R_P : \C_{P^-,+} \to \C_{P,-}$ is an isomorphism. The inverse is given by the formula 
\begin{equation} \label{e:I^-1v}
    R_P^{-1}(\varphi)(g_2) = \int_{G/U} \varphi(g_1) \xi_P(g_1,g_2) dg_1, \quad\quad \varphi\in \C_{P,-},\, g_2 \in G/U^-
\end{equation}
\end{prop}

\begin{rem}
Our proof of Proposition~\ref{prop:R_Pv^-1} is different from the one in \cite[Proposition 7.5(b)]{BK}. This proof was suggested by V.~Drinfeld.
%The proof of Proposition~\ref{prop:IBK} given in \emph{loc.~cit.}~does not apply to the case
%when $P$ is an arbitrary parabolic. R.~Bezrukavnikov told the author an alternative proof that works for the general case, by private communication. 
%

We give a separate, self-contained proof of invertibility of 
$R_P^K : \C_{P^-,+}^K \to \C_{P,-}^K$ with explicit formulas 
in Corollary~\ref{cor:Iisom}.
\end{rem}

\begin{proof}
Lemma~\ref{lem:IT=1} implies that $R_{P}$ has a right inverse
given by \eqref{e:I^-1v}. 
It remains to show that $R_{P}$ has a left inverse. 
Apply Lemma~\ref{lem:IT=1} to the \emph{opposite parabolic} $P^-$ 
to get a map $T_{\xi_{P^-}} : \C_{P^-,-} \to \C_{P,+}$ such that 
$R_{P^-} \circ T_{\xi_{P^-}} = \id$ on $\C_{P^-,-}$.
Taking the dual operators gives an equality 
\begin{equation}\label{e:TI=1}
    T_{\xi_{P^-}}^* \circ R_{P^-}^* = \id 
\end{equation}
on $\C_{P^-,c} \subset \wtilde\C_{P^-,-}$. 
Let $\varphi \in \C_{P^-,c}$. 
Lemma~\ref{lem:Iadjoint} (applied to $P^-$) implies that $R_{P^-}^*(\varphi) = R_P(\varphi)$.
Define $\tilde \xi \in S^*_b(\bbX_P)^G$ by 
\[ \tilde\xi(g_1,g_2) = \xi_{P^-}(g_2,g_1). \] 
It follows formally from the definition of $T_{\xi_{P^-}}$ that 
$T_{\xi_{P^-}}^*( R_P(\varphi) ) = T_{\tilde \xi}( R_P(\varphi) )$,
where the map $T_{\tilde\xi} : \C_{P,-} \to \C_{P^-,+}$ is defined by Lemma~\ref{lem:Txi}. 
Therefore \eqref{e:TI=1} implies that $T_{\tilde\xi} \circ R_P = \id$ on 
$\C_{P^-,c}$.
Then the same argument as in the proof of Lemma~\ref{lem:IT=1} shows that $T_{\tilde \xi} \circ R_P = \id$ on $\C_{P^-,+}$, so we conclude that $R_P$ has a left inverse.
\end{proof}

\section{Formulas on $K$-invariants}\label{s:Kinvariants}

Let $F_v$ be an arbitrary non-Archimedean local field. 
We use the same notation and conventions as in \S\ref{s:local} (e.g., $G = G(F_v),\, K=K_v$, etc.). 

\smallskip

Restricting to $K$-invariants, we see that the intertwining operator is essentially convolution with a measure $\mu_{M}$ on $M$. 
We compute the Satake transform of $\mu_{M}$ using the non-Archimedean Gindikin--Karpelevich formula. 
We give a formula for $\Asymp_{P}(\delta_K)$, where $\delta_K$ is the characteristic function of $K$, in terms of the convolution inverse of $\mu_{M}$.

\smallskip

Fix the Haar measures on $G,\, M,\, T,\, U,\, U^-$ so that
$K,\, K_M,\, K_T,\, K \cap U,\, K \cap U^-$ all have measure $1$.

\subsection{Intertwining operator on $K$-invariants} 
Let $K \subset G$ act on $G/U,\, G/U^-$ on the left. Recall that $K \bs G/ U = K\bs G/U^- = K_M \bs M$.

\subsubsection{The measure $\mu_{M}$}  \label{sss:mu}
Let $\bar\mu$ denote the direct image of the Haar measure on $U$ 
under the map 
\begin{equation} \label{e:m(u)}
U \into G \to K\bs G/U^- = K_M \bs M,
\end{equation}
where $G = K \cdot M \cdot U^-$ by the Iwasawa decomposition. 
In other words,
$\bar\mu(\Omega)$ is the measure of $U \cap (K\cdot \Omega \cdot U^-) \subset U$, 
where $\Omega \subset K_M \bs M$.
Since \eqref{e:m(u)} is equivariant with respect to the action of $K_M$ by
conjugation, $\bar\mu$ is right $K_M$-invariant. 
Define $\mu_M$ to be the $K_M$-bi-invariant measure on $M$ whose pushforward to 
$K_M \bs M$ equals $\bar\mu$.

\subsubsection{A formula in terms of convolution}
Let $R_{P}^K$ denote the restriction of the intertwining operator \eqref{e:IG}
to $\C_{P^-,+}^K \to \C_{P,-}^K$. 
Then we have the formula
\begin{equation} \label{eqn:Iconv}
	R_{P}^K(\varphi)(m) = \delta_P(m)^{-1}\int_{U} \varphi(u m) du = 
	\delta_P(m)^{-1} \int_M \varphi(m_1 m) \mu_M(m_1) 
\end{equation}
where $\varphi \in \C_{P^-,+}^K,\, m \in M$. 

Comparing with \eqref{e:R=T}, we see that 
\begin{equation} \label{e:eta=mu}
((\delta_K \ot 1) * \eta_P)(m,1) = \mu_M(m), \quad\quad m\in M, 
\end{equation}
where we consider $\mu_M$ as a $K_M$-bi-invariant function using the fixed Haar measure on $M$.

\subsection{Satake isomorphism} 
We extend the classical Satake isomorphism to an isomorphism between certain larger algebras (defined below) that are convenient for our purposes. 

\smallskip

%Let $H_M$ denote the space of $K_M$-bi-invariant compactly supported
%measures on $M$, with algebra structure induced by convolution (i.e., $H_M$ is the spherical Hecke algebra). 
%

\subsubsection{The completed Hecke algebra}  \label{sss:completion}
Let $H_M^+$ denote the space of $K_M$-bi-invariant measures on $M$ (with values in $E$) whose support is contained in $\ord_M^{-1}(\Lambda^{\pos,\bbQ}_{U} \cap \Lambda^+_M)$, where $\Lambda^{\pos,\bbQ}_{U}$ is the rational cone defined in \S\ref{sss:LambdaU}. 

\begin{rem} 
Lemma~\ref{lem:suppeta} and \eqref{e:eta=mu} imply that $\mu_{M}$ belongs to $H^+_M$. 
\end{rem}
 
\begin{lem} \label{lem:Hconv}
(i) Suppose that $\Sigma$ is a submonoid of $\Lambda_M^+$ such that 
if $\lambda \in \Lambda^+_M$ and there exists $\nu \in \Sigma$
such that $\lambda \le_M \nu$, then $\lambda \in \Sigma$. 
Then the vector space of compactly supported
$K_M$-bi-invariant measures on $M$ whose support is contained in $\ord_M^{-1}(\Sigma)$ 
is closed under the convolution product, so the vector space becomes an algebra. 

(ii) If, in addition, $\Sigma$ generates a strongly convex cone and 
the intersection of $\Sigma$ with any shift of $\Lambda^\pos_M$ is finite, then 
convolution extends, by continuity, to the space of all
$K_M$-bi-invariant measures on $M$ whose support is contained in $\ord_M^{-1}(\Sigma)$.
Then this space is also an algebra.
\end{lem}
\begin{proof}
The lemma follows from the usual properties of the Hecke algebra.
\end{proof}

Lemma~\ref{lem:wthull} implies that $\Sigma = \Lambda^{\pos,\bbQ}_{U} \cap \Lambda^+_M$ satisfies the condition in Lemma~\ref{lem:Hconv}(i), and 
one observes from the definition that $\Lambda^+_{U}$ also satisfies condition (ii).
Therefore $H^+_M$ is an algebra with respect to convolution.

\subsubsection{}
Let $H^+_T$ denote the space of $K_T$-bi-invariant measures on $T$ whose support
is contained in $\ord_T^{-1}(\Lambda^{\pos,\bbQ}_{U} \cap \Lambda)$. 
Lemma~\ref{lem:Hconv}(ii) implies that $H^+_T$ is an algebra with respect to convolution. 
Observe that the Weyl group of $M$ acts on $H^+_T$. 
\medskip

Using our fixed Haar measures, we identify locally constant functions on $M$ (resp.~$T$) with locally constant measures on $M$ (resp.~$T$). 
We also fix the Haar measure on $U_B^- \cap M$ such that $U_B^- \cap K_M$ has measure $1$. 

\begin{lem} 
The usual Satake transform extends to an isomorphism 
$\CT : H^+_M \to (H^+_T)^{W_M}$ given by the formula
\begin{equation} \label{eqn:Sat+}
	\CT(h)(t) = \delta_{B \cap M}(t)^{-1/2} \int_{U_B^-\cap M} h(t \bar n)d\bar n,
\end{equation}
where $h\in H^+_M$ is considered as a function on $M$.
\end{lem}
Here $\CT$ stands for `constant term'.
Since the image of the Satake transform is $W_M$-invariant, \eqref{eqn:Sat+} does
not depend on the choice of Borel subgroup of $M$. 

\begin{proof} Let $1_{\ord^{-1}_M(\lambda)}$ denote the characteristic function 
	of $\ord_M^{-1}(\lambda) \subset M$ for $\lambda \in \Lambda^+_M$. 
	It is known (cf.~\cite[\S 4.2]{Cartier}) that $\CT (1_{\ord_M^{-1}(\lambda)})$
	does not vanish on $\ord_T^{-1}(\lambda'),\, \lambda'\in\Lambda_M^+$,
only if $\lambda' \le_M \lambda$. 
Thus we deduce that $\CT$ is well-defined and an isomorphism from the usual Satake isomorphism and the fact that $\Lambda^{\pos,\bbQ}_{U} \cap \Lambda^+_M$ satisfies Lemma~\ref{lem:Hconv}.
\end{proof}

\begin{rem} \label{rem:Alocal}
Note that the algebra $H^+_T$ is isomorphic to the completion of the
semigroup algebra of $\Lambda^{\pos,\bbQ}_{U} \cap \Lambda$ at the augmentation ideal. 
In particular, it is a local ring, and $\hat h \in H^+_T$ is a unit
if and only if $\hat h(1)\ne 0$. We deduce that $H^+_M$ and $(H^+_T)^{W_M}$ 
are also complete local rings.
\end{rem}

\subsection{Gindikin--Karpelevich formula}
In this subsection we rewrite the non-Archimedean Gindikin--Karpelevich formula\footnote{The Gindikin--Karpelevich formula for non-Archimedean local fields is due to Langlands \cite{L:Euler} and MacDonald \cite{MacDonald} independently, with a generalization by Casselman \cite{Casselman}.} as a formula for 
$\CT(\mu_M) \in (H_T^+)^{W_M}$. 

\subsubsection{}
Recall that we defined $\delta_P(m) = \abs{\det \Ad_{\Lie(U)}(m)}$ for $m\in M$. 
Let $2\check\rho_P : =2\check\rho - 2\check\rho_M$ be the sum of the positive roots in $G$ that are not roots of $M$.
For $m \in M$, we have $\delta_P(m)= q_v^{-\brac{2\check\rho_P, \ord_M(m)}}$.

\smallskip

For $\lambda \in \Lambda$, let $1_{\ord_T^{-1}(\lambda)}$ denote the characteristic
function of $\ord_T^{-1}(\lambda) \subset T$. 
Set 
\[ e^\lambda = q_v^{\brac{\check\rho_P, \lambda}}\cdot 1_{\ord_T^{-1}(\lambda)} \in H_T^+.\]

\begin{prop} We have 
\begin{equation}\label{eqn:Satmu}
	\CT(\mu_{M}) = \prod_{\alpha \in \Phi^+_G - \Phi_M} \frac{1-q_v^{-1} e^\alpha}{1-e^\alpha}, 
\end{equation}
where the r.h.s.~is considered as an element of $(H_T^+)^{W_M}$ by Remark~\ref{rem:Alocal}.
\end{prop}
\begin{proof} For the purpose of this proof, we may assume $E = \bbQ$ (since $\mu_{M}$ takes values in $\bbQ$).  
	Let $\check\lambda \in \check\Lambda \ot \bbC$ satisfy $\on{Re}\brac{\check\lambda,\alpha} > 0$
for every positive coroot $\alpha$ of $G$. 
Let $\chi_{\check\lambda}$ be the unramified character $T \to \bbC^\times$ 
sending $t \mapsto q_v^{-\brac{\check\lambda, \ord_T(t)}}$. 
Define the function $\phi_{K,\check\lambda}$ on $G$ by 
\[ \phi_{K,\check\lambda}(k \cdot t \cdot \bar n) = \chi_{\check\lambda}(t) \delta_B^{1/2}(t), 
\quad k \in K,\, t\in T,\, \bar n \in U(B^-) \]
where $G = K \cdot B^-$ by the Iwasawa decomposition.
The Gindikin--Karpelevich formula for non-Archimedean local fields \cite[p.~18]{L:Euler} implies that
\[ \int_M \phi_{K,\check\lambda}(m) \mu_{M,v}(m) = \int_{U} \phi_{K,\check\lambda}(u) du 
= \prod_{\alpha\in\Phi^+ -\Phi_M} \frac{1-q_v^{-1 - \brac{\check\lambda, \alpha}}}{1- q_v^{-\brac{\check\lambda, \alpha}}} \]
where the l.h.s.~converges absolutely. 
Integrating over $M = K_M \cdot (B^-\cap M)$ using the Iwasawa decomposition (cf.~\cite[Equations (5), (10)]{Cartier}), the l.h.s.~equals $\int_T \chi_{\check\lambda}(t) \delta_P^{1/2}(t) \CT(\mu_{M,v})(t) dt$. 
Note that for $\nu \in \Lambda$, we have
$\int_T \chi_{\check\lambda}(t) \delta_P^{1/2}(t) e^\nu(t)dt = q_v^{-\brac{\check\lambda, \nu}}$. 
Therefore equation \eqref{eqn:Satmu} holds after integrating against $\chi_{\check\lambda}$ for  
any $\check\lambda \in \check\Lambda \ot \bbC$ satisfying $\on{Re}\brac{\check\lambda,\alpha} > 0$
for all $\alpha \in \Phi^+_G$. This implies the equality \eqref{eqn:Satmu} 
of elements in $(H^+_T)^{W_M}$.
\end{proof}

\begin{cor} The measure $\mu_{M}$ is invertible in $H^+_M$. 
\end{cor}
\begin{proof}
	The Gindikin--Karpelevich formula \eqref{eqn:Satmu} implies that $\CT(\mu_{M})(1) = 1$,
	so $\CT(\mu_{M})$ is invertible by Remark~\ref{rem:Alocal}. 
	The extended Satake isomorphism \eqref{eqn:Sat+} then implies that $\mu_{M}$ is invertible.
\end{proof}

\subsubsection{} \label{def:nu}
Let $\nu_{M} \in H_M^+$ denote the convolution inverse of $\mu_{M}$. 
We consider it as a $K_M$-bi-invariant measure on $M$.

\begin{cor}\label{cor:Iisom}
	The operator $R_P^K : \C_{P^-,+}^K \to \C_{P,-}^K$ is an isomorphism.
The inverse is given by the formula 
\begin{equation} \label{eqn:nuconv}
	(R_P^K)^{-1}(\varphi)(m) = \int_M \delta_P(m_1 m)\varphi(m_1 m) \nu_{M}(m_1), 
\end{equation}
where $\varphi\in \C_{P,-}^K,\, m \in M$.
\end{cor}
\begin{proof}
Define $\xi \in \C_b(\bbX_P)^{K \xt K}$ by $\xi(m,1) = \nu_{M}(m)$ 
for $m \in M$. 
The fact that $\nu_{M}$ belongs to $H_M^+$ implies that
$\xi$ indeed has bounded support. 
Then the r.h.s.~of \eqref{eqn:nuconv} equals $T_\xi(\varphi)$, 
where $T_\xi : \C_{P,-} \to \C_{P^-,+}$ is defined in Lemma~\ref{lem:Txi}. 
In particular, the r.h.s.~of \eqref{eqn:nuconv} is well-defined. 
Note that by the Iwasawa decomposition, $\C_{P^-,+}^K$ and $\C_{P,-}^K$ identify with the same space of $K_M$-invariant functions on $M$. 
Equation \eqref{eqn:Iconv} expresses $R_{P}^K$ in terms of the convolution action of $\mu_{M}$ on $\C_{P^-,+}^K = \C_{P,-}^K$. 
This action is compatible with the convolution product of $H_M^+$. 
Thus we deduce from invertibility of $\mu_{M}$ that $R_P^K$ is an isomorphism
with inverse given by \eqref{eqn:nuconv}.
\end{proof}

\subsection{Langlands' reformulation} \label{sect:localL}
We will reformulate the Gindikin--Karpelevich formula in terms of Langlands' reinterpretation of the classical Satake isomorphism.

\subsubsection{}
Let $\check G$ (resp.~$\check M,\, \check T$) denote the Langlands dual group of $G$ (resp.~$M,\, T$) over $E$.
Let $\check \fu_P$ be the Lie algebra corresponding to the roots 
$\Phi^+_G - \Phi_M$ in $\check G$, so $\check \fu_P$ is a $\check M$-module by the
adjoint action $\Ad_{\check\fu_P}$. 

\smallskip

The map $1_{\ord_T^{-1}(\lambda)} \mapsto \lambda$ defines 
an isomorphism $C^\infty_c(T)^{K_T} \cong E[\check T]$, which is compatible
with the $W_M$-action. 
Recall that $E[\check T]^{W_M} = E[\check M]^{\check M}$,
where $\check M$ acts on itself by conjugation. 
Let $\bK(\Rep(\check M))$ denote the Grothendieck group of 
the abelian category of finite dimensional $\check M$-modules. 
Give $\bK(\Rep(\check M))$ the tensor product multiplication.
Then we have an algebra isomorphism by taking characters:
\[ \bK(\Rep(\check M)) \ot_\bbZ E \overset{\Ch}\longrightarrow E[\check M]^{\check M} : 
[V] \mapsto \tr(\sigma, V), \quad\quad \sigma \in \check M. \]

\subsubsection{}
Let $\Rep^+(\check M)$ denote the subcategory of 
$\check M$-modules with weights contained in $\Lambda^{\pos,\bbQ}_{U}$. 
Since $2\check\rho_P\in \check\Lambda$ is perpendicular to all coroots of $M$, 
we may consider it as a central cocharacter of $\check M$. 
We have a non-negative grading of the Grothendieck group $\bK(\Rep^+(\check M))$ by the eigenvalues of $2\check\rho_P$. 
Let $\bK^+(\Rep(\check M))$ be the completion
of $\bK(\Rep^+(\check M))$ with respect to the augmentation ideal of this grading.  
Then one sees 
that $\Ch^{-1} \circ \CT$ extends to an algebra isomorphism 
\[ \eS : H^+_M \to \bK^+(\Rep(\check M)) \hat\ot E, \]
where $\hat \ot$ is the completed tensor product.

\subsubsection{}  Let $V \in \Rep(\check M)$. 
Consider the expression $\sum_n t^n [\Sym^n V]$, which is a
formal series in $\bK(\Rep(\check M))\tbrac{t}$. Here $t$ is a formal parameter (that is unrelated to the torus). It is well-known that the inverse of this
series equals 
\[ \Lambda(t,V) := \sum_n (-t)^n [\wedge^n V]. \] 
If we consider coefficients in $E[t]$ rather than $E$, we have $\tr(\sigma,\Lambda(t,V)) = \det(\Id - \sigma \cdot t, V)$ for $\sigma \in \check M(E)$.

\subsubsection{} Suppose that the weights of $V$ are contained in 
$\Lambda^{\pos,\bbQ}_{U} - \{0\}$. Let $\tau \in E^\times$. Then 
the series
\[ \mathrm{S}(\tau, V) := \sum_n \tau^n [\Sym^n V] \]
is a well-defined element of the completed Grothendieck group 
$\bK^+(\Rep(\check M)) \hat\ot E$, and it is the inverse of
$\Lambda(\tau, V) \in \bK^+(\Rep(\check M)) \hat\ot E$.

\subsubsection{} \label{sect:u_i}

The central cocharacter $2\check\rho_P$ defines a non-negative $\check M$-module
grading of $\check \fu_P$ by the eigenspace decomposition. 
Let $\gr^i(\check \fu_P)$ denote the eigenspace of $\Ad_{\check\fu_P}(2\check \rho_P)$ with weight 
$2a_i$, where $a_i$ is a positive integer.
Then in the above language, equation \eqref{eqn:Satmu} and its multiplicative inverse have the reformulations 
\begin{equation}\label{eqn:SatLnu}
    \eS(\mu_{M,v}) = \prod_i \frac{\Lambda(q_v^{-1+a_i} , \gr^i(\check\fu_P))}{\Lambda(q_v^{a_i}, \gr^i (\check \fu_P))} , \quad\quad 
    \eS(\nu_{M,v}) = \prod_i \frac{\Lambda(q_v^{a_i}, \gr^i (\check \fu_P))}{\Lambda(q_v^{-1+a_i}, \gr^i (\check \fu_P))}.
\end{equation}
\noindent
The formula for $\eS(\mu_{M,v})$ essentially appears in \cite[p.~33]{L:Euler}. 

\smallskip

Using the equality $\Lambda(q_v^{-1+a_i}, \gr^i(\check\fu_P))^{-1} = \mathrm{S}(q_v^{-1+a_i}, \gr^i(\check \fu_P))$, we have the expansion 
\begin{equation} \label{eqn:Lprod}
	\frac{\Lambda(q_v^{a_i}, \gr^i(\check\fu_P))}{\Lambda(q_v^{-1+a_i}, \gr^i(\check\fu_P))}
= \left(\sum_n (-1)^n [\wedge^n \gr^i(\check\fu_P)] \cdot q_v^{a_i n}\right)
\left( \sum_n [\Sym^n \gr^i(\check\fu_P)] \cdot q_v^{-n+ a_i n} \right) 
\end{equation}
in $\bK^+(\Rep(\check M)) \hat\ot E$.

\subsection{Asymptotics on $K$-invariants}
Let $\delta_K \in C^\infty_c(G)$ denote the characteristic function of $K$.
Note that $\Asymp_{P}(\delta_K) = (\delta_K \ot 1) * \xi_{P} = (1\ot \delta_K) * \xi_{P}$ is $K \xt K$-invariant. 
Using \eqref{eqn:nuconv} and \eqref{e:I^-1v}, we deduce the formula
\begin{equation} \label{eqn:Asymp0}
	\Asymp_{P}(\delta_K)(m,1) = \nu_{M}(m). 
\end{equation}
When $P = B$ is a Borel subgroup and $F_v$ has characteristic $0$, \eqref{eqn:Asymp0} is proved in \cite[Theorem 6.8]{Sak} in the more general setting of spherical varieties. 

\begin{rem} \label{rem:AsympK}
	Note that $\nu_{M}(1)= \CT(\nu_{M})(1) = 1$ by the explicit formula \eqref{eqn:Satmu}.
Thus \eqref{eqn:Asymp0} implies that $\Asymp_{P}(\delta_K)$ takes constant value $1$
on the $K \xt K$ orbit of $(1,1) \in \bbX_P$. 

In the notation of Remark~\ref{rem:barXO} we also see that $\Asymp_{P}(\delta_K)$
has support contained in $\mathsf{\wbar X_P}(\fo_v)$ since $\nu_{M}\in H_M^+$.
\end{rem}

\section{The bilinear form \texorpdfstring{$\eB$}{B}}  \label{s:B}

We work over the function field $F$ with adele ring $\bbA$.
Let $X$ be the corresponding geometrically connected smooth projective curve over $\bbF_q$. In this section, we define the bilinear form $\eB$ and prove Theorem~\ref{thm:B=b}. 

In our notation, we will add a subscript $v$ when referring to the objects or spaces defined in \S\ref{s:local} over $F_v$ (e.g., $\C_P$ becomes $\C_{P,v}$, $\Asymp_P$ becomes $\Asymp_{P,v}$).

\subsection{Definition of $\eB$}  \label{ss:defB}
Fix a Haar measure on $G(\bbA)$. 
For $f_1,f_2\in \eA_c$, set
\begin{equation} \label{e:defB}
    \eB(f_1,f_2) := \sum_P (-1)^{\dim Z(M)} \cdot \eB_P(f_1,f_2) 
\end{equation}
where the sum ranges over standard parabolic subgroups $P\subset G$ with Levi subgroup $M$,
and $\eB_P$ is a $G(\bbA)$-invariant bilinear form defined in 
\S\ref{sss:defB_P} below. 
The form $\eB$ is $G(\bbA)$-invariant since each $\eB_P$ is. 
It will also be evident that $\eB$ is symmetric.
Let us note that $\eB_P$ and $\eB$ slightly depend on the choice of a Haar measure on $G(\bbA)$.

% this part holds for any global field
%Let $\mathsf{X_P}$ denote the $\fo_F$-model of $\bbX_P$ determined by
%the pinning of $(G,T,B)$, and set $\mathsf{\wbar X_P} := \spec \Gamma(\mathsf{X_P}, \eO_{\mathsf{X_P}})$. Then $\mathsf{\wbar X_P} \xt_{\spec \fo_F} \spec F = \spec F[\bbX_P] =: \wbar \bbX_P$
%is the affine closure of $\bbX_P$. 

\subsection{Definition of $\eB_P$}  \label{ss:B_P}
Fix a standard parabolic subgroup $P$. 
Define the boundary degeneration $\bbX_P = (G \xt G)/(P \xt_M P^-)$ as in \S\ref{sss:X_P}, 
where $\bbX_P$ is now a strongly quasi-affine variety over $\bbF_q$. 
Let $\wbar \bbX_P$ denote the affine closure. 

The topological space $\bbX_P(\bbA)$ is isomorphic to the restricted product
of $\bbX_P(F_v)$ with respect to the compact subspaces $\bbX_P(\fo_v)$. 
The topological space $\wbar \bbX_P(\bbA)$ is isomorphic to the restricted product
of $\wbar\bbX_P(F_v)$ with respect to $\wbar \bbX_P(\fo_v)$. 
Recall that the topology on $\bbX_P(\bbA)$ is \emph{not} the subspace topology induced
from $\wbar \bbX_P(\bbA)$. 

We say that a function on $\bbX_P(\bbA)$ has bounded support
if the support is relatively compact in $\wbar \bbX_P(\bbA)$.
Let $\C_b(\bbX_P(\bbA))$ denote the space of $K\xt K$-finite $C^\infty$ 
functions on $\bbX_P(\bbA)$ with bounded support.

Note that the action of $P^- \xt P$ on $1 \in G$ and $(1,1) \in \bbX_P$ 
have the same stabilizer equal to the diagonal embedding of $M$. 
Fix the measure on $\bbX_P(\bbA)$ to be the unique $G(\bbA) \xt G(\bbA)$-invariant measure such that on the $P^-(\bbA) \xt P(\bbA)$-orbit of $(1,1)$, 
it coincides with the restriction of the chosen Haar measure on $G(\bbA)$ to $P^-(\bbA)\cdot P(\bbA)$.

\subsubsection{} \label{sss:Asymp}
Define $\Asymp_P : C^\infty_c(G(\bbA)) \to \C_b(\bbX_P(\bbA))$ by  
\begin{equation} \label{eqn:asymp}
	\Asymp_P(\ot_v f_v) = \ot_v \Asymp_{P,v}(f_v)
\end{equation}
where $f_v \in C^\infty_c(G(F_v))$ and 
$f_v = \delta_{K_v}$ is the characteristic function of $K_v$ for almost all $v$.
Observe that $\Asymp_P$ is well-defined since $\Asymp_{P,v}(\delta_{K_v})$ equals $1$ on $\bbX_P(\fo_v)$ by Remark~\ref{rem:AsympK}.
The product $\ot_v \Asymp_{P,v}(f_v)$ has bounded support in $\bbX_P(\bbA)$ 
because the support of $\Asymp_{P,v}(\delta_{K_v})$ is contained in 
$\wbar \bbX_P(\fo_v)$ by Remark~\ref{rem:AsympK}.

\subsubsection{} Define the generalized function $\xi_P \in S^*(\bbX_P(\bbA))$ by
\begin{equation} \label{e:xi_P} 
\xi_P = \ot_v \xi_{P,v} 
\end{equation}
where $\xi_{P,v} \in S^*(\bbX_P(F_v))$ is defined by \eqref{e:xi_Pv}.
Equation \eqref{e:xi_P} is well-defined because 
any element of $C^\infty_c(\bbX_P(\bbA))$ is $K_v$-invariant for almost all $v$,
and $\delta_{K_v} * \xi_{P,v} = \Asymp_{P,v}(\delta_{K_v})$ equals $1$ on 
$\bbX_P(\fo_v)$ by Remark~\ref{rem:AsympK}.
From \eqref{e:xiAsymp} we also deduce that $\xi_P$ has essentially
bounded support, i.e., for any $\tilde f \in C^\infty_c(G(\bbA))$, the convolution
$(\tilde f \ot 1) * \xi_P = (1 \ot \tilde f^\vee) * \xi_P = \Asymp_P(\tilde f)$
has bounded support, where $\tilde f^\vee(g) := \tilde f(g^{-1})$.

\subsubsection{} 
Define a bilinear form $\tilde\eB_P : C^\infty_c(G(\bbA)) \ot C^\infty_c(G(\bbA)) \to E$ by the formula
\begin{equation}\label{e:tildeB} 
    \tilde\eB_P(\tilde f_1, \tilde f_2) := \sum_{x\in \bbX_P(F)} \Asymp_P(\tilde f_1^\vee * \tilde f_2)(x), \quad\quad \tilde f_1,\tilde f_2 \in C^\infty_c(G(\bbA)), 
\end{equation}
where $\tilde f_1^\vee(g):= \tilde f_1(g^{-1})$, and $*$ denotes convolution over $G(\bbA)$. 
The sum is finite because $\Asymp_P(\tilde f_1^\vee * \tilde f_2)$ has bounded support, and the intersection of 
the discrete subset $\bbX_P(F) \subset \wbar \bbX_P(\bbA)$ with a bounded subset of $\bbX_P(\bbA)$ is finite.

Using \eqref{e:xiAsymp}, one can also write \eqref{e:tildeB} as 
\begin{equation}\label{e:tildeBxi} 
\tilde \eB_P(\tilde f_1, \tilde f_2) = \sum_{x\in \bbX_P(F)} 
\int_{(G \xt G)(\bbA)} \tilde f_1(g_1) \tilde f_2(g_2) \xi_P((g_1,g_2)x) dg_1 dg_2. 
\end{equation}

For $g \in G(\bbA)$, let $\delta_g$ denote the delta (generalized) function at $g$. Observe that 
\begin{equation} \label{e:Bleft}
    \tilde \eB_P(\delta_g * \tilde f_1, \delta_g * \tilde f_2) = 
\tilde \eB_P(\tilde f_1, \tilde f_2), \quad\quad g\in G(\bbA).
\end{equation}
By $(G \xt G)(\bbA)$-equivariance of $\Asymp_P$, we have
\begin{equation} \label{e:Bright}
    \tilde \eB_P(\tilde f_1 * \delta_{g_1}, \tilde f_2 * \delta_{g_2}) = 
\tilde \eB_P(\tilde f_1, \tilde f_2) , \quad\quad g_1,g_2 \in G(F).
\end{equation}

\subsubsection{} \label{sss:defB_P}
We define the bilinear form $\eB_P : \eA_c \ot \eA_c \to E$ as follows.
For $f_1,f_2 \in \eA_c$, there exist $\tilde f_1, \tilde f_2 \in C^\infty_c(G(\bbA))$ whose direct images are $f_1,f_2$. 
Set
\begin{equation} \label{eqn:B_P}
	\eB_P(f_1,f_2) = \tilde \eB_P(\tilde f_1, \tilde f_2),
\end{equation}
which does not depend on the choices of $\tilde f_1, \tilde f_2$ by \eqref{e:Bright}. 
The form $\eB_P$ is $G(\bbA)$-invariant by \eqref{e:Bleft}.
Formula \eqref{eqn:B_P} was suggested by Y. Sakellaridis in a private communication.

\subsection{Restriction of $\eB_P$ to $\eA_c^K$}
Fix the Haar measure on $G(\bbA)$ so that $K$ has measure $1$. 
Let $\tilde f_1, \tilde f_2 \in C^\infty_c(G(\bbA))^K$ be left $K$-invariant functions.
Let $\delta_K = \ot \delta_{K_v}$ denote the characteristic function of $K$ on $G(\bbA)$. 
Note that averaging $\delta_K * \delta_1 = \delta_1 * \delta_K = \delta_K$. 
Let $f_1, f_2 \in \eA_c^K$ denote the direct images of $\tilde f_1,\tilde f_2$. 
We deduce from \eqref{e:tildeBxi} that 
\begin{equation} \label{e:BbP}
	\eB_P(f_1,f_2) = \int_{(G\xt G)(\bbA)/ (G \xt G)(F)} f_1(g_1) f_2(g_2) b_P(g_1,g_2) dg_1 dg_2,
\end{equation}
where $b_P(g_1,g_2) = \sum_{\bbX_P(F)} \Asymp_P(\delta_K)((g_1,g_2)x)$.

\subsubsection{}
Observe that $b_P$ is obtained from $\Asymp_P(\delta_K) \in \C_b(\bbX_P(\bbA))^{K \xt K}$ by pull-push along 
the diagram
\begin{equation} \label{e:pullpushb}
	(G \xt G)(\bbA)/ (G \xt G)(F) \leftarrow 
(G \xt G)(\bbA) \xt^{(G \xt G)(F)} \bbX_P(F) \to \bbX_P(\bbA).
\end{equation}

\subsection{Geometric interpretation}  \label{ss:geom}
As explained in \cite[\S 1.2.3, Remark 1.2.17]{GL:Tamagawa}, we can 
identify\footnote{The identification relies on the assumptions that any $G$-bundle $\eF_G$ on $X$ is trivial when restricted to $\spec F$ and $\spec \fo_v$ for each place $v$. We know the restriction of $\eF_G$ to $\spec \fo_v$ is trivial by smoothness of $G$ and Lang's theorem (any $G$-bundle over a finite field is trivial). The generic triviality of $\eF_G|_{\spec F}$ follows from the Hasse principle for split reductive groups over a function field, which is proved by \cite{Harder}.} the double cosets 
$K \bs G(\bbA) / G(F)$ with $\abs{\Bun_G(\bbF_q)}$, the isomorphism classes of $G$-bundles on $X$. 
Let us give a geometric interpretation of $b_P$ as a function 
on $(\Bun_G \xt \Bun_G)(\bbF_q)$.

\subsubsection{}
Let $\eF_G^1, \eF_G^2 \in \Bun_G(\bbF_q)$ be $G$-bundles.
Fixing trivializations of $\eF_G^i \xt_X \spec(\fo_v)$ for all places $v$, we get lifts of $\eF_G^i \in K \bs G(\bbA)/G(F)$ to 
$g_i \in G(\bbA)/G(F)$ for $i=1,2$. The pre-image of $(g_1,g_2)$ 
in $(G \xt G)(\bbA) \xt^{(G \xt G)(F)} \bbX_P(F)$ under the left arrow
of \eqref{e:pullpushb} is in bijection with the set of rational
sections of the morphism 
\[ (\bbX_P)_{\eF_G^1,\eF_G^2} := (\eF_G^1 \xt_X \eF_G^2) \xt^{G \xt G} \bbX_P \to X. \]
Given a rational section $\beta \in (\bbX_P)_{\eF_G^1,\eF_G^2}(F)$, we can restrict
to $(\bbX_P)_{\eF_G^1,\eF_G^2}(F_v)$ for any place $v$, which is isomorphic to $\bbX_P(F_v)$ 
by the trivializations of $\eF_G^i \xt_X \spec(\fo_v),\, i=1,2$. This describes the 
right arrow in \eqref{e:pullpushb}. 

\subsubsection{}
Let $\ds (\wbar \bbX_P)_{\eF_G^1,\eF_G^2} := \wbar \bbX_P \xt^{G \xt G} (\eF_G^1 \xt_X \eF_G^2)$. 
We have an isomorphism $(\wbar \bbX_P)_{\eF_G^1,\eF_G^2}(\fo_v) \cong 
\wbar \bbX_P(\fo_v)$ compatible with the aforementioned identification
$(\bbX_P)_{\eF_G^1,\eF_G^2}(F_v) \cong \bbX_P(F_v)$. 

Remark~\ref{rem:AsympK} implies that the support of $\Asymp_P(\delta_K)$ is contained in 
$\bbX_P(\bbA) \cap \wbar \bbX_P(\fo_\bbA)$.
Therefore $\Asymp_P(\delta_K)$ does not vanish at the image of $\beta$ in $\bbX_P(\bbA)$
only if $\beta$ extends to a \emph{regular} section $X \to (\wbar \bbX_P)_{\eF_G^1,\eF_G^2}$. Such an extension is unique since $\wbar\bbX_P$ is separated. 
Thus 
\begin{equation} \label{eqn:bXP}
	b_P(\eF_G^1,\eF_G^2) = \sum_{\beta} \prod_v \Asymp_{P,v}(\delta_{K_v})(\beta_v) 
\end{equation}
where the sum is over sections $\beta : X \to (\wbar \bbX_P)_{\eF^1_G, \eF^2_G}$ that generically land in the non-degenerate locus 
$(\bbX_P)_{\eF_G^1,\eF_G^2}$, and $\beta_v \in \bbX_P(F_v)$ is the image of $\beta$ under the right arrow in \eqref{e:pullpushb}.
The $K_v \xt K_v$-orbit of $\beta_v$ does not depend on the choice of trivializations. 

Note that $\Asymp_{P,v}(\delta_{K_v})(\beta_v) = 1$ if $\beta_v \in \bbX_P(\fo_v)$. Thus
the product is only over those places $v$ that $\beta$ sends to the degenerate locus 
$(\wbar \bbX_P)_{\eF_G^1,\eF_G^2} - (\bbX_P)_{\eF_G^1,\eF_G^2}$. 

\begin{rem}
The product $\prod_v \Asymp_{P,v}(\delta_{K_v})(\beta_v)$ 
 is a $K$-invariant function on $\bbX_P(\bbA)\cap \wbar\bbX_P(\fo_\bbA)$. 
Its value does not depend on the choice of trivializations of $\eF^i_G \xt_X \spec(\fo_v)$, so we may also consider it as a function $\Asymp_P(\delta_K)(\beta)$ 
of $\beta$.
\end{rem}

\begin{rem}
Theorem~\ref{thm:i^*j_*} interprets $\Asymp_P(\delta_K)(\beta)$
as the trace of the geometric Frobenius acting on the $*$-stalks of an $\ell$-adic sheaf. 
\end{rem}

\begin{proof}[Proof of {Theorem~\ref{thm:B=b}}]
Let $\eF_G^1, \eF_G^2 \in \Bun_G(\bbF_q)$. 
Using the geometric interpretation \eqref{eqn:bXP} and Theorem~\ref{thm:b}, 
we get the equality
\[ b(\eF^1_G, \eF^2_G) = \sum_P (-1)^{\dim Z(M)} b_P(\eF^1_G, \eF^2_G), \]
where the sum ranges over standard parabolic subgroups. 
The theorem now follows from the definition of $\eB$ and the formula \eqref{e:BbP}.
\end{proof}

\section{Global intertwining operators} \label{s:global}

Let $P$ denote a standard parabolic subgroup.
We define the subspaces $\C_{P,\pm}$ of functions on $G(\bbA)/M(F)U(\bbA)$ and recall the definition of the 
constant term operator. We show that the product of the local intertwining operators induces an operator $R_P : \C_{P^-,+} \to \C_{P,-}$, and we prove 
that $R_P$ is invertible (Proposition~\ref{prop:I^-1}). 
We prove Theorem~\ref{thm:global} at the end of the section.

We continue to add a subscript $v$ to the notation of \S\ref{s:local} when appropriate.

\subsection{The spaces $\eC_P,\, \eC_{P,\pm}$} \label{ss:C_P}

\subsubsection{} Let $\C_P$ denote the space of $K$-finite $C^\infty$ functions on $G(\bbA)/M(F)U(\bbA)$. 
Let $\C_{P,c} \subset \C_P$ stand for the subspace of compactly supported functions. 

\medskip

As in \S\ref{ss:geom}, the quotient $K_M \bs M(\bbA) / M(F)$
identifies with $\abs{\Bun_M(\bbF_q)}$, the set of isomorphism classes of $M$-bundles on $X$. 
Recall that this identification uses the fact that any $M$-bundle on $X$ is generically trivial. Since we have an exact sequence 
\[ 0=H^1(\spec F, U) \to H^1(\spec F, P) \to H^1(\spec F, M), \] 
we deduce that any $P$-bundle on $X$ is also generically trivial.
This allows us to make the identification $K \bs G(\bbA) / P(F) = \abs{\Bun_P(\bbF_q)}$ by the decomposition 
$G(\bbA) = K \cdot P(\bbA)$. This space projects to 
$K\bs G(\bbA) / M(F) U(\bbA) = \abs{\Bun_M(\bbF_q)}$. 

\subsubsection{} 

Let $\Lambda_{G,P} = \pi_1(M)$ denote the quotient of $\Lambda$ by the 
subgroup generated by the coroots of $M$. 
It is well-known that there is a bijection $\deg_M : \pi_0(\Bun_M) \simeq \pi_1(M)$. Note that $\Lambda^\bbQ_{G,P} := \Lambda_{G,P} \ot \bbQ = \Lambda_{M/[M,M]}^\bbQ = \Lambda^\bbQ_{Z_0(M)}$. We call the composition 
\[ \Bun_M \to \pi_1(M) \to \Lambda^\bbQ_{G,P} \] 
the \emph{slope} map. 
We define the map 
\[ \deg_P^\bbQ : G(\bbA) / U(\bbA) \to \Lambda^\bbQ_{G,P} \] 
by setting $\deg_P^\bbQ(g)$ equal to the slope of the $M$-bundle corresponding to $g\in G(\bbA)$. 
Equivalently, if $g = (g_v),\, g_v \in G(F_v)$, then $\deg_P^\bbQ(g) = \sum_v \deg_{P,v}(g_v)$, where $\deg_{P,v}$ is as defined in \S\ref{ss:CPv}.

\subsubsection{} 

Let $\Lambda^{\pos,\bbQ}_{G,P}$ denote the image of $\Lambda^{\pos,\bbQ}_G$ under the projection $\Lambda^\bbQ \to \Lambda^\bbQ_{G,P}$.
We define the global spaces $\C_{P,\pm}$ analogously to the definitions of the local spaces 
$\C_{P,\pm,v}$ in \S\ref{ss:CPv}: 

\medskip

Let $\C_{P,+} \subset \C_P$ be the set 
of all functions $\varphi \in \C_P$ such that $\deg_P^\bbQ(\supp \varphi)$ is contained 
in $S_0 + \Lambda^{\pos,\bbQ}_{G,P}$ for some finite subset $S_0 \subset \Lambda^\bbQ_{G,P}$. 
Similarly, let $\C_{P,-} \subset \C_P$ denote the set of all $\varphi\in \C_P$ such that
$-\deg^\bbQ_P(\supp\varphi)$ is contained in $S_0 + \Lambda^{\pos,\bbQ}_{G,P}$ for some
finite set $S_0$.

%\begin{rem} \label{rem:C+}
%Using Remark~\ref{rem:discreteHN} and Corollary~\ref{cor:C-U}, we see that 
%$\varphi \in \eC_{P,+}$ if and only if there exists a finite subset $S_0 \subset \Lambda^\bbQ$ such that 
%$\HNSupp(\varphi)$ is contained in 
%the set $\{ \theta + \mu \mid \theta \in S_0,\, \mu \in \Lambda^{\pos,\bbQ}_{U} \}$.
%\end{rem} 

%Since $\Bun_M^{(\lambda)}$ is quasi-compact and $\Lambda^{\pos,\bbQ}_G$ is strictly convex, we deduce from Remark~\ref{rem:C+} that $\C_{P,+} \cap \C_{P,-} = \C_{P,c}$.

One similarly defines the spaces $\C_{P^-,\pm} \subset \C_{P^-}$. 
We emphasize that $\C_{P^-,+}$ is defined with respect to the cone $-\Lambda^{\pos,\bbQ}_{G,P}$. So $\C_{P^-,\pm}$ is the
space of all $\varphi\in \C_{P^-}$ such that $\mp \deg^\bbQ_{P^-}(\supp\varphi)$
is contained in $S_0 + \Lambda^{\pos,\bbQ}_{G,P}$ for some finite set $S_0$.

\begin{rem} 
In the case $P=G$, we have $\Lambda^\pos_{G,G} =0$. 
So we observe that $\C_{G,+} = \C_{G,-} \subset \eA$ is the set of functions $f\in \eA$ such that $\deg_G^\bbQ(\supp f)$ is finite.
\end{rem}

%\mycomment{Claim:}
%The pushforward \[ C^\infty_b(U_P(\bbA)\bs G(\bbA)) \to C^\infty_b(M(F) U_P(\bbA) \bs G(\bbA)) \]
%is well-defined and surjective. 
%
%
%
%
%\begin{plem} The map $\C_{P,c}^K \ot_{H^\ur_M} R^\ur_M \to \C_{P,\pm}^K$ 
%	is an isomorphism. 
%\end{plem}

\subsection{The Harder--Narasimhan--Shatz stratification} 
Before discussing the constant term operator, we need to recall 
some reduction theory, which we state in terms of the 
Harder--Narasimhan--Shatz stratification of $\Bun_M$. 
This stratification of $\Bun_M$ was defined in \cite{HN, Sh1, Sh2} in the case $M = \on{GL}(n)$. 
For any reductive $M$ it was defined in \cite{R1,R2,R3} and \cite{Beh, Beh1}. 
We also refer the reader to \cite{Schieder:HN}. 

\smallskip

Let $\Lambda^{+,\bbQ}_M$ denote the rational cone corresponding to the monoid $\Lambda^+_M$. 
For $\lambda \in \Lambda^{+,\bbQ}_M$, we follow the notation of \cite[Theorem 7.4.3]{DG:CG} and 
let $\Bun_M^{(\lambda)} \subset \Bun_M$ denote\footnote{In \emph{loc.~cit.} the stratification is defined over an algebraically closed field. To define the stratification over $\bbF_q$, we first base change to $\wbar\bbF_q$ and then note that the Harder--Narasimhan strata are defined over $\bbF_q$ by Galois invariance.}  the quasi-compact locally closed reduced substack 
of $M$-bundles with Harder--Narasimhan coweight $\lambda$. 
We have a map 
$\HN : \abs{\Bun_M(\bbF_q)} \to \Lambda^{+,\bbQ}_M$,
which sends an $M$-bundle to its unique Harder--Narasimhan coweight. 
We will also use $\HN$ to denote the composition 
\[ \HN : G(\bbA)/M(F)U(\bbA) \to \abs{\Bun_M(\bbF_q)} \to \Lambda^{+,\bbQ}_M. \]
The map $\HN$ will be our global analog of the map $\ord_{M,v}$ defined in 
\eqref{e:ordMG}. 

For $\lambda \in \Lambda^\bbQ$, let $[\lambda]_P$ denote the projection of $\lambda$ 
to $\Lambda^\bbQ_{G,P}$. Then for $x \in G(\bbA)/M(F)U(\bbA)$, 
we have $[\HN(x)]_P = \deg_P^\bbQ(x)$. 

\begin{rem} \label{rem:discreteHN}
There exists an integer $N$ such that the image of $\HN$ lies in 
$\frac 1 N \Lambda_M^+$.
\end{rem}

\subsection{The constant term operator} \label{ss:CT}
We will always fix the Haar measure on $U(\bbA)$ so that $U(\bbA)/U(F)$ has measure $1$. 

In \S\ref{sss:A_c}, we defined the spaces $\eA$ and $\eA_c \subset \eA$. 
The \emph{constant term} operator $\CT_P : \eA \to \eC_P$ is defined by
the formula 
\begin{equation}
\CT_P(f)(g) = \int_{U(\bbA)/ U(F)} f(g u) du, \quad\quad 
f\in \eA,\, g \in G(\bbA).
\end{equation}
In other words, $\CT_P$ is the pull-push along the diagram 
\begin{equation} \label{e:CTpp}
    G(\bbA)/G(F) \leftarrow G(\bbA) / P(F) \to G(\bbA)/M(F)U(\bbA). 
\end{equation}

Recall from \S\ref{sss:bddbelow} what it means for a subset
of $\Lambda^\bbQ$ to be bounded above (resp.~below) with respect to the 
partial (rational) ordering $\le_G^\bbQ$. 

\begin{lem}  \label{lem:CT-}
Let $f\in \eA_c$. Then $\HN(\supp \CT_P (f)) \subset \Lambda^{+,\bbQ}_M$ is bounded above. 
Consequently, $\CT_P : \eA \to\C_P$ sends $\eA_c$ to $\C_{P,-}$.
\end{lem}
\begin{proof}
If we pass to $K$-orbits in the diagram \eqref{e:CTpp}, then we get 
the $\bbF_q$-points of the diagram of stacks
\begin{equation}\label{e:CTBunpp}
    \Bun_G \leftarrow \Bun_P \to \Bun_M. 
\end{equation}
For $\theta \in \Lambda^{+,\bbQ}_G$, let 
$\Bun_G^{(\le \theta)} \subset \Bun_G$ denote the open substack of $G$-bundles
having Harder-Narasimhan coweight $\le_G^\bbQ \theta$. 
%By \cite[Proposition 7.3.5]{DG:CG}, the open substack $\Bun_G^{(\le \theta)}$ is quasi-compact. 
Let $f \in \eA_c$. 
Then the $K$-orbits of its support are contained in 
\[ \bigcup_{\theta\in S} \Bun_G^{(\le \theta)}(\bbF_q) \]
for a finite subset $S \subset \Lambda^{+,\bbQ}_G$. It follows from the 
definition of Harder-Narasimhan coweight that the image of 
\[ \Bun_P^{(\lambda)} := \Bun_P \xt_{\Bun_M} \Bun_M^{(\lambda)},\quad\quad \lambda \in \Lambda^{+,\bbQ}_M \] intersects $\Bun_G^{(\le \theta)}$ 
only if $\lambda \le_G^\bbQ \theta$ (cf.~\cite[Theorem 7.4.3(3)]{DG:CG}).
Now by pull-push along the diagram \eqref{e:CTBunpp}, 
we conclude that $\HN (\supp \CT_P(f) )$ is contained in 
the set of $\lambda \in \Lambda^{+,\bbQ}_M$ such that
$\lambda \le_G^\bbQ \theta$ for some $\theta \in S$. 
Therefore $\HN(\supp\CT_P(f))$ is bounded above.

Since $[\HN(x)]_P = \deg^\bbQ_P(x)$, we deduce that 
$\deg_P^\bbQ(\supp \CT_P(f)) \subset \{ [\theta]_P - \mu \mid \theta \in S, \mu \in \Lambda^{\pos,\bbQ}_{G,P} \}$. By definition, this means that $\CT_P(f) \in \C_{P,-}$.
\end{proof}

\subsection{The operator $R_P : \C_{P^-,+} \to \C_{P,-}$}  \label{ss:R_P}
Let $Z$ denote the space of pairs $(g_1,g_2)$, where 
$g_1 \in G(\bbA)/P^-(F),\, g_2 \in G(\bbA)/P(F)$ have equal image in 
$G(\bbA) / (P^-\cdot P)(F)$. We have projections from $Z$ to 
$G(\bbA)/P^-(F)$ and $G(\bbA)/P(F)$. 

\medskip

Define $R_P : \C_{P^-,+} \to \C_{P}$ to be the pull-push along the diagram
\[ G(\bbA)/M(F)U^-(\bbA) \leftarrow G(\bbA)/P^-(F) \leftarrow Z
\to G(\bbA)/P(F) \to G(\bbA)/M(F)U(\bbA). \]
%as the pull-push along the diagram
%\[ G(\bbA)/M(F)U^-(\bbA) \leftarrow G(\bbA) / P^-(F) \leftarrow G(\bbA)/M(F)
%\to G(\bbA)/P(F) \to G(\bbA)/M(F)U(\bbA). \]
\noindent 
Equivalently, $R_P$ is given by the explicit formula
\begin{equation} \label{eqn:R_P}
	R_P(\varphi)(g) = \int_{U(\bbA)} \varphi(g u)d u, \quad \quad
	\varphi \in \C_{P^-,+},\, g \in G(\bbA). 
\end{equation}
It is evident from the definition that $R_P$ is $G(\bbA)$-equivariant.

\begin{prop} \label{prop:R_P}
The operator $R_P : \C_{P^-,+} \to \C_P$ is well-defined, and the
image of $R_P$ is contained in $\C_{P,-}$.
More specifically, for $\varphi \in \C_{P^-,+}$ we have 
\[ \HN(\supp R_P(\varphi) ) \subset \{ \lambda - \mu \mid \lambda \in \HN(\supp \varphi),\, \mu \in\Lambda^{\pos,\bbQ}_G \}.\] 
\end{prop}
\begin{proof}
Let $(g_1,g_2) \in Z$. 
The quotient $(K \xt K)\bs Z$ identifies with the set of isomorphism classes 
of the $\bbF_q$-points of the stack $\Maps^\circ(X, P^- \bs G / P)$ of maps generically landing in $P^- \bs (P^-\cdot P) / P$. 
By \cite[Proposition 3.2]{BFGM}, the stack $\Maps^\circ(X, P^- \bs G / P)$
is isomorphic to the relative version of the open Zastava space 
$\oo\eZ_{\Bun_M}$. 
In particular, there is a map $\oo\eZ_{\Bun_M} \to 
\eH^+_M$, where $\eH^+_M  := \Maps^\circ(X, M \bs \wbar M / M)$ is the Hecke substack introduced in \S\ref{sect:H_M}. 
This map is induced from the contraction $G \to \wbar M$ defined in 
\cite[\S 4.2.9]{Wa}. 
The image of $g_1$ in $K \bs G(\bbA) / M(F)U^-(\bbA) = \abs{\Bun_M(\bbF_q)}$ 
defines an $M$-bundle $\eF_M^1$. Similarly, the image of $g_2$ in 
$K \bs G(\bbA) / M(F)U(\bbA) = \abs{\Bun_M(\bbF_q)}$ defines $\eF_M^2$. 
Then $(g_1,g_2) \in Z$ maps to a point $(\eF_M^1, \eF_M^2, \beta_M) \in \eH^+_M(\bbF_q)$, where $\beta_M$ is an $\wbar M$-morphism $\eF_M^2 \to \eF_M^1$ in the language of \S\ref{sect:tildeM-mor}. 

Let $\lambda_1,\lambda_2 \in \Lambda^{+,\bbQ}_M$ be the Harder-Narasimhan coweights of $\eF_M^1, \eF_M^2$ respectively. 
Then Remark~\ref{rem:heckeHN} implies that $\lambda_1 - \lambda_2 \in w_0^M\Lambda^{\pos,\bbQ}_G$.
The definition of $\C_{P^-,+}$ implies that 
the set of $\lambda_1 \in \Lambda^{+,\bbQ}_M$ for which $\varphi(g_1) \ne 0$ 
satisfy $-[\lambda_1]_P \in S_0 + \Lambda^{\pos,\bbQ}_{G,P}$ for a finite set $S_0$. 
We deduce that the intersection of $\HN(\supp \varphi)$ with 
$\lambda_2 + w_0^M\Lambda^{\pos,\bbQ}_G$ is finite. 
Thus $R_P(\varphi)(g_2)$ is an integral over the $K$-orbits of $G(\bbA)/M(F)U^-(\bbA)$ corresponding to the union of $\Bun^{(\lambda_1)}_M(\bbF_q)$ 
ranging over a finite set of $\lambda_1$, i.e., $R_P(\varphi)$ is well-defined. 

Remark~\ref{rem:heckeHN} also gives the inequality $\lambda_2 \le_G^\bbQ \lambda_1$, 
which proves the second statement of the proposition. 
It immediately follows that $R_P(\varphi) \in \C_{P,-}$.
\end{proof}

\subsection{Invertibility of $R_P$} 
Below we will prove Proposition~\ref{prop:I^-1}, which says that 
the operator $R_P : \C_{P^-,+} \to \C_{P,-}$ is invertible. 
We deduce the proposition from the local results of \S\ref{s:local}.  

\subsubsection{} Fix the Haar measure on $G(\bbA)$. 
Recall that we defined a generalized function $\xi_P$
on $\bbX_P(\bbA)$ by \eqref{e:xi_P}, which slightly depends on
the choice of measure on $G(\bbA)$. 
The Haar measures on $G(\bbA)$ and $U(\bbA)$
induce a $G(\bbA)$-invariant measure on $G(\bbA)/U(\bbA)$. 

\begin{prop} \label{prop:I^-1}
The map $R_P : \C_{P^-,+} \to \C_{P,-}$ is an isomorphism.
The inverse is given by the formula 
\begin{equation} \label{e:I^-1}
R_P^{-1} (\varphi)(g_2) = \int_{G(\bbA)/U(\bbA)} \varphi(g_1)\xi_P(g_1,g_2) dg_1, \quad\quad \varphi \in \C_{P,-},\, g_2 \in G(\bbA).
\end{equation}
\end{prop}
\begin{proof} 
Let us first show that the right hand side of \eqref{e:I^-1} is well-defined
for any $\varphi \in \C_{P,-}$ and $g_2 \in G(\bbA)$. 
Since $\varphi$ is $K$-finite, there exists a compact open subgroup $K' = \prod K'_v \subset K$ such that $\varphi$ is $K'$-invariant. 
Let $\delta_{K'} \in C^\infty_c(G(\bbA))$ equal $\frac 1{\on{mes}(K')}$ times the characteristic function of $K'$. Recall that $(\delta_{K'} \ot 1) * \xi_P = \Asymp_P(\delta_{K'}) \in \C_b(\bbX_P(\bbA))$, where $\Asymp_P$ is defined in \S\ref{sss:Asymp}. 
Thus the r.h.s.~of \eqref{e:I^-1} equals 
\[ \int_{G(\bbA)/U(\bbA)} \varphi(g_1) \Asymp_P(\delta_{K'})(g_1,g_2)dg_1. \] 
Let $\ord_{M,v} : \bbX_P(F_v) \to \Lambda^+_M$ be the map \eqref{e:ordMX}. 
Proposition~\ref{prop:Xbound} implies that for every $v$ there exists a finite subset $S_v \subset \Lambda$ such that 
\[ \ord_{M,v}(\supp(\Asymp_{P,v}(\delta_{K'_v}))) \subset 
\{ \theta - \mu \mid \theta \in S_v,\, \mu \in \Lambda^{\pos,\bbQ}_{U} \}, \]
and we can take $S_v = \{0\}$ for almost all $v$ by Remark~\ref{rem:AsympK}. 
For $g_1\in G(\bbA)/U(\bbA)$, consider the image of $(g_1,g_2)$ in $\bbX_P(\bbA)$. 
Let $\eF^1_M \in \Bun^{(\lambda_1)}_M(\bbF_q)$ (resp.~$\eF^2_M \in \Bun^{(\lambda_2)}_M(\bbF_q)$) be the image 
of $g_1$ in $K \bs G(\bbA)/ M(F)U(\bbA)$ (resp.~$g_2$ in $K\bs G(\bbA)/M(F)U^-(\bbA)$).  
Lemma~\ref{lem:heckeHN} implies that 
\begin{equation} \label{e:ineqasymp}
 -\sum_v \log_q(q_v) \cdot \ord_{M,v}(g_1,g_2) \le_M^\bbQ  \lambda_1 - \lambda_2 \le_M^\bbQ -w_0^M\sum_v \log_q(q_v) \cdot \ord_{M,v}(g_1,g_2).
\end{equation}
Let $S = \{ \sum_v \log_q (q_v) \cdot \lambda_v \mid \lambda_v \in S_v \}$. 
Suppose $\Asymp_P(\delta_{K'})(g_1,g_2) \ne 0$. 
Then we have 
\begin{align} \label{e:ineqS1}
\lambda_1 &\in \{ \lambda_2 - w_0^M\theta + \mu \mid \theta \in S, \mu \in w_0^M\Lambda^{\pos,\bbQ}_G \} \\
\label{e:ineqS2}
\lambda_2 &\in \{ \lambda_1 + \theta - \mu \mid \theta \in S,\, \mu \in \Lambda^{\pos,\bbQ}_G \}
\end{align}
and we emphasize that $S$ is a finite set depending only on the
stabilizer in $K$ of $\varphi$. 
From \eqref{e:ineqS1} and the definition of $\varphi\in \C_{P,-}$, we conclude that the r.h.s.~of \eqref{e:I^-1} is a finite integral, which we temporarily denote $T(\varphi)(g_2)$. 
From \eqref{e:ineqS2} we deduce that $T(\varphi)$ defines an element in $\C_{P^-,+}$.

\smallskip

It remains to show that $T : \C_{P,-} \to \C_{P^-,+}$ is inverse to 
$R_P$. Let $\varphi \in \C_{P,-}$. 
Then \eqref{e:ineqS2} implies that 
$\HN(\supp T(\varphi)) \subset \{ \lambda + \theta - \mu \mid \lambda \in \HN(\supp\varphi),\,\theta \in S,\, \mu \in \Lambda^{\pos,\bbQ}_G \}$.
Proposition~\ref{prop:R_P} implies in turn that 
\[ \HN(\supp R_P(T\varphi)) \subset \{ \lambda + \theta - \mu \mid \lambda \in \HN(\supp\varphi),\,\theta \in S,\, \mu \in \Lambda^{\pos,\bbQ}_G \}.\] 
Therefore we deduce that to show $R_P \circ T = \id$, it suffices to check 
the equality for $\varphi \in \C_{P,c}$. 
Any such $\varphi$ is the pushforward of an element in 
$C^\infty_c(G(\bbA)/U(\bbA))$, which is isomorphic
to the restricted tensor product of $\C_{P,c,v}$ over all places $v$. 
Since $R_P \circ T$ is defined as a product of local integrals, 
$(R_P \circ T)(\varphi) = \varphi$ follows from Proposition~\ref{prop:R_Pv^-1}.

Similarly, it suffices to check that $T \circ R_P = \id$ on $\varphi \in \C_{P^-,c}$. This again follows from the corresponding local statement Proposition~\ref{prop:R_Pv^-1}.
\end{proof}

\subsection{A formula for $\eB_P$ in terms of $R_P$}
\label{ss:altdef}
We give a formula for the bilinear form $\eB_P$ defined in \S\ref{ss:B_P} in terms of the global intertwining operator $R_P$. This formula is the analog of \cite[Definition 3.1.1]{DW} for a general reductive group $G$. 

\subsubsection{} Fix some Haar measure on $G(\bbA)$, and 
fix the Haar measure on $U^-(\bbA)$ such that $\on{mes}(U^-(\bbA)/U^-(F)) = 1$. 
Then we get an invariant measure on $G(\bbA)/M(F) U^-(\bbA)$ 
and therefore a pairing between $\C_{P^-,c}$ and $\C_{P^-}$ 
defined by  
\begin{equation} 
\brac{ \varphi_1, \varphi_2 } := \int_{G(\bbA)/M(F)U^-(\bbA)} \varphi_1(x) \varphi_2(x) dx.
\end{equation}
Lemma~\ref{lem:CT-} implies that this pairing is well-defined 
when $\varphi_1 \in \C_{P^-,+}$ and $\varphi_2 \in \CT_{P^-}(\eA_c)$. 

\begin{prop} \label{prop:alt}
For any $f_1, f_2 \in \eA_c$, one has 
\begin{equation}  \label{e:Sak}
\eB_P(f_1, f_2) = \brac{ R_P^{-1}\CT_P(f_1), \CT_{P^-}(f_2) }. 
\end{equation}
\end{prop}

\begin{proof}
Choose $\tilde f_1, \tilde f_2 \in C^\infty_c(G(\bbA))$ that pushforward
to $f_1,f_2$. Then 
\[ \CT_P(f_1)(g) = \sum_{\gamma \in G(F)/U(F)} \int_{U(\bbA)} 
\tilde f_1(g u \gamma^{-1})du, \quad\quad g\in G(\bbA). \]
The formula \eqref{e:I^-1} directly gives 
\[ (R_P^{-1}\circ \CT_P)(f_1)(g_2) =  
\sum_{\gamma \in G(F)/U(F)} \int_{G(\bbA)} \tilde f_1(g_1 \gamma^{-1}) \xi_P(g_1, g_2) dg_1.
\]
Let $\tilde f_1^\vee(g) := \tilde f_1(g^{-1})$. 
By left $G(\bbA)$-equivariance of $\Asymp_P$, the right hand side of \eqref{e:Sak} equals 
\begin{multline*}
 \underset{G(\bbA)/P^-(F)}\int \sum_{\gamma \in G(F)/U(F)} 
\Asymp_P(\tilde f_1^\vee)(\gamma, g) f_2(g) dg \\ = 
\underset{G(\bbA)/G(F)}\int \sum_{x\in \bbX_P(F)} \Asymp_P(\tilde f_1^\vee)((1,g)x) f_2(g)dg. 
\end{multline*}
By right $G(\bbA)$-equivariance of $\Asymp_P$ and \eqref{e:tildeB}, we
conclude that the right hand side equals $\eB_P(f_1,f_2)$.
\end{proof}

\begin{proof}[Proof of {Theorem~\ref{thm:global}}] 
The theorem follows immediately from \eqref{e:defB} and Proposition~\ref{prop:alt}. 
\end{proof}

\section{The operator $L$ and its inverse}  \label{s:L}
We first recall basic facts from the theory of Eisenstein series. 
Then we define the operator $L:\eA_c \to\eA$ in terms of the Eisenstein operator, the inverse of the standard intertwining operator, and the constant term operator. 
Motivated by a characterization of $\eA_c$ due to Harder, we define the subspace $\eA_\psc \subset \eA$ of ``pseudo-compactly'' supported functions in \S\ref{ss:A_psc}. 
We check that $L$ sends $\eA_c$ to this new subspace $\eA_\psc$. 
Lastly, we prove that $L : \eA_c \to \eA_\psc$ is invertible in Theorem~\ref{thm:L^-1} and give the formula for its inverse.

We will continue to use the notation from \S\ref{s:global}. 
In this section, we will assume that the field of coefficients $E$ equals $\bbC$.

\subsection{Constant term revisited} Recall that in Lemma~\ref{lem:CT-}, we showed 
that the constant term operator $\CT_P : \eA \to \C_P$ sends $\eA_c$ to $\C_{P,-}$
for all parabolic subgroups $P$. The proof of \cite[Theorem 1.2.1]{Harder} 
shows that the converse is also true:

\begin{lem} \label{lem:compactCT}
For $f\in \eA$, we have $f$ is compactly supported if and only if $\CT_P(f)$ lies in $\C_{P,-}$ for all standard parabolics $P$.
\end{lem}
\begin{proof}
The ``only if'' direction is proven by Lemma~\ref{lem:CT-}.

Note that if $f\in \C_{G,-}\subset \eA$, then $\deg^\bbQ_G(\supp f)$ is finite. 
Therefore to prove the ``if'' direction, we may assume that $f\in \eA$ and 
there is a fixed $\theta \in \Lambda^\bbQ$ such that 
\[ \deg_P^\bbQ(\supp \CT_P(f)) \subset -\Lambda^{\pos,\bbQ}_{G,P} + [\theta]_P  \] 
for all standard parabolics $P$. 
By reduction theory, there exists a number $c$ such that any $x \in G(\bbA)/G(F)$ 
has a representative $g \in G(\bbA)$ with $\brac{\check\alpha,\deg_B^\bbQ(g)} > c$
for all simple roots $\check\alpha$ of $G$. % and g satisfies Lemma I.2.7 of {MW}
Suppose $f(g) \ne 0$ for $g\in G(\bbA)$ with $\lambda = \deg_B^\bbQ(g) \in \Lambda^\bbQ$ as above.
By \cite[Lemma I.2.7]{MW} (cf.~\cite[Lemma 1.2.2]{Harder}), there exists $c'$ such that if 
$\brac{\check\alpha, \lambda} > c'$ for all simple roots $\check\alpha$ 
which are not simple roots of $M$ for some standard parabolic $P$ with Levi $M$, then $\CT_P(f)(g)=f(g)$.
Let $M$ be the Levi such that the simple roots of $M$ are precisely the simple roots $\check\alpha$ of $G$ such that $\brac{\check\alpha,\lambda} \le c'$. Then $\CT_P(f)(g)=f(g)\ne 0$
implies that $\deg_P^\bbQ(g)=[\lambda]_P$ lies in $-\Lambda^{\pos,\bbQ}_{G,P}+[\theta]_P$.
On the other hand the choice of $M$ implies that $\lambda$
belongs in a translate of $-\Lambda^{+,\bbQ}_M \subset -\Lambda^{\pos,\bbQ}_M + \Lambda^\bbQ_{Z(M)}$. We deduce that the set of all possible $\lambda$ is bounded above. 
We also have $\brac{\check\alpha,\lambda}>c$ for all simple roots $\check\alpha$, so 
there are in fact only finitely many possibilities for $\lambda$. 
We conclude that $f$ is compactly supported.
\end{proof}

For $P=MU$ a standard parabolic, we also use the notation $\CT^G_M :=\CT_P$ below.

\medskip

Let $P_1 \subset P$ be standard parabolic subgroups with Levi subgroups $M_1 \subset M$. 
Then the constant term operator $\CT^M_{M_1}$ can be considered as an operator 
$\CT^M_{M_1} : \C_P \to \C_{P_1}$. We say that $\varphi \in \C_P$ is $M$-cuspidal 
if $\CT^M_{M_1}(\varphi)=0$ for all standard Levi subgroups $M_1 \subset M$.

\subsection{Eisenstein operator} 
Let $P=MU$ be a parabolic subgroup of $G$. We define the Eisenstein operator\footnote{The authors of \cite{MW} call it ``pseudo-Eisenstein''.} $\Eis_P : \C_{P,c} \to \eA_c$ to be the pull-push along the diagram
\begin{equation}\label{e:Eispp}
    G(\bbA)/M(F)U(P)(\bbA) \leftarrow G(\bbA)/P(F) \to G(\bbA)/G(F), 
\end{equation}
where the left arrow is proper. Explicitly,
\[ \Eis_P(\varphi)(g) := \sum_{\gamma \in G(F)/P(F)} \varphi(g\gamma), 
\quad\quad \varphi \in \C_{P,c},\, g\in G(\bbA). \] 
We also use the notation $\Eis^G_M := \Eis_P$ for $P$ standard.

It is well known that 
\begin{equation} \label{e:(Eis,CT)}
    \brac{ \CT_P(f_1), \varphi_2 } = \eB_\naive( f_1, \Eis_P(\varphi_2) )
\end{equation}
for $f_1 \in \eA,\, \varphi_2\in \C_{P,c}$. 
By this adjunction, we see that it is actually possible to define 
$\Eis_P(\varphi_2)$ for any $\varphi_2$ such that
$\brac{\varphi_1,\varphi_2}$ is finite for all $\varphi_1 \in \CT_P(\eA_c)$.
Lemma~\ref{lem:CT-} implies that all $\varphi_2 \in \C_{P,+}$ satisfy this condition.
Thus $\Eis_P$ extends to  an operator 
\[ \Eis_P : \C_{P,+} \to \eA. \]

\subsection{Intertwining operators revisited}
In this section we recall some facts about the standard intertwining operators corresponding to elements of the Weyl group.
%Let $W(M)$ be the set of elements $w$ of the Weyl group of $G$ of minimal length modulo the Weyl group of $M$, and such that $wMw^{-1}$ is the Levi of a standard parabolic subgroup of $G$.

Let $P = MU$ and $P'=M'U'$ be \emph{standard} parabolic subgroups such that
$M'= wMw^{-1}$ for some $w\in W$. 
Then the intertwining operator\footnote{In \cite[II.1.6]{MW}, $R_w$ is denoted $M(w,\pi)$ for a cuspidal representation $\pi$. But we prefer to avoid multiple uses of the letter $M$.}  
 $R_w$ is an operator $\C_{P,c} \to \C_{P'}$
defined by the explicit formula 
\begin{equation}\label{e:I_w}
    (R_w \varphi)(g) = \int_{U'(\bbA)/ (U'(\bbA) \cap w U(\bbA) w^{-1})} \varphi(g u w) du. 
\end{equation}
Proposition~\ref{prop:CTEis} below implies that if $\varphi \in \C_{P,c}$, then 
$R_w\varphi \in \C_{P',-}$. % since $(\CT_{P'} \circ \Eis_P)(\varphi) \in \C_{P',-}$. 
The same proposition also implies that $R_w\varphi \in \C_{P'}$ converges absolutely for 
any $\varphi \in \C_{P,+}$.

We will use the extra notation $R^G_{M,w} = R_w$ when necessary for clarity (note that the standard parabolic $P'$ is determined by its Levi $wMw^{-1}$).

When $P=MU$ and $P'=M'U'$ are two (not necessarily standard) parabolic subgroups containing $T$ and $M'=wMw^{-1}$ for $w\in W$, we will use the notation
\[ R_{P':P,w} : \C_{P,c} \to \C_{P',-} \] 
for the intertwining operator, which is defined by the same formula \eqref{e:I_w}. 
We also use $R_{P':P}$ to denote $R_{P':P,1}$. 
So the operator $R_P$ considered in \S\ref{ss:R_P} is now denoted by 
$R_{P:P^-}$. 

%Let $P$ and $P'$ be standard parabolic subgroups 
%such that $M' = wMw^{-1}$. Then $w$ defines a linear isomorphism $\Lambda^\bbQ_{G,P} \to \Lambda^\bbQ_{G,P'}$. Let $\C_{P',-}^w \subset \C_{P'}$ be the set of all functions $\varphi \in \C_{P'}$ such that $\deg_{P'}^\bbQ(\supp \varphi)$ is contained in finitely many translates of the rational cone $(-\Lambda^{\pos,\bbQ}_{G,P'}) \cap w \Lambda^{\pos,\bbQ}_{G,P} \subset \Lambda^\bbQ_{G,P'}$. Let $\C_{P,+}^w \subset \C_P$ be the set of all functions $\varphi \in \C_P$ such that $\deg_P^\bbQ(\supp \varphi)$ is contained in finitely
%many translates of the rational cone $(-w^{-1} \Lambda^\pos_{G,P'}) \cap \Lambda^{\pos,\bbQ}_{G,P} \subset \Lambda^\bbQ_{G,P}$. 

Let $V \subset \C_P$ denote the subspace of functions $\varphi \in \C_P$ 
such that $R_{P':P,w}(\varphi)$ converges absolutely. 
Let $V' \subset V$ denote the subspace of $\varphi\in V$ such that $R_{P':P,w}(\varphi)$ is compactly supported. 
%Proposition~\ref{prop:CTEis} below implies that $R_{P':P,w}(\varphi)$ converges 
%absolutely whenever $\Eis_P(\varphi)$ does. In particular, this includes all $\varphi \in \C_{P,+}$.

\begin{prop} \label{lem:I_supp}
The map $R_{P':P,w} : V' \to \C_{P',c}$ is an isomorphism.  
The inverse map $R_{P':P,w}^{-1} : \C_{P',c} \to \C_P$ is an integral operator.
\end{prop}
\begin{proof}
It suffices to consider the case $w=1$. 
Let $\varphi \in V'$. Then $\varphi' := R_{P':P}\varphi \in \C_{P',c}$ 
and $R_{P^-:P'}(\varphi') \in \C_{P^-,-}$ must equal $R_{P^-:P}\varphi$. 
The operator $R_{P^-:P} : \C_{P,+} \to \C_{P^-,-}$ is invertible by Proposition~\ref{prop:I^-1}, 
so we have $(R_{P^-:P}^{-1} \circ R_{P^-:P'})(R_{P':P}\varphi) = \varphi$, which proves injectivity.

Now take $\varphi' \in \C_{P',c}$. Since $\C_{P',c} \subset \C_{P',-}$, we have 
$R_{P':P'^-}^{-1}(\varphi') \in \C_{P'^-,+}$.
Then $\varphi := (R_{P:P'^-}\circ R_{P':P'^-}^{-1})(\varphi') \in \C_P$ is well-defined. 
Moreover, $R_{P':P}(\varphi) =( R_{P':P'^-} \circ R_{P':P'^-}^{-1})(\varphi') = \varphi'$. 
Hence $\varphi \in V'$, and we have shown surjectivity of $R_{P':P}$.
\end{proof}

\begin{rem}
The operators $R_{P':P,w}$ and $R_{P':P,w}^{-1}$ are defined on larger spaces of functions, but we will not consider the corresponding support conditions in this article.
\end{rem}

\subsection{The composition $\CT_{P'} \circ \Eis_P$} 
Let $P=M U$ and $P' = M' U'$ be standard parabolic subgroups of $G$. 
Let $\check\Phi^+_M$ denote the set of positive roots of $M$. 

Let $\varphi \in \C_{P,c}$ be $M$-cuspidal. 
Then \cite[Proposition II.1.7]{MW} gives the formula
\begin{equation} \label{e:CTEiscusp}
(\CT_{P'} \circ \Eis_P)(\varphi) = \sum_{w\in W(M,M')} (\Eis^{M'}_{wMw^{-1}} \circ R^G_{M,w})(\varphi).
\end{equation}
where $W(M,M') := \{ w \in W \mid w^{-1}\check\alpha >0,\, \forall \check\alpha \in \check\Phi^+_{M'} \text{ and } wMw^{-1} \text{ is a standard Levi of } M' \}$. 

\begin{prop} \label{prop:CTEis}
Let $\varphi \in \C_{P,+}$ be arbitrary. Then one has
\begin{equation} \label{e:CTEis}
(\CT_{P'} \circ \Eis_P)(\varphi) = \sum_{w\in W^\bullet_{M,M'}} (\Eis^{M'}_{wM_1 w^{-1}} \circ R^G_{M_1,w} \circ \CT^M_{M_1})(\varphi) 
\end{equation}
where $W^\bullet_{M,M'} := \{ w\in W \mid \forall \check\alpha \in \check\Phi_M^+,\, w\check\alpha > 0,\; \forall \check\alpha \in \check\Phi_{M'}^+,\, w^{-1}\check\alpha >0\}$ and $M_1 := M \cap w^{-1}M'w$, and every term on the right hand side converges absolutely.
\end{prop}
\begin{proof}
See the proof of \cite[Proposition II.1.7]{MW}.
\end{proof}

\subsection{The operator $L : \eA_c \to \eA$}
One has the operators
\[ \eA_c \overset{\CT_P}\longrightarrow \C_{P,-} \overset{R_{P:P^-}^{-1}}\longrightarrow
\C_{P^-,+} \overset{\Eis_{P^-}}\longrightarrow \eA. \]
Thus we deduce that 
\begin{equation} 
    \eB(f_1, f_2) = \eB_\naive(L f_1, f_2) 
\end{equation}
where the operator $L : \eA_c \to \eA$ is defined by 
\[ L := \sum_P (-1)^{\dim Z(M)} \Eis_{P^-} \circ R_{P:P^-}^{-1} \circ \CT_P \] 
and the sum ranges over the standard parabolic subgroups. 
Unlike the form $\eB$, the operator $L$ does not depend on the choice of Haar measure on $G(\bbA)$.

Observe that for $f \in \eA_c$ cuspidal, we have $Lf = (-1)^{\dim Z(G)} f$. 

\begin{rem}
Theorem~\ref{thm:B=b} implies that the miraculous duality functor $\on{Ps-Id}_{\Bun_G,!}$ defined in 
\cite[\S 4.4.8]{DG:CG} is the D-module analog of the operator $q^{-\dim \Bun_G} \cdot L^K : \eA_c^K \to \eA^K$ on $K$-invariants via the functions--sheaves dictionary (cf.~\cite[\S A.8.4]{DW}).
\end{rem}

The following proposition shows the interplay between the operator $L$ and the Eisenstein operators. 
\begin{prop} 
Let $P$ be a standard parabolic subgroup. 
Let $\varphi \in \C_{P,c}$ be $M$-cuspidal. Then 
\[ (L \circ \Eis_P)(\varphi) = (-1)^{\dim Z(M)} (\Eis_{P^-} \circ R_{P:P^-}^{-1})(\varphi). \]
\end{prop}
\begin{proof} Let $P'$ be another standard parabolic subgroup.
By formula \eqref{e:CTEiscusp}, 
\begin{equation}  \label{e:LEis1}
 (\Eis_{P'^-} \circ R_{P':P'^-}^{-1} \circ \CT_{P'} \circ \Eis_P) (\varphi) = 
\sum_{w\in W(M,M')} (\Eis_{P'^-} \circ R_{P':P'^-}^{-1} \circ \Eis^{M'}_{wMw^{-1}} \circ R^G_{M,w})(\varphi) 
\end{equation}
and each term on the r.h.s.~converges absolutely.
Let $P'_1 \subset M'$ be the standard parabolic subgroup with Levi $wMw^{-1}$, so $\Eis^{M'}_{wMw^{-1}}=\Eis_{P'_1}$. 
One can check that 
\[ R_{P':P'^-}^{-1} \circ \Eis_{P'_1} = \Eis_{P'_1} \circ R_{P'_1 U': P'_1 U'^-}^{-1} \] 
when either side converges absolutely (recall that $R_{P':P'^-}^{-1}$ is defined as a product of local integrals by \eqref{e:I^-1}).
Let $P'_2 = w_0^{M'} P'^-_1 w_0^{M'} \subset M'$. Let $Q = P'_2 U'$ be the standard parabolic subgroup of $G$ with Levi $w'Mw'^{-1},\, w' = w_0^{M'}w$.
Then $\Eis_{P'_1} \circ R_{P'_1 U': P'_1 U'^-}^{-1} = 
\Eis_{P'^-_2} \circ R_{P'^-_2 U' : Q^-}^{-1} \circ w_0^{M'}$,
where $w_0^{M'}$ denotes right translation by a representative of $w_0^{M'}$ in $M'(F)$.

From the definition \eqref{e:I_w}, we have 
\[ (w_0^{M'} \circ R^G_{M,w})(\varphi)(g) = \int_{(U_{P'_2}^- U')(\bbA) \cap w' U^-(\bbA) w'^{-1}} \varphi(g u w') du, \quad\quad g\in G(\bbA). \]
If $\check\alpha$ is a positive root of $G$ not in $M$, then $w'\check\alpha = w_0^{M'}w\check\alpha$ must be negative if it is a root of $M'$, by definition of $W(M,M')$.
Thus 
$U_{P'_2}^- U' \cap w' U^- w'^{-1} = U' \cap w' U^-w'^{-1}$ and 
$U_{P'_2} U' \cap w' U w'^{-1} = U' \cap w' U w'^{-1}$.
We deduce that 
$R_{P'^-_2 U' : Q^-} = w_0^{M'} \circ R^G_{M,w} \circ R_{P: Q^-, w'^{-1}}$ 
and therefore 
\[ (R_{P'_1U' : P'_1 U'^-}^{-1} \circ w_0^{M'} \circ R^G_{M,w})(\varphi)
= R_{P:Q^-, w'^{-1}}^{-1}(\varphi) \]
for $\varphi \in \C_{P,c}$.
Next, observe that $\Eis_{P'^-} \circ \Eis_{P'^-_2} = \Eis_{Q^-}$. 
Thus the r.h.s.~of \eqref{e:LEis1} equals 
\[ \sum_{w\in W(M,M')} (\Eis_{Q^-} \circ R_{P:Q^-,w'^{-1}}^{-1} )(\varphi). \]
Note that $Q$ depends only on $w'Mw'^{-1}$ and not on $M'$. 
Summing over all standard parabolics $P'$, we get the formula
\begin{equation} \label{e:LEisPtemp}
(L \circ \Eis_P)(\varphi) = 
\sum_{w'\in W}\left( \sum_{M' \left| \substack{M'\cap B^- \subset w'B w'^{-1}\\ w'Mw'^{-1}\subset M'}\right.} (-1)^{\dim Z(M')} \right) 
(\Eis_{Q^-} \circ R_{P:Q^-,w'^{-1}}^{-1})(\varphi).
\end{equation}
Fixing $w' \in W$, let us classify all $M'$
such that $M' \cap B^- \subset w'Bw'^{-1}$ and $w'Mw'^{-1} \subset M'$. 
The two conditions above are equivalent to
requiring $\check\Delta_{M'} \subset w' \check\Phi^-_G$ and 
$\check\Delta_G \cap w' \check\Phi_M \subset \check\Delta_{M'}$.
Thus the inner sum in \eqref{e:LEisPtemp} equals 
\[ (1+(-1))^{\abs{\check\Delta_G \cap w'( \check\Phi^-_G - \check\Phi_M)}}, \]
which vanishes unless $w'( \check\Phi^-_G - \check\Phi_M)$
contains no simple roots. 
If $w' \notin W_M$, then $B \not\subset w'Pw'^{-1}$, so 
$w'(\check\Phi^+_G \cup \check\Phi_M)$ does not contain all the simple roots, i.e., $w'( \check\Phi^-_G - \check\Phi_M)$
contains a simple root. Therefore in order for the sum to not vanish, $w'$ must be in $W_M$ and $M \cap B^- \subset w' B w'^{-1}$. 
Hence $w' = w_0^M$, and we have $M'=M,\, Q = P$. 
Since $R_{P:P^-,w_0^M} = R_{P:P^-}$ by $M(F)$-invariance, 
we get $(L \circ \Eis_P)(\varphi) = (-1)^{\dim Z(M)}(\Eis_{P^-} \circ R_{P:P^-}^{-1})(\varphi)$.
\end{proof}

For a standard parabolic $P$, define the ``second Eisenstein'' operator 
$\Eis'_P : \C_{P,-} \to \eA$ by 
\[ \Eis'_P := \Eis_{P^-} \circ R_{P:P^-}^{-1}.\] 

\smallskip

Let $\eA^M$ denote the space of smooth $K$-finite functions on $M(\bbA)/M(F)$.
Let $L^M : \eA^M_c \to \eA^M$ denote the operator $L$ with respect to the reductive group $M$. 
Applying $\ind^{G(\bbA)}_{P(\bbA)}$, we also let $L^M$ denote the 
induced operator $\C_{P,c}\to \C_P$.  
Recall that $L^M$ is $(-1)^{\dim Z(M)}$ times the identity on $M$-cuspidal functions in $\C_{P,c}$.

\begin{cor} One has the equality
\begin{equation} \label{e:LEis=EisL}
L \circ \Eis_P = \Eis'_P \circ L^M : \C_{P,c} \to \eA.
\end{equation}
\end{cor}

The operator $L$ is self-adjoint with respect to $\eB_\naive$, so 
\eqref{e:LEis=EisL} gives
\begin{equation} \label{e:CTL=LCT}
\CT_P \circ L = L^M \circ \CT'_P : \eA_c \to \C_P, 
\end{equation}
where $\CT'_P : \eA_c \to \C_{P,+}$ is defined by 
$\CT'_P = R_{P^-:P}^{-1} \circ \CT_{P^-}$.

\begin{rem}
As explained in \cite[\S A.11.7]{DW}, equation \eqref{e:LEis=EisL} is an analog
of the ``strange'' functional equation for geometric Eisenstein series stated in  \cite[Theorem 4.1.2]{G:miraculous}. 
\end{rem}

\begin{rem}\label{rem:Lcc}
Observe that for any $f\in \eA_c$, we have $\deg_G^\bbQ(\supp f) = \deg_G^\bbQ(\supp Lf) \subset \Lambda^\bbQ_{G,G}$ from the definitions (i.e., $L$ does not change the connected component of the image of the support in $\Bun_G$). 
As a consequence, for any $\varphi \in \C_P$ on which $L^M\varphi$ converges, we have $\deg_P^\bbQ(\supp \varphi) = \deg_P^\bbQ(\supp L^M\varphi)$.
\end{rem}

\subsection{The space $\eA_\psc$ of ``pseudo-compactly'' supported functions}
\label{ss:A_psc}
We are inspired by Lemma~\ref{lem:compactCT} to make the following definition. 

\begin{defn} 
Let $\eA_\psc$ be the space of all functions $f \in \eA$ such that
$\CT_P(f) \in \C_{P,+}$ for all standard parabolic subgroups $P \subset G$.
\end{defn}

%\mycomment{It is enough to check one representative in each conjugacy class of parabolics.}

\begin{prop} For any $f \in \eA_c$, one has $Lf \in \eA_\psc$. 
\end{prop}
\begin{proof}
Let $P=MU$ be a standard parabolic subgroup of $G$. 
By \eqref{e:CTL=LCT}, we have $\CT_P(Lf) = L^M(\CT'_P (f))$
and $\CT'_P(f) \in \C_{P,+}$. 
Remark~\ref{rem:Lcc} implies that $L^M(\CT'_P(f)) \in \C_{P,+}$ as well.
Hence $\CT_P(Lf) \in \C_{P,+}$ and $Lf \in \eA_\psc$.
\end{proof}

\begin{thm} \label{thm:L^-1}
The operator $L : \eA_c \to \eA_\psc$ is invertible. 
For $f \in \eA_\psc$ one has
\begin{equation} \label{e:Linverse}
    L^{-1}f = \sum_P (-1)^{\dim Z(M)} (\Eis_P \circ \CT_P)(f). 
\end{equation}
where the sum ranges over standard parabolic subgroups.
\end{thm}
\begin{proof}
For $f \in \eA_\psc$, set $L' f$ equal to the r.h.s.~of \eqref{e:Linverse}. 
First we need to check that $L' f \in \eA_c$.  
Let $P'$ be another standard parabolic. Then by \eqref{e:CTEis}, 
\[ (\CT_{P'}\circ \Eis_P \circ \CT_P)(f) = \sum_{w\in W^\bullet_{M,M'}}
(\Eis^{M'}_{wM_1 w^{-1}} \circ R^G_{M_1,w} \circ \CT^G_{M_1} )(f) \]
where $W^\bullet_{M,M'} = \{ w\in W \mid w\check\alpha > 0,\, \forall \check\alpha \in \check\Phi_M^+,\; w^{-1}\check\alpha >0,\, \forall \check\alpha \in \check\Phi_{M'}^+ \}$ is a set of representatives for $W_{M'} \bs W / W_M$, and $M_1 = M \cap w^{-1}M'w$. 
Summing over $P$, we have
\begin{multline*}
 \CT_{P'}(L'f) = \\
\sum_{w\in W} \sum_{M_1 \mid wM_1w^{-1}\subset M'} \left(\sum_{M \left| \substack{M_1 = M\cap w^{-1}M'w \\ w\in W^\bullet_{M,M'}}\right.} (-1)^{\dim Z(M)}\right) (\Eis_{wM_1 w^{-1}}^{M'} \circ R^G_{M_1,w} \circ \CT^G_{M_1}) (f). 
\end{multline*}
Note that $\check\Delta_{M_1}$ is any subset of $w^{-1}\check\Phi^+_{M'} \cap \check\Delta_G$, while $\check\Delta_M - \check\Delta_{M_1}$ is any subset of $w^{-1}(\check\Phi^+_G - \check\Phi^+_{M'}) \cap \check\Delta_G $.
Thus the inner sum over $M$ vanishes unless $w^{-1}(\check\Phi^+_G - \check\Phi^+_{M'})$ contains no simple roots. This only occurs if $w = w_0^{M'} w_0$
and $M_1=M$. By considering $M'_1 = w_0^{M'}w_0 M_1 w_0 w_0^{M'}$ instead of $M_1$, we see that 
\begin{equation}\label{e:CTLinv} 
    \CT_{P'}(L' f) = \sum_{M'_1 \subset M'} (-1)^{\dim Z(M'_1)}(\Eis^{M'}_{M'_1} 
\circ R_{P'_1:(M'\cap P'_1)U'^-} \circ \CT_{(M'\cap P'_1)U'^-}) (f) 
\end{equation}
where the sum is over all Levi subgroups of $M'$, and $P'_1$ is the standard parabolic subgroup of $G$ with Levi $M'_1$. 
Let $Q = (M'\cap P'_1)U'^-$. Then $w_0 w_0^{M'} Qw_0^{M'} w_0$ is a standard parabolic subgroup. Since $f\in \eA_\psc$, we have 
$\CT_Q(f) \in \C_{Q,+}$. 
Observe that $\Eis^{M'}_{M'_1} \circ R_{P'_1:Q} = R_{P':P'^-} \circ \Eis^{M'}_{M'_1}$ when either side converges absolutely. 
Also note that 
$\deg_{P'}^\bbQ(\supp \Eis^{M'}_{M'_1}(\CT_Q f)) = \deg_{P'}^\bbQ(\supp (\CT_Q f))$. 
From this we deduce that 
$(\Eis^{M'}_{M'_1} \circ \CT_Q)(f) \in \C_{P'^-,+}$. 
Therefore 
\[ (R_{P':P'^-} \circ \Eis^{M'}_{M'_1}\circ \CT_Q)(f)\in \C_{P',-}. \] 
Consequently, $\CT_{P'}(L' f) \in \C_{P',-}$ and $L'$ defines
an operator $\eA_\psc \to \eA_c$.

\medskip

Now we check that $L'$ is inverse to $L$. 
For $f\in \eA_\psc$, it follows from \eqref{e:CTLinv} and the ensuing discussion that 
\[ LL'f = \sum_{P'} \sum_{P'_1 \subset P'} (-1)^{\dim Z(M') - \dim Z(M'_1)} 
(\Eis_{P'^-} \circ R_{P':P'^-}^{-1} \circ R_{P':P'^-} \circ \Eis^{M'}_{M'_1} 
\circ \CT_{(M'\cap P'_1)U'^-}) (f). \]
Observe that 
$\Eis_{P'^-} \circ \Eis^{M'}_{M'_1} = \Eis_{(M'\cap P'_1)U'^-}$. 
Set $M'_2 := w_0^{M'} M'_1 w_0^{M'}$ and let $P'_2$ be the corresponding standard parabolic subgroup. Then 
$\Eis_{(M'\cap P'_1)U'^-} \circ \CT_{(M'\cap P'_1)U'^-} 
= \Eis_{P'^-_2} \circ \CT_{P'^-_2}$.
Hence
\[ LL'f = \sum_{P'_2} \left(\sum_{P'_2\subset P'} (-1)^{\dim Z(M')-\dim Z(M'_2)}\right)
(\Eis_{P'^-_2} \circ \CT_{P'^-_2})(f). \]
The inner sum vanishes unless $P'_2=G$. Therefore $LL'f = f$.

For $f \in \eA_c$, we apply \eqref{e:CTL=LCT} to get 
\[ L'Lf = \sum_P \sum_{M_1 \subset M} (-1)^{\dim Z(M)-\dim Z(M_1)} (\Eis_{(M \cap P_1^-)U} \circ \CT'^M_{M \cap P_1^-} 
\circ \CT'_P)(f), \]
where $P=MU$ ranges over the standard parabolic subgroups, $M_1$ ranges over Levi subgroups of $M$, and $P_1 \subset G$ is the standard parabolic subgroup with Levi $M_1$.
Observe that $\CT'^M_{M \cap P_1^-} \circ \CT'_P = \CT'_{(M\cap P_1^-)U}$.
Conjugating by $w_0^M$, one sees that 
\[ \Eis_{(M\cap P_1^-)U} \circ \CT'_{(M\cap P_1^-)U} 
= \Eis_{P_2} \circ \CT'_{P_2}, \] where
$P_2$ is the standard parabolic subgroup with Levi 
$M_2:=w_0^M M_1 w_0^M$. 
Then 
\[ L'Lf= \sum_{P_2} \left(\sum_{P_2\subset P} (-1)^{\dim Z(M)-\dim Z(M_2)} \right) (\Eis_{P_2} \circ \CT'_{P_2})(f), \]
and the inner sum vanishes unless $P_2=G$. 
Therefore $L'Lf = f$, and we have proved that $L'$ is the inverse of $L$. 
\end{proof}

\begin{rem} \label{rem:Aubert}
% aubert involution
We observe that formula \eqref{e:Linverse} for $L^{-1}$ may be thought of as an analog of the Aubert--Zelevinsky involution for smooth representatives of a $\fp$-adic group: 
Let $F_v$ denote a non-Archimedean local field. For every smooth $G(F_v)$-module $M$
one can form a complex 
\[ 0 \to M \to \bigoplus_P i^G_P r^G_P (M) \to \dotsb \to i^G_B r^G_B (M) \to 0 \]
where $i^G_P, r^G_P$ denote, respectively, the parabolic induction and Jacquet functors, and the sum in the $i$-th term runs over standard parabolic subgroups of corank $i$ in $G$. 
We call this complex the Deligne--Lusztig complex associated to $M$ and denote it by 
$\on{DL}(M)$. 
Analogous complexes were considered in \cite{DL} for representations of a finite Chevalley group. 
In the Grothendieck group, we have 
\[ [\on{DL}(M)] = \sum_P (-1)^{\dim Z(G) - \dim Z(M)} [ i^G_P r^G_P(M) ] \]
where the sum ranges over standard parabolic subgroups. 
In fact, this defines an involution of the Grothendieck group, which is often called the Aubert--Zelevinsky involution.
If one considers $\Eis_P, \CT_P$ as global analogs of $i^G_P, r^G_P$, respectively, 
then formula \eqref{e:Linverse} suggests that $L^{-1}$ is a global analog of 
the Aubert--Zelevinsky involution (although $L^{-1}$ is no longer an involution). 
\end{rem}

\begin{rem}
The main result of \cite{G:miraculous} (namely, Theorem 0.1.6) says that 
the stack $\Bun_G$ is miraculous, i.e., the functor $\on{Ps-Id}_{\Bun_G,!} : \on{D-mod}(\Bun_G)_{\on{co}} \to \on{D-mod}(\Bun_G)$ is an equivalence. 
This equivalence is a D-module analog of the part of Theorem~\ref{thm:L^-1} 
that says that the operator $L^K : \eA_c^K\to \eA^K$ induces an isomorphism $\eA_c^K \to \eA_\psc^K$.

The formula \eqref{e:Linverse} for $L^{-1}$ may be useful in describing the functor inverse to $\on{Ps-Id}_{\Bun_G,!}$ (we expect that one can mimic the construction of the Deligne--Lusztig complex using the functors $\Eis^{\on{enh}}_P, \CT^{\on{enh}}_P$ defined in \cite[\S 6]{G:outline}). 
We refer the reader to \cite[Conjecture C.2.1]{DW} for an explicit conjecture in the case $G = \SL(2)$.
\end{rem}

\appendix

\section{Substacks of the Hecke stack} \label{s:appendixHecke}
To keep the notation consistent throughout this paper, in this section $M$ denotes 
an arbitrary connected split reductive group over a perfect field $k$. 

We will attach to any algebraic normal irreducible monoid $\tilde M$
with group of units $M$ a substack of a symmetrized version of the Hecke stack. 
This substack is the global model
for the formal arc space of the embedding $M \into \tilde M$, and it was
also considered in \cite[\S 2]{BNY}.
We are particularly interested in the case when 
$M$ is the Levi factor of a parabolic subgroup of $G$ and $\tilde M= \wbar M$
is the closure of the $M$-orbit of the coset $U$ in $\wbar{G/U}$. 
The monoid $\wbar M$ is studied in detail in \cite{Wa}. 

We recall the relation between the Hecke stack and the Beilinson--Drinfeld Grassmannian. We use a symmetrized factorizable version of the affine Grassmannian. A detailed exposition on the factorizable version of the affine Grassmannian over the Ran space can be found in \cite{Zhu}.

\subsection{Recollections on normal reductive monoids}\label{sect:tildeM-mor}
Let $\tilde M$ be an algebraic normal irreducible monoid with group of units $M$. 

\subsubsection{}
Fix a maximal torus $T \subset M$. 
The \emph{Renner cone} $\check C \subset \check\Lambda^\bbQ$ of $\tilde M$ is the rational 
convex cone corresponding by \cite{KKMS} to the closure of $T$ in $\tilde M$ after base changing to an algebraic closure of $k$. 
This cone is stable under the actions of $W_M$ and $\Gal(\bar k/k)$. 
The Renner cone is canonical and only depends on the abstract Cartan 
of $M$ (which identifies with $T$ after choosing a Borel subgroup). 
L.~Renner showed in \cite[Theorem 5.4]{Renner} that algebraic normal irreducible monoids with group of units $M$ bijectively correspond (via the Renner cone) to 
convex rational polyhedral cones generating $\check\Lambda^\bbQ$ as a vector space and stable under the actions of $W_M$ and $\Gal(\bar k/k)$.

\subsubsection{}
Since $M$ is scheme-theoretically dense in $\tilde M$, the restriction functor 
\[\Rep(\tilde M) \to \Rep(M)\] 
is fully faithful, so we
may consider $\Rep(\tilde M)$ as a full subcategory of $\Rep(M)$.

\subsubsection{}  
We will consider the algebraic stack $M \bs \tilde M / M$ which sends a test scheme $S$
to the groupoid of pairs of $M$-bundles $\eF^1_M, \eF^2_M$ on $S$ equipped with a 
section 
\[ \beta_M : S \to \tilde M \xt^{M \xt M} (\eF^1_M \xt_S \eF^2_M).\] 
Such a section $\beta_M$ will be called an $\tilde M$-morphism from $\eF^2_M$ to $\eF^1_M$. 
By the Tannakian formalism, giving an $\tilde M$-morphism $\beta_M$ is the same as 
giving a collection of assignments 
\[ V \in \Rep(\tilde M) \rightsquigarrow \beta_M^{V} : V_{\eF_M^2} \to V_{\eF_M^1} \]
where $\beta_M^{V}$ is $\eO_S$-linear, and the Pl\"ucker relations hold.
This means that for $V$ being the trivial representation, 
$\beta_M^{V}$ is the identity map $\eO_S \to \eO_S$, and for an 
$M$-morphism $V_1 \ot V_2 \to V_3$, the diagram 
\[ \xymatrix{ (V_1 \ot V_2)_{\eF_M^2} \ar[rr]^{\beta_M^{V_1} 
\ot \beta_M^{V_2} } \ar[d] && (V_1 \ot V_2)_{\eF_M^1}  \ar[d] \\ 
(V_3)_{\eF_M^2} \ar[rr]^{\beta_M^{V_3}} && (V_3)_{\eF_M^1} } \]
commutes. Observe that the Pl\"ucker relations imply that the assignments $\beta_M^{V}$ are 
functorial in $V$.

\subsubsection{} 
Observe that the fiber bundle $M \xt^{M \xt M} (\eF^1_M \xt_S \eF^2_M)$
canonically identifies with the scheme of isomorphisms $\Isom(\eF^2_M, \eF^1_M)$
over $S$. In other words, an $M$-morphism is an isomorphism of $M$-bundles. 
Therefore the stack $M \bs \tilde M / M$ contains the open substack
$M \bs M / M$, which canonically identifies with 
the classifying stack $\bbB M = M\bs{\on{pt}}$ of $M$.

\subsection{Definition of $\tilde\eH_M$} \label{ss:tH_M}
Let $X$ be a smooth projective geometrically connected curve over 
a field $k$. We consider the mapping stack $\Maps(X, M \bs \tilde M / M)$
whose value on a test scheme $S$ is the groupoid of all maps
$X \xt S \to M \bs \tilde M / M$. Such a map is said to be non-degenerate
if the preimage of $M \bs M / M \subset M \bs \tilde M / M$ is a subset
of $X \xt S$ whose projection on $S$ is surjective. We will denote 
\[ \tilde\eH_M := \Maps^\circ(X, M \bs \tilde M / M) \subset \Maps(X, M \bs \tilde M / M) \]
the open substack consisting of non-degenerate maps. 
A map $X \xt S \to M \bs \tilde M / M$ is the datum of a pair of $M$-bundles $\eF^1_M, \eF^2_M$ 
on $X \xt S$ and an $\tilde M$-morphism $\beta_M : \eF^2_M \to \eF^1_M$ between them. 
In the Tannakian language, $\beta_M$ is non-degenerate
if and only if for every geometric point $s\to S$, the restriction 
of $\beta_M^{V}$ to the fiber $X \xt s$ is generically an isomorphism.
The last condition is 
equivalent to requiring that $\beta_M^{V}$ is an embedding of coherent sheaves
such that the quotient is $S$-flat.
% local criterion of flatness, Mastumura comm alg 2ed (20.E) 

\subsubsection{Relation to the full Hecke stack} \label{sss:HDiv}
The projection $M \to M/[M,M]$ induces an inclusion $\check \Lambda_{M/[M,M]} \subset
\check\Lambda$. 
Recall that $\check C$ denotes the Renner cone of $\tilde M$, which is a $W_M$-stable
convex polyhedral cone that generates $\check\Lambda^\bbQ$ as a group.
If $M/[M,M]$ is not finite (i.e., $M$ is not semisimple), 
then there exists a character $\check\lambda \in \check\Lambda$ that lies
in the interior of the cone $\check C \cap \check\Lambda_{M/[M,M]}^\bbQ$. 
If $M$ is semisimple, we put $\check\lambda=0$ (in this case $\tilde M$ must equal $M$).

Considering $\check\lambda$ as a homomorphism $\tilde M \to \bbA^1$, 
the open subscheme $\check\lambda^{-1}(\bbG_m)$ coincides with $M \subset \tilde M$.
The closed subscheme $\check\lambda^{-1}(0)$ has the same reduced scheme structure
as $(\tilde M - M)_\red$.

Let $\Div_+$ denote the scheme of relative effective divisors of $X$. 
The map $\check\lambda: \tilde M \to \bbA^1$ induces a map
\begin{equation}\label{e:He-Div}
    \tilde\eH_M \to \Div_+. 
\end{equation}
More explicitly, an $S$-point $X \xt S \to M \bs \tilde M / M$ 
is sent to the preimage of $M \bs \check\lambda^{-1}(0) / M$, which 
is a relative effective divisor of $X \xt S$. 
% stacks project tag 062Y also see Quot scheme

\subsubsection{} 
Define the stack $\Hecke(M)_{\Div_+}$ as follows: its $S$-points are 
quadruples $(D, \eF^1_M, \eF^2_M, \beta_M)$ where $D$ is a relative effective divisor
on $X \xt S$, $\eF^1_M$ and $\eF^2_M$ are two $M$-bundles on $X \xt S$, and $\beta_M$ 
is an isomorphism of $M$-bundles on the restrictions
\[ \eF^2_M|_{X\xt S - D} \cong \eF^1_M|_{X \xt S - D}. \]
From this definition it is evident that we have a closed embedding of stacks 
\[ \tilde\eH_M \into \Hecke(M)_{\Div_+}. \]

\begin{rem}
The above inclusion slightly depends on the choice of $\check\lambda$, which
is used to define \eqref{e:He-Div}. This choice goes away if we consider 
Ran versions of the corresponding Hecke stacks, since a point of $\Hecke(M)_{\Div_+}$ only depends on the reduced structure $D_\red$ of the divisor $D$.
\end{rem}

\subsubsection{}
We let $\overset\leftarrow h, \overset \rightarrow h$ denote 
the two forgetful maps $\tilde \eH_M \to \Bun_M$. We use the convention where
$\overset \leftarrow h$ is the map corresponding to $\eF^1_M$.

\begin{prop} \label{prop:heckefin}
The morphism $\tilde\eH_M \to \Bun_M \xt \Bun_M$ is schematic, quasi-affine, and of
finite presentation. 
\end{prop}
\begin{proof}
Consider a test scheme $S$ and a map $S \to\Bun_M \xt \Bun_M$ corresponding to 
$M$-bundles $\eF_M^1, \eF_M^2$ on $X \xt S$. 
Then the fiber product $\ds \tilde\eH_M \xt_{\Bun_M \xt\Bun_M} S$ is isomorphic to 
an open subspace of the space of sections of the map
\[ \tilde M \xt^{M \xt M} (\eF^1_M \xt_{X \xt S} \eF^2_M)\to X \xt S. \] 
This map is affine and of finite presentation because $\tilde M$ is, which implies that the
space of sections is representable by a scheme affine and 
of finite presentation over $S$.
% gaitsgory 2009 seminar, BunG notes Sep 17
\end{proof}

\subsubsection{}
The monoid structure on $\tilde M$ allows us to compose two $\tilde M$-morphisms
$\eF_M^2 \to \eF_M^1$ and $\eF_M^3 \to \eF_M^2$ to get an $\tilde M$-morphism 
$\eF_M^3 \to \eF_M^1$. The composition of two non-degenerate morphisms
is non-degenerate. This defines a map
\[ \comp: \tilde\eH_M \xt_{\overset\rightarrow h,\Bun_M, \overset\leftarrow h} \tilde\eH_M \to \tilde\eH_M. \]

\begin{rem} \label{rem:compproper}
The map \eqref{e:He-Div} takes the composition of $\tilde M$-morphisms into 
the sum of effective divisors (which is a proper map $\Div_+ \xt \Div_+ \to \Div_+$). 
Hence Proposition~\ref{prop:Hproper} below implies that $\comp$ is a proper map.
\end{rem}

\begin{lem}
Let $\eF_M \in \Bun_M(S)$ for a $k$-scheme $S$. 
Any non-degenerate $\tilde M$-morphism $\beta_M : \eF_M \to \eF_M$
is an $M$-bundle automorphism.
\end{lem}
\begin{proof}
Let $V \in \Rep(\tilde M)$. Then $\beta_M^{V}:V_{\eF_M} \to V_{\eF_M}$
is generically an isomorphism on geometric fibers of $X \xt S$. 
The same is true for $\det(\beta_M^{V})$, and $\Gamma(X,\eO_X)=k$
implies that $\beta_M^{V}$ is an isomorphism. Since $M$ is the
group of units of $\tilde M$, we conclude that $\beta_M$ 
is an automorphism of $\eF_M$.
\end{proof}

\begin{cor}
Suppose there exist $\tilde M$-morphisms $\beta_M : \eF^2_M \to \eF^1_M$
and $\beta'_M : \eF^1_M \to \eF^2_M$. 
Then $\beta_M, \beta'_M$ are both isomorphisms.
\end{cor}

\subsection{Affine Grassmannians} 
In this section, we explain the relation between the Hecke stack and the 
Beilinson--Drinfeld affine Grassmannian. We refer the reader to \cite{Zhu} 
for a far more complete treatment of the affine Grassmannian.

For this purpose we introduce the divisor version of the jet group $M(\fo)$. 
For $D \in \Div_+(S)$, let $\hat D$ denote
the formal completion of $D$ in $X\xt S$ (which is a formal scheme). 
Define 
\[ M(\fo)_{\Div_+} = \{ (D, \gamma) \mid D \in \Div_+(S),\, \gamma \in M(\hat D) \},
\]
which is representable by a scheme affine over $\Div_+$ (cf.~\cite[Proposition 3.1.6]{Zhu}).

\subsubsection{} We define the symmetrized version of the Beilinson-Drinfeld
affine Grassmannian as 
\[ \Gr_{M,\Div_+} := \spec(k) \xt_{\eF^0_M,\Bun_M, \overset\leftarrow h} \Hecke(M)_{\Div_+} \]
where $\eF^0_M \in \Bun_M(k)$ is the trivial bundle.

\subsubsection{} \label{sss:heckeGr}
The Hecke stack $\Hecke(M)_{\Div_+}$ can be regarded as a 
twisted product $\Bun_M \tilde\xt \Gr_{M,\Div_+}$.
More precisely, consider the stack
\[ \eY = 
\{ (D, \eF_M^1, \tilde\gamma) \mid D \in \Div_+,\, \eF_M^1\in \Bun_M,\, 
\tilde\gamma: \eF_M^0|_{\hat D} \cong \eF_M^1|_{\hat D} \}. \]
Then $\eY \to \Div_+\xt \Bun_M$ is a $M(\fo)_{\Div_+}$-torsor.
Observe that the group scheme $M(\fo)_{\Div_+}$ also acts on $\Gr_{M,\Div_+}$ over $\Div_+$. We have a canonical isomorphism
\[
    \Hecke(M)_{\Div_+} \cong \eY \xt^{M(\fo)_{\Div_+}} \Gr_{M,\Div_+}, 
\]
where on the right hand side we take the fiber product $\eY \xt_{\Div_+} \Gr_{M,\Div_+}$
and quotient by the anti-diagonal action of $M(\fo)_{\Div_+}$.

\begin{prop} \label{prop:Hproper}
The morphism $\tilde\eH_M \to \Bun_M \xt \Div_+$ is proper, 
where $\tilde\eH_M$ maps to $\Bun_M$ by either $\overset\leftarrow h$ or $\overset\rightarrow h$. 
\end{prop}
\begin{proof} Without loss of generality we will consider the projection 
$\overset\leftarrow h: \tilde \eH_M \to \Bun_M$.
Let $\wtilde \Gr_{M,\Div_+} := \spec(k) \xt_{\eF^0_M,\Bun_M} \tilde\eH_M$,
which is a closed subspace of $\Gr_{M,\Div_+}$. 
Choose a uniformizer $\varpi_v \in \fo_v$ at a place $v$. Let 
$C \subset \Lambda^\bbQ$ denote the dual of the Renner cone of $\tilde M$.
Over a divisor $n_v \cdot v$ supported at a single point, the fiber of $\wtilde\Gr_{M,\Div_+}$ is equal to the union of the orbits 
\[ M(\fo_v) \cdot \theta_v(\varpi_v) \subset (\wbar M(\fo_v) \cap M(F_v))/M(\fo_v) \] for $\theta_v \in C \cap \Lambda^+_M$ satisfying $\brac{\check\lambda, \theta_v} = n_v$, where $\check\lambda$ is the character chosen in \S\ref{sss:HDiv}. Since $\check\lambda$ lies in the interior of $\check C \cap \check\Lambda^\bbQ_{M/[M,M]}$, there are only finitely many $\theta_v$ such that $\brac{\check\lambda,\theta_v} = n_v$. 
We deduce, by factorization, that $\wtilde\Gr_{M,\Div_+}$ is representable by 
a scheme of finite type over $\Div_+$.
It is known that $\Gr_{M,\Div_+}$ is ind-proper over $\Div_+$ (cf.~\cite[Remark 3.1.4]{Zhu}). Thus $\wtilde\Gr_{M,\Div_+}$ is proper over $\Div_+$.
The Proposition follows by considering $\tilde\eH_M$ as
a twisted product $\Bun_M \tilde\xt \wtilde\Gr_{M,\Div_+}$ as explained in  
\S\ref{sss:heckeGr}.
\end{proof}

\subsection{Slope comparisons}  \label{ss:slope}
Let $\pi_1(M)$ denote the quotient of $\Lambda$ by the subgroup generated by 
coroots of $M$. It is well-known that there is a bijection $\deg_M : \pi_0(\Bun_M) \simeq \pi_1(M)$. Note that $\pi_1(M) \ot \bbQ = \Lambda_{M/[M,M]}^\bbQ = \Lambda^\bbQ_{Z_0(M)}$. We call the composition 
\[ \Bun_M \to \pi_1(M) \to \Lambda^\bbQ_{Z_0(M)} \] 
the \emph{slope} map. Its fibers are not necessary connected but have finitely many connected components. 
The slope map coincides with the composition 
\[ \Bun_M \to \Bun_{M/[M,M]} \to 
\pi_0(\Bun_{M/[M,M]}) = \Lambda_{M/[M,M]} \subset \Lambda^\bbQ_{Z_0(M)}.\]

Following the notation of \cite[Theorem 7.4.3]{DG:CG}, 
let $\Bun^{(\lambda)}_M,\, \lambda \in \Lambda^{+,\bbQ}_M$ denote the 
quasi-compact locally closed reduced substack of $M$-bundles with Harder-Narasimhan coweight $\lambda$. 
We refer the reader to \cite[\S 7]{DG:CG} and \cite{Schieder:HN} for 
statements and proofs of the main results of reduction theory for a general reductive group.

\begin{lem} \label{lem:heckeHN}
	Suppose $\eF^1_M \in \Bun^{(\lambda_1)}_M(k),\, \eF^2_M \in \Bun^{(\lambda_2)}_M(k)$
for $\lambda_1,\lambda_2 \in \Lambda^{+,\bbQ}_M$, and there exists
an isomorphism $\beta_M : \eF^2_M|_{X-D} \to \eF^1_M|_{X-D}$
for a divisor $D \subset X$.
At each closed point $v \in D$, the restriction of $\beta_M$ to $F_v$ 
determines a coweight $\lambda_v \in \Lambda^+_M = K_{M,v} \bs M(F_v)/ K_{M,v}$.
Then 
\begin{equation} \label{heckeineq}
	w_0^M \sum_{v} n_v\cdot \lambda_v \le_M^\bbQ \lambda_1 - \lambda_2 \le_M^\bbQ \sum_{v} n_v\cdot \lambda_v,
\end{equation}
where $n_v = \dim_k(k_v)$ is the dimension of the residue field of $v$.
\end{lem}
\begin{proof}
Consider the Harder--Narasimhan flag of $\eF^1_M$: this is a canonical
reduction of $\eF^1_M$ to a $Q$-bundle $\eF^1_Q$, where $Q\subset M$ is the
parabolic subgroup corresponding to the Harder--Narasimhan coweight $\lambda_1$.
Recall that a reduction of $\eF^2_M$ to $Q$ is the same as a section
of the proper map $\eF^2_M/Q:=\eF^2_M \xt^M M/Q \to X$. The reduction $\eF^1_Q$ 
and the isomorphism $\beta_M$ determine a section $X-D \to \eF^2_M/Q$. 
By properness, this extends to a reduction $\eF^2_Q$ of $\eF^2_M$
with an isomorphism \[ \beta_Q : \eF^2_Q|_{X-D} \to \eF^1_Q|_{X-D} \]
inducing $\beta_M$.
Let $L$ denote the Levi quotient of $Q$, and 
let $\eF^1_L, \eF^2_L, \beta_L$ denote the corresponding induced objects. 
Then the restriction of $\beta_L$ to $\spec F_v$ for $v\in D$ determines a coweight 
$\nu_v \in \Lambda^+_L$
via the quotient map $Q \to L$. Since $\nu_v$ and $\lambda_v$ are both
induced by $\beta_Q$, we see that $\nu_v(\varpi_v) U_Q(F_v) \cap 
K_{M,v} \lambda_v(\varpi_v) K_{M,v} \ne \emptyset$ where $\varpi_v \in \fo_v$ is
a uniformizer. This implies that $\nu_v$ is contained in the convex hull
of $W_M \cdot \lambda_v$ (cf.~\cite[p.~148]{Cartier}).

Let $\nu_2 \in \Lambda^\bbQ_{L/[L,L]}$ be the slope of $\eF^2_L$.
Then $\lambda_1 - \nu_2$ is equal to the 
image of $\sum_{v\in D} n_v\cdot \nu_v$ under the projection $\Lambda^\bbQ \to \Lambda_{L/[L,L]}^\bbQ = \Lambda_{Z_0(L)}^\bbQ$.
In particular, $\lambda_1 - \nu_2$ lies in the convex hull of the $W_M$-orbit of
$\sum_{v\in D} n_v \cdot\nu_v$, so $\lambda_1 - \nu_2 \le_M^\bbQ \sum n_v\cdot \lambda_v$. 
By the comparison theorem \cite[Theorem 4.5.1]{Schieder:HN}, 
$\nu_2 \le_M^\bbQ \lambda_2$. 
We have shown the second inequality in \eqref{heckeineq}. 
Switching $\eF^1_M,\eF^2_M$ and considering $\beta_M^{-1}$ proves
the first inequality by symmetry.
\end{proof}

\begin{cor} \label{cor:heckeHN}
Suppose $\eF^1_M \in \Bun^{(\lambda_1)}_M(k), \eF^2_M \in \Bun^{(\lambda_2)}_M(k)$
for $\lambda_1,\lambda_2 \in \Lambda^{+,\bbQ}_M$, and there exists
a non-degenerate $\tilde M$-morphism $\beta_M : \eF^2_M \to \eF^1_M$.
Let $C \subset \Lambda^\bbQ$ be the dual of the Renner cone %$\check C \subset \check\Lambda^\bbQ$ 
of $\tilde M$. Then $\lambda_1 - \lambda_2$ belongs to the rational convex cone 
generated by $C$ and $\Lambda^{\pos,\bbQ}_M$.
\end{cor}
\begin{proof}
Since $\beta_M$ is non-degenerate, there exists a divisor $D \subset X$
such that $\beta_M|_{X-D}$ is an isomorphism. For $v \in D$, the
coweight $\lambda_v$ defined in Lemma~\ref{lem:heckeHN} corresponds to
an element in $K_{M,v} \bs (M(F_v) \cap \tilde M(\fo_v)) / K_{M,v}$. 
By the classification of double orbits of $M(F_v)\cap\tilde M(\fo_v)$, the latter set identifies
with $C \cap \Lambda^+_M$. Therefore 
the $W_M$-stable cone $C$ contains $w_0^M\sum n_v\cdot \lambda_v$. 
The corollary follows from Lemma~\ref{lem:heckeHN}.
\end{proof}

\subsection{The stack $\eH^+_M$} \label{sect:H_M}
Let $P$ be a standard parabolic subgroup of $G$ with Levi factor $M$. 
We recall that $G/U$ is quasi-affine. Let $\wbar{G/U}$ denote the affine closure.
We embed $M \into G/U$ by $m \mapsto mU$ and define $\wbar M$ to be
the closure of $M$ in $\wbar{G/U}$. 
Then $\wbar M$ is a monoid acting on $\wbar{G/U}$ (cf.~\cite{Wa}).
\smallskip

We now specialize the above discussion of Hecke stacks to the case 
$\tilde M = \wbar M$. Let 
\begin{equation} 
	\eH^+_M = \Maps^\circ(X, M \bs \wbar M / M ) 
\end{equation} 
denote the stack studied above.

\begin{rem} \label{rem:heckeHN}
By \cite[Lemma 3.1.4]{Wa}, the dual of the Renner cone of $\wbar M$ equals $\Lambda^{\pos,\bbQ}_{U}$. 
Therefore if we are in the setting of Corollary~\ref{cor:heckeHN}, 
the rational cone generated by $\Lambda^{\pos,\bbQ}_{U}$ and $\Lambda^{\pos,\bbQ}_M$
is $\Lambda^{\pos,\bbQ}_G$, and the corollary implies that $\lambda_1 - \lambda_2 \in \Lambda^{\pos,\bbQ}_G$, i.e., $\lambda_2 \le_G^\bbQ \lambda_1$.

The rational cone generated by $\Lambda^{\pos,\bbQ}_U$ and $-\Lambda^{\pos,\bbQ}_M$
is $w_0^M\Lambda^{\pos,\bbQ}_G$, and Lemma~\ref{lem:heckeHN} also implies that
$\lambda_1-\lambda_2 \in w_0^M \Lambda^{\pos,\bbQ}_G$.

%don't the above statements imply $\lambda_1 - \lambda_2 \in \Lambda^{\pos,\bbQ}_U$? YES because if \lambda in Lambda^pos_U and \mu \le_M^\bbQ \lambda is in Lambda^pos_G, we have \mu in Lambda^pos_U
\end{rem}

\subsubsection{Remarks on $G$} We assume for the rest of this Appendix 
that $G$ has a simply connected derived group $[G,G]$. The reader may refer to 
\cite[\S 7]{Schieder:HN} for how to remove this hypothesis, and the relevant geometry remains the same.

\subsubsection{The graded Ran space}
Let $\Lambda_{G,P} := \pi_1(M)$ denote the quotient of $\Lambda$ by the subgroup generated by the coroots of $M$.  
We have a natural projection $\Lambda \to \Lambda_{G,P}$. 
Let $\Lambda^\pos_{G,P}$ denote the submonoid 
of $\Lambda_{G,P}$ generated by the image of the positive coroots of $G$. 

Any $\theta \in \Lambda^\pos_{G,P}$ can be uniquely written as 
a sum $\sum n_j \alpha_j$ for $j \in \Gamma_G - \Gamma_M$. 
Let $X^\theta$ denote $\prod X^{(n_j)}$. 
Define $\Ran(X,\Lambda^\pos_{G,P})$ to be the disjoint union of $X^\theta$ for $\theta \in \Lambda^\pos_{G,P}$. Here we are using the notation of \cite{Whatacts}, but 
we include $0 \in \Lambda^\pos_{G,P}$ with $X^0 = \spec(k)$.
We can regard $\Ran(X,\Lambda^\pos_{G,P})$ as the scheme of 
$\Lambda^\pos_{G,P}$-colored divisors on $X$, which is a $\Lambda^\pos_{G,P}$-graded
version of the Ran space. The grading allows us to use the language of
factorization algebras graded by a monoid introduced in \cite[\S 2]{Whatacts}, which 
is slightly simpler than the more general set-up of factorization algebras from \cite{CHA} (the difference is that we can replace the Ran space with genuine schemes).

\subsubsection{} 
Let us review the construction of the map $\eH^+_M \to \Ran(X,\Lambda^\pos_{G,P})$, following \cite[\S 3.1.7]{Schieder:gen}. 
Recall from \cite{BG} that the quotient $G/[P,P]$ is strongly quasi-affine,
and let $\wbar{G/[P,P]}$ denote its affine closure. 
Let $\wbar{M/[M,M]}$ denote the closure of $M/[M,M]$ in 
$\wbar{G/[P,P]}$ under the natural embedding 
\[ M/[M,M] = P/[P,P] \into G/[P,P] \subset \wbar{G/[P,P]}. \]
The projection $G/U \to G/[P,P]$ extends to a map of affine closures 
$\wbar{G/U} \to \wbar{G/[P,P]}$, and therefore the 
projection $M \to M/[M,M]$ extends to a map $\wbar M \to \wbar{M/[M,M]}$. 
This induces a map of stacks
\begin{equation}\label{e:HRan}
    \eH^+_M \to \Maps^\circ(X, \wbar{M/[M,M]} / (M/[M,M]) ). 
\end{equation}
Consider $\check\Lambda_{M/[M,M]}$ as a sub-lattice of $\check \Lambda_M$.
Then one can check that $k[\wbar{M/[M,M]}]$ has a basis 
consisting of the characters in the submonoid
$\check\Lambda^+_G \cap \check \Lambda_{M/[M,M]}$. 
Since $[G,G]$ is assumed to be simply connected, 
$\Lambda^\pos_{G,P}$ is the monoid dual to $\check\Lambda^+_G \cap \check\Lambda_{M/[M,M]}$. 
We deduce that the right hand side of \eqref{e:HRan} is isomorphic 
to the scheme $\Ran(X,\Lambda^\pos_{G,P})$.  
Thus \eqref{e:HRan} becomes a map of stacks 
\[ \eH^+_M \to \Ran(X,\Lambda^\pos_{G,P}). \] 
For $\theta \in \Lambda^\pos_{G,P}$, denote the preimage of the
connected component $X^\theta \subset \Ran(X,\Lambda^\pos_{G,P})$ by $\eH^+_{M,X^\theta}$.

\begin{rem} We defined the map \eqref{e:HRan} ``group-theoretically'' following \cite[\S 3.1.7]{Schieder:gen}. One can also define this map using the Tannakian formalism, which is essentially done in \cite{BG, BFGM}.
\end{rem}

\subsubsection{}
We use the character $2\check\rho_P = 2\check\rho - 2\check\rho_M \in \check\Lambda_{M/[M,M]}$ to define 
the map \eqref{e:He-Div}, which is then equal to the composition 
\[ \eH^+_M \to \Ran(X,\Lambda^\pos_{G,P}) \to \Div_+,\] 
where the last map sends $X^\theta$ to the symmetric power $X^{(2\abs\theta)} \subset \Div_+$ for $\abs\theta := \brac{\check\rho_P,\theta}$.

As in \S\ref{sss:heckeGr}, we can express $\eH^+_{X^\theta}$ as a twisted product
\begin{equation}\label{e:HeGr+}
    \eH^+_{M,X^\theta} \cong \Bun_M \tilde\xt \Gr^+_{M,X^\theta}
\end{equation}
where 
$\Gr^+_{M,X^\theta} := \spec(k) \xt_{\eF^0_M,\Bun_M,\overset\leftarrow h} \eH^+_{M,X^\theta}$,
using the action of the jet group $M(\fo)_{X^{(\abs\theta)}}$.
We will always consider the twisted product with respect to the projection $\overset\leftarrow h$.

\subsubsection{} \label{sect:factorHecke}
By a partition $\fA$ of $\theta$ we mean a decomposition
\[ \theta = \sum_{\lambda \in \Lambda^\pos_{G,P} - 0}  n_\lambda \cdot \lambda, \, n_\lambda \in \bbZ_+. \]
Let $\cP_\theta$ denote the set of all partitions of $\theta$. 
For a partition $\fA \in \cP_\theta$, let $X^\fA := \prod_\lambda X^{(n_\lambda)}$. 
This is a scheme of dimension $\abs \fA := \sum n_\lambda$. 
Note that there is a natural map $X^\fA \to X^\theta$.

Let $(X^\fA)_\disj \subset X^\fA$ denote the open subscheme with all diagonals removed. 
A $k$-point $x^\fA \in (X^\fA)_\disj$ is a formal sum 
$\sum_{v \in \abs X} \theta_v \cdot v$ 
for $\theta_v \in \Lambda^\pos_{G,P}$, 
such that for each $\lambda \in \Lambda^\pos_{G,P}$, we have 
$n_\lambda = \sum_{v \mid \theta_v = \lambda} \deg(v)$. 
The composition $(X^\fA)_\disj \into X^\fA \to X^\theta$ is a locally closed embedding, and 
the subschemes $(X^\fA)_\disj$ for $\fA \in \cP_\theta$ form a stratification of $X^\theta$.
Thus we can stratify $\eH^+_{M,X^\theta}$ by the substacks 
\[ \eH^{+,\fA}_M := \eH^+_{M,X^\theta} \xt_{X^\theta} (X^\fA)_\disj. \]
This stack $\eH^{+,\fA}_M$ is the same as the one defined in \cite[\S 1.8]{BFGM}.

The diagonal $X \to X^\theta$ corresponds to the trivial partition $\theta = 1\cdot \theta$, 
and we denote by $\eH^{+,\theta}_M$ (resp.~$\Gr^{+,\theta}_M$) 
the stack $\ds \eH^+_{M,X^\theta} \xt_{X^\theta} X$ (resp.~the scheme $\ds \Gr^+_{M,X^\theta} \xt_{X^\theta} X$).

\subsubsection{} 
Let $\fA \in \cP_\theta$. The stack $\eH^{+,\fA}_M$ is fibered
over $(X^\fA)_\disj \xt \Bun_M$ with respect to $\overset\leftarrow h$. 
Similar to \S\ref{sss:heckeGr}, one can express $\eH^{+,\fA}_M$ as a 
twisted product 
\begin{equation}\label{e:H+Gr}
    \eH^{+,\fA}_M \cong \Bun_M \tilde\xt \Gr^{+,\fA}_M, 
\end{equation}
where $\Gr^{+,\fA}_M := \Gr^+_{M,X^\theta} \xt_{X^\theta} (X^\fA)_\disj$. 
To define the twisted product one considers the action of the jet group 
$M(\fo)_{X^{(\abs \fA)}}$ on $\Gr_M^{+,\fA}$ over $X^{(\abs \fA)}$.
The embedding 
\begin{equation} \label{e:H+strat}
    \eH^{+,\fA}_M \into \eH^+_{M,X^\theta}
\end{equation}
lies over the map of symmetric powers $X^{(2\abs \fA)} \to
X^{(2\abs\theta)}$. The latter map induces a map of
jet groups $M(\fo)_{X^{(2\abs \fA)}}\to M(\fo)_{X^{(2\abs\theta)}}$.
Therefore \eqref{e:H+strat} can be thought of as a twisted product
of $\id_{\Bun_M}$ with the embedding 
$\Gr_M^{+,\fA} \into \Gr_{M,X^\theta}^+$, 
which is equivariant with respect to the actions of the corresponding jet groups.

\subsubsection{}
Let $x^\fA \in (X^\fA)_\disj(k)$
be a $\Lambda^\pos_{G,P}$-colored divisor $\sum \theta_v \cdot v$,
and let $\eF_M^1$ be a $k$-point of $\Bun_M$.
Then the fiber of $\eH^{+,\fA}_M$ over this $k$-point $(x^\fA, \eF_M^1)$ 
is isomorphic to 
\[ \prod \Gr^{+,\theta_v}_{M,v} \]
where $\Gr^{+,\theta_v}_{M,v}$ is the closed subscheme of 
the affine Grassmannian $\Gr_{M,v}$ defined in \cite[\S 1.6]{BFGM}. 
% define Gr_{M,v} / k

In terms of loop and jet groups, 
\[ \Gr^+_{M,v}(k) = (\wbar M(\fo_v) \cap M(F_v))/M(\fo_v)
\subset \Gr_{M,v}(k) = M(F_v)/M(\fo_v).\]

\subsection{Convolution products}
Consider the diagram
\begin{equation}  \label{eqn:Heckemonoid}
	\xymatrix{	\eH^+_M & \ar[l]_-{\comp} \ds \eH^+_M \xt_{\overset\rightarrow h, \Bun_M, \overset\leftarrow h} \eH^+_M \ar[rr]^-{(\pr_1,\pr_2)} &&  \eH^+_M \xt \eH^+_M  }
\end{equation}
and recall (Remark~\ref{rem:compproper}) that the left arrow $\comp$ is proper. 
Using \eqref{e:HeGr+}, one sees that $\comp$ is isomorphic to the
twisted product of $\id_{\Bun_M}$ with the proper map
\begin{equation} \label{e:Grconv}
    \conv : \Gr^+_{M,X^{\theta_1}} \tilde\xt \Gr^+_{M,X^{\theta_2}} \to 
    \Gr^+_{M,X^\theta}, 
\end{equation}
where $\theta = \theta_1 + \theta_2$ and the left hand side is the
convolution Grassmannian (cf.~\cite[(3.1.21)]{Zhu}).

\subsubsection{}
We give $D(\eH^+_M)$ the structure of a monoidal category 
by the convolution product 
\[ \tilde\eF_1, \tilde\eF_2 \mapsto \tilde\eF_1 \star \tilde\eF_2 := \comp_!( (\pr_1,\pr_2)^*(\tilde\eF_1 \bt \tilde\eF_2) ) = \comp_!(\pr_1^*(\tilde\eF_1) \ot \pr_2^*(\tilde\eF_2)), \] 
where $\pr_1,\pr_2 : \eH^+_M \xt_{\Bun_M} \eH^+_M \to \eH^+_M$ are the projection maps. 
%We also have the dual product 
%\[ \eF \bt^! \eF' := (\pr_1,\pr_2)^!(\eF \bt \eF') = \pr_1^!(\eF) \ot^! \pr_2^!(\eF'). \]
\smallskip

Let $\Sph^+_{M,X^\theta} = D(\Gr^+_{M,X^\theta})^{M(\fo)_{X^{(\abs\theta)}}}$. 
Define a product 
\[ \star : \Sph^+_{M,X^{\theta_1}} \ot \Sph^+_{M,X^{\theta_2}} \to 
\Sph^+_{M,X^\theta}, \]
by $\eF_1 \star \eF_2 := \conv_!(\eF_1 \tbt \eF_2)$, 
which is a symmetrized version of the ``external convolution product'' (cf.~\cite[\S5.4]{Zhu})
for $\theta = \theta_1 + \theta_2$.

\smallskip

For $\eF_1 \in \Sph^+_{M,X^{\theta_1}},\, \eF_2 \in \Sph^+_{M,X^{\theta_2}}$, we can form 
the sheaves $(\wbar\bbQ_\ell)_{\Bun_M} \tbt \eF_i \in D(\eH^+_{M,X^{\theta_i}})$ using \eqref{e:HeGr+}.
By construction, there is a canonical isomorphism 
\[     (\wbar\bbQ_\ell \tbt \eF_1) \star (\wbar\bbQ_\ell \tbt \eF_2) \cong 
    \wbar\bbQ_\ell \tbt (\eF_1 \star \eF_2),
\]
We remark that $\star$ commutes with Verdier duality on $\Sph^+_{M,X^\theta}$. 

\smallskip

The category 
\[ \Sph^+_M := \{ \theta \mapsto \eF^\theta \in \Sph^+_{M,X^\theta} \} \]
has a natural monoidal structure with respect to $\star$: for two families
$\{ \eF^\theta_1\}$ and $\{ \eF^\theta_2\}$ the value of their product
in $\Sph^+_{M,X^\theta}$ is 
\[ \bigoplus_{\theta=\theta_1+\theta_2} \eF^{\theta_1}_1 \star \eF^{\theta_2}_2. \]

\subsubsection{Factorization property}  \label{sss:factorizationGr}
For $\theta = \theta_1 + \theta_2$, 
let $(X^{\theta_1} \xt X^{\theta_2})_\disj$ denote the the open locus
of $X^{\theta_1} \xt X^{\theta_2}$ consisting of 
pairs of colored divisors with disjoint supports. 
We have a natural \'etale map $(X^{\theta_1} \xt X^{\theta_2})_\disj \to X^\theta$. 

The schemes $\Gr^+_{M,X^\theta},\, \theta \in \Lambda^\pos_{G,P}$ factorize in the sense that there exist Cartesian diagrams
\[ \xymatrix{ \ds (\Gr^+_{M, X^{\theta_1}} \xt \Gr^+_{M,X^{\theta_2}}) \xt_{X^\theta_1 \xt X^{\theta_2}} (X^{\theta_1}\xt X^{\theta_2})_\disj \ar[r] \ar[d] & \Gr^+_{M,X^\theta} \ar[d] \\ 
(X^{\theta_1} \xt X^{\theta_2})_\disj \ar[r] & X^\theta } \]
for $\theta = \theta_1 + \theta_2$.

\subsubsection{} \label{sss:circledast}

The internal convolution (i.e., fusion) product (cf.~\cite[(5.4.4)]{Zhu}) 
\[ \circledast : D(\Gr^{+,\theta_1}_M)^{M(\fo)_X} \ot D(\Gr^{+,\theta_2}_M)^{M(\fo)_X} 
\to D(\Gr^{+,\theta}_M)^{M(\fo)_X} \]
for $\theta = \theta_1 +\theta_2$ is related to the $\star$ product as follows: 
Let $\Delta^\theta : \Gr^{+,\theta}_M \into \Gr^+_{M,X^\theta}$ denote the closed
embedding. 
Then there is a canonical isomorphism 
\begin{equation}\label{e:circlebox}
    \eF_1 \circledast \eF_2 \cong \Delta^{\theta *}( \Delta^{\theta_1}_*(\eF_1) \star
\Delta^{\theta_2}_*(\eF_2) ). 
\end{equation}

\section{Factorization algebras in \texorpdfstring{$\Sph^+_M$}{Sph+M}} \label{s:factorization}

In this section we review some of the objects introduced in \cite{BG2, Whatacts} 
and their properties. 
For our purposes, we only need to work with these objects at
a very coarse level (e.g., in the Grothendieck group) so we omit much of the
higher categorical nuances. 

Let $G$ be a connected split reductive group over a perfect field $k$.
Let $P$ be a standard parabolic subgroup of $G$. While the results of \emph{loc.~cit.} 
are stated only in the case $P=B$, we state them for arbitrary parabolics. 
The reader may check that the proofs readily generalize. 

We will continue using the notation of Appendix \ref{s:appendixHecke} for the monoid $\wbar M$, 
the Hecke stack, the affine Grassmannian, and the graded Ran space.

\subsection{Remarks on $G$} For simplicity, we assume throughout this Appendix 
that $G$ has a simply connected derived group $[G,G]$, so that we may use the same construction
of $\wtilde \Bun_P$ as in \cite{BG, BFGM}. The reader may refer to 
\cite[\S 7]{Schieder:HN} for how to remove this hypothesis, and the basic geometry of the Zastava space and Drinfeld's compactification remains the same. 

\subsection{Geometric Satake}
For simplicity, we will only use the non-factorizable geometric Satake functor. 
Let $\check M$ denote the Langlands dual group of $M$ over the field $\wbar\bbQ_\ell$.
Observe that each $\theta \in \Lambda^\pos_{G,P}$ defines a central character
of $\check M$. Let $\Rep(\check M)_\theta$ denote the
subcategory of $\check M$-modules with central character $\theta$.
Then we have a t-exact (with respect to the perverse t-structures) functor 
\[ 
\Sat^\naive_X : D(\Rep(\check M)_\theta) \ot D(X) \to D(\Gr^{+,\theta}_{M,X})^{M(\fo)_X},
\]
which is a special case of the factorizable geometric Satake functor
$\Sat^\naive_{\Ran(X)}$ constructed in \cite[\S 6]{Raskin}.
If we allow $\theta$ to range over all of $\Lambda^\pos_{G,P}$, then 
$\Sat^\naive_X$ is monoidal with respect to the usual tensor structure on the left
hand side and the internal convolution product $\circledast$ 
on the Beilinson-Drinfeld Grassmannian (cf.~\S\ref{sss:circledast}) on the right hand side.

\begin{rem} \label{rem:trSat}
Suppose $k = \bbF_q$.
Fix a closed point $v \in \abs X$ and let $m_v \in \Gr^{+,\theta}_{M,v}(\bbF_q)$.
As explained in \cite[\S 5.6]{Zhu}, the geometric Satake functor corresponds to 
the classical Satake isomorphism $\eS_v : H_{M,v} \to \bK(\Rep(\check M)) \ot \wbar\bbQ_\ell$ 
by Grothendieck's functions--sheaves dictionary: 
the trace of the geometric Frobenius at the
$*$-fiber at $m_v$ of $\Sat^\naive_X(V \ot (\wbar\bbQ_\ell)_X)$ equals 
\[ \eS_v^{-1}([V])(m_v), \]
where $[V]$ is the image of $V$ in the Grothendieck group 
$\bK(\Rep(\check M))$.
Here $H_{M,v}$ is the spherical Hecke algebra of $M(F_v)$, and $\eS_v$ is Langlands' reformulation of the classical Satake isomorphism (cf.~\S\ref{sect:localL}).
\end{rem}

\subsection{Factorization algebras}
We use the language of factorization algebras graded by the monoid $\Lambda^\pos_{G,P}$ introduced in \cite[\S 2]{Whatacts}. This is simply a particular case of the general notion of factorization algebra defined in \cite{CHA}. 

\subsubsection{} A $\Lambda^\pos_{G,P}$-graded factorization algebra $\eF \in \Sph^+_M$ is a family of sheaves $\eF^\theta \in \Sph^+_{M,X^\theta}$ 
such that for $\theta = \theta_1 + \theta_2$, we have an isomorphism 
\[ 
\eF^\theta |_{(X^{\theta_1} \xt X^{\theta_2})_\disj} \cong 
(\eF^{\theta_1} \bt \eF^{\theta_2}) |_{(X^{\theta_1} \xt X^{\theta_2})_\disj},
\]
of sheaves on $\Gr^+_{M,X^\theta} \xt_{X^\theta} (X^{\theta_1} \xt X^{\theta_2})_\disj$
satisfying the natural compatibilities. On the left hand side, we are restricting along the \'etale map $(X^{\theta_1} \xt X^{\theta_2})_\disj \to X^\theta$. On the right hand side, we restrict along the open embedding $(X^{\theta_1} \xt X^{\theta_2})_\disj \into X^{\theta_1} \xt X^{\theta_2}$ and use the factorization property of $\Gr^+_{M,X^\theta}$ explained in \S\ref{sss:factorizationGr}. 

\subsubsection{} Let $\eF \in \Sph^+_M$ be a commutative algebra object in the monoidal category. For $\theta = \theta_1 + \theta_2$, the multiplication map 
$\eF^{\theta_1} \star \eF^{\theta_2} \to \eF^\theta$
induces, by adjunction, a map 
\[  
(\eF^{\theta_1} \bt \eF^{\theta_2}) |_{(X^{\theta_1} \xt X^{\theta_2})_\disj}\to
\eF^\theta |_{(X^{\theta_1} \xt X^{\theta_2})_\disj}. \] 
We say that $\eF$ is a \emph{commutative factorization algebra} if these maps are isomorphisms for all $\theta = \theta_1 + \theta_2$.

\subsubsection{} Let $\eF \in \Sph^+_M$ be a cocommutative coalgebra object in the monoidal category. For $\theta = \theta_1 + \theta_2$, the comultiplication map 
$\eF^\theta \to \eF^{\theta_1} \star \eF^{\theta_2}$
induces, by adjunction, a map 
\[  \eF^\theta |_{(X^{\theta_1} \xt X^{\theta_2})_\disj} \to 
(\eF^{\theta_1} \bt \eF^{\theta_2}) |_{(X^{\theta_1} \xt X^{\theta_2})_\disj}. \] 
We say that $\eF$ is a \emph{cocommutative factorization algebra} if these maps are isomorphisms for all $\theta = \theta_1 + \theta_2$.

\subsection{Cocommutative factorization algebras}

\subsubsection{} \label{sect:u_P}
Let us recall the definition of the
cocommutative factorization algebra $\Upsilon(\check\fu_P) \in \Sph^+_M$ introduced in \cite{BG2}. 

For $\theta \in \Lambda^\pos_{G,P}$, let $\Delta^\theta : \Gr^{+,\theta}_M \into 
\Gr^+_{M,X^\theta}$ denote the closed diagonal. The $\Lambda^\pos_{G,P}$-graded 
$\check M$-module 
\[ \check \fu_P  = \bigoplus_{\alpha \in \Phi^+_G - \Phi^+_M} 
\check \fu_\alpha \]
gives a complex 
\[ \check\fu_{P,\Sph^+_M} := \bigoplus_{\alpha \in \Phi^+ - \Phi^+_M} 
\Delta^\alpha_*( \Sat^\naive_X (\check \fu_\alpha \ot (\wbar\bbQ_\ell)_X) ) \in \Sph^+_M. \]
The Lie algebra structure on $\check \fu_P$ gives a Lie algebra structure
to $\check\fu_{P,\Sph^+_M}$ with respect to the $\star$ monoidal structure
on $\Sph^+_M$.
Then 
\[ \Upsilon(\check \fu_P) := C_\bullet( \check\fu_{P,\Sph^+_M} ) \]
is the homological Chevalley complex associated to this Lie algebra, 
and $\Upsilon(\check\fu_P)$ is a cocommutative factorization algebra. 
%Miraculously, $\Upsilon(\check\fu_P)$ is in fact perverse (\cite{BG2}).

\subsubsection{} 
Let $U(\check \fu_P)_{\Sph^+_M}$ denote the universal 
enveloping algebra of the Lie algebra $\check \fu_{P,\Sph^+_M}$. 
This is a cocommutative factorization algebra in $\Sph^+_M$ with a compatible
associative algebra structure with respect to $\star$.

\begin{rem}
If we consider $\check \fu_{P,\Sph^+_M}[1]$ as a Lie superalgebra in degree $-1$,
then \[ \Upsilon(\check \fu_P) = U(\check \fu_{P,\Sph^+_M}[1]). \]
\end{rem}

\subsubsection{Restriction to strata}
\label{sss:Uprodv}
For $\theta \in \Lambda^\pos_{G,P}$ and $\fA \in \cP_\theta$ a partition,
let $x^\fA \in (X^\fA)_\disj(k)$ be a $\Lambda^\pos_{G,P}$-colored divisor 
$\sum \theta_v \cdot v$.
Then the fiber of $\Gr^{+,\fA}_M \to (X^\fA)_\disj$ over $x^\fA$
is isomorphic to $\prod \Gr^{+,\theta_v}_{M,v}$ by the factorization property. 

Since $\Upsilon(\check\fu_P)$ is a factorization algebra and
$\Sat^\naive_X$ is a monoidal functor, \eqref{e:circlebox} implies that  
the $*$-restriction of $\Upsilon(\check\fu_P)$ to this fiber of $\Gr^{+,\fA}_M$ canonically identifies with
\[ \bt_v \Sat^\naive_v ( C_\bullet(\check \fu_P )^{\theta_v} ) \]
where $C_\bullet(\check\fu_P)^{\theta_v}$ denotes the $\theta_v$-graded
piece of the Chevalley complex of the Lie algebra $\check \fu_P$ 
in $D(\Rep(\check M))$, and $\Sat^\naive_v$ is the non-relative
geometric Satake functor for $\Gr_{M,v}^{+,\theta_v}$.

The $*$-restriction of $U(\check\fu_P)_{\Sph^+_M}$ to the fiber $\prod \Gr^{+,\theta_v}_{M,v}$ canonically identifies with 
\[ \bt_v \Sat^\naive_v( U(\check\fu_P)^{\theta_v} ) \]
where $U(\check\fu_P)^{\theta_v}$ is the $\theta_v$-graded piece of the
universal enveloping algebra of $\check\fu_P$ in $D(\Rep(\check M))$.

\subsubsection{Koszul resolution} \label{sss:koszulco}
Let $\mathbf 1$ denote the constant sheaf $\wbar\bbQ_\ell$ on $\spec(k) = \Gr^+_{M,X^0}$. Then $\mathbf 1$ is the unit in the monoidal category $\Sph^+_M$.

Consider the acyclic complex $\check\fu_{P,\Sph^+_M} \to \check \fu_{P,\Sph^+_M}$
as a Lie superalgebra in degrees $-1,0$. 
The universal enveloping algebra of this Lie superalgebra 
is quasi-isomorphic to $\mathbf 1$. 
This resolution endows $\mathbf 1$ with the structure of a comodule 
with respect to $\Upsilon(\check\fu_P)$ and of a module with respect to 
$U(\check\fu_P)_{\Sph^+_M}$.

\subsection{Commutative factorization algebras}

\subsubsection{} We define 
the commutative factorization algebra $\Omega(\check \fu_P^-) \in \Sph^+_M$ as 
the Verdier dual of $\Upsilon(\check\fu^-_P)$. 
(Here we use the opposite Lie algebra $\check \fu_P^-$ so that 
$\Omega(\check\fu_P^-)$ is still $\Lambda^\pos_{G,P}$-graded.)

One can also define $\Omega(\check \fu^-_P)$ from
scratch by considering 
\[ (\check\fu_{P,\Sph^+_M}^-)^\vee := \bbD(\check\fu_{P,\Sph^+_M}^-)  \cong
\bigoplus_{\alpha \in \Phi^+_G - \Phi^+_M} \Delta_*^\alpha(\Sat^\naive_X(\check \fu_{-\alpha}^\vee \ot (\wbar\bbQ_\ell)_X(1)[2])). \]
%Here we have used the fact that $\Sat^\naive_X$ is compatible with Verdier duality.
Then $(\check\fu_{P,\Sph^+_M}^-)^\vee$ 
is a Lie coalgebra in $\Sph^+_M$ with respect to the $\star$ monoidal structure.
There is a canonical isomorphism 
\[ \Omega(\check \fu_P) \cong C^\bullet( (\check\fu_{P,\Sph^+_M}^-)^\vee ), \]
where the right hand side is the cohomological Chevalley complex associated to this Lie coalgebra.

\subsubsection{} Let $U^\vee(\check\fu_P^-)_{\Sph^+_M}$ denote the
universal co-enveloping coalgebra of $(\check\fu_{P,\Sph^+_M}^-)^\vee$.
This is a commutative factorization algebra in $\Sph^+_M$ with a compatible coalgebra
structure with respect to $\star$. 

\subsubsection{Restriction to strata}
\label{sss:Omegaprodv}
For $\theta \in \Lambda^\pos_{G,P}$ and $\fA \in \cP_\theta$ a partition,
let $x^\fA \in (X^\fA)_\disj(k)$ be a $\Lambda^\pos_{G,P}$-colored divisor 
$\sum \theta_v \cdot v$.
Then the fiber of $\Gr^{+,\fA}_M \to (X^\fA)_\disj$ over $x^\fA$ 
is isomorphic to $\prod \Gr^{+,\theta_v}_{M,v}$.

Since $\Omega(\check\fu_P^-)$ is a factorization algebra and 
$\Sat^\naive_X$ is a monoidal functor, 
the $*$-restriction of $\Omega(\check\fu_P^-)$ to this fiber of $\Gr^{+,\fA}_M$ canonically identifies with
\[ \bt_v \Sat^\naive_v ( C^\bullet( (\check \fu_P^-)^\vee (1)[2] )^{\theta_v} ) \]
where $C^\bullet((\check\fu_P^-)^\vee (1)[2])^{\theta_v}$ denotes the $\theta_v$-graded
piece of the cohomological Chevalley complex, and $\Sat^\naive_v$ is the non-relative
geometric Satake functor for $\Gr_{M,v}^{+,\theta_v}$.

The $*$-restriction of $U^\vee(\check\fu^-_P)_{\Sph^+_M}$ to the fiber $\prod \Gr^{+,\theta_v}_{M,v}$ canonically identifies with 
\[ \bt_v \Sat^\naive_v( U((\check\fu^-_P)^\vee (1)[2] )^{\theta_v} ) \]
where $U( (\check\fu^-_P)^\vee (1)[2] )^{\theta_v}$ is the $\theta_v$-graded piece of the
universal co-enveloping coalgebra.

\subsubsection{Koszul resolution} \label{sss:koszul}
Applying Verdier duality to the Koszul resolution in \S\ref{sss:koszulco} 
gives $\mathbf 1 \in \Sph^+_M$ the structure of a module with respect to $\Omega(\check\fu_P^-)$
and of a comodule with respect to $U^\vee(\check\fu_P)_{\Sph^+_M}$.

\subsection{Eisenstein series}
Let  \[ \jmath : \Bun_P \into \wtilde\Bun_P \]
denote the open embedding into Drinfeld's compactification of $\Bun_P$. 

\subsubsection{Action on Drinfeld's compactifications}
For every $\theta \in \Lambda^\pos_{G,P}$ there corresponds a proper map 
\[ \bar\iota^\theta : \wtilde\Bun_P \xt_{\Bun_M} \eH^+_{M,X^\theta} \to \wtilde\Bun_P. \]
Using \eqref{e:HeGr+}, we can express 
$\wtilde\Bun_P \xt_{\Bun_M} \eH^+_{M,X^\theta}$ as a twisted product 
$\wtilde\Bun_P \tilde\xt \Gr^+_{M,X^\theta}$.
Given $\eE \in D(\wtilde\Bun_P)$ and $\eF \in \Sph^+_{M,X^\theta}$,
we can form a sheaf $\eE \tbt \eF \in D(\wtilde \Bun_P \xt_{\Bun_M} \eH^+_{M,X^\theta})$.
We define an action of $\Sph^+_{M,X^\theta}$ on $D(\wtilde\Bun_P)$ 
by 
\[ \eE, \eF \mapsto \eE \star \eF := \bar\iota^\theta_*(\eE \tbt \eF), \]
which is compatible with the monoidal product $\star$ on $\Sph^+_M$.

\subsubsection{}

It was established in \cite[Theorem 4.2]{BG2} (cf.~\cite[Theorem 5.2.2]{Whatacts}) 
that there exists a map 
\[ \jmath_*(\IC_{\Bun_P}) \to \jmath_*(\IC_{\Bun_P}) \star \Upsilon(\check\fu_P)  \]
that gives $\jmath_*(\IC_{\Bun_P})$ the structure of an 
$\Upsilon(\check \fu_P)$-comodule.

By \cite[Theorem 6.6]{BG2} (cf.~\cite[Theorem 5.2.4]{Whatacts}), we have a 
quasi-isomorphism 
\begin{equation} \label{e:jotOmega}
\jmath_*(\IC_{\Bun_P}) \cotimes_{\Upsilon(\check\fu_P)} \mathbf{1} 
 \to \IC_{\wtilde\Bun_P} 
\end{equation}
where $\cotimes$ denotes the homotopy cotensor product over $\Upsilon(\check\fu_P)$ 
(i.e., the coBar construction), which  
can be computed using the Koszul resolution of $\mathbf{1}$ from 
\S\ref{sss:koszul}.

Recall that $\mathbf 1$ is also a $U(\check\fu_P)$-module. Thus 
applying the homotopy tensor product over $U(\check\fu_P)$ to \eqref{e:jotOmega}, we get
a quasi-isomorphism
\begin{equation} \label{e:j_*}
    \jmath_*(\IC_{\Bun_P}) \cong \jmath_*(\IC_{\Bun_P}) \cotimes_{\Upsilon(\check\fu_P)} 
    \mathbf{1} \ot_{U(\check\fu_P)} \mathbf{1} \cong \IC_{\wtilde\Bun_P} \ot_{U(\check\fu_P)} \mathbf 1,
\end{equation}
where the first quasi-isomorphism follows from the Koszul duality $\Upsilon(\check\fu_P) \cong 
\mathbf 1 \ot_{U(\check\fu_P)} \mathbf 1$.

\subsection{The factorization algebra $\tilde\Upsilon(\check\fu_P)$}
\label{ss:tildeU}

\subsubsection{}
Let $\iota^\theta$ denote the composition 
\[ \bar\iota^\theta \circ (\jmath \xt \id_{\eH^+_{M,X^\theta}}) : 
\Bun_P \xt_{\Bun_M} \eH^+_{M,X^\theta} \to \wtilde \Bun_P. \]
The maps $\iota^\theta$ are locally closed embeddings and their images define a 
stratification of $\wtilde\Bun_P$ (cf.~\cite[\S 6.2]{BG}, \cite{BFGM}). 

\subsubsection{}
Following \cite[Proposition 6.1.3]{Whatacts}, there exists a canonically
defined factorization algebra $\tilde\Upsilon(\check\fu_P)$ equipped with 
the structure of a coassociative coalgebra in $\Sph^+_M$ such that 
\[ \iota^{\theta *} (\jmath_*(\IC_{\Bun_P})) \cong \IC_{\Bun_P} \tbt \tilde\Upsilon(\check\fu_P)^\theta. \]

\subsubsection{} 
We now describe the image of $\tilde\Upsilon(\check\fu_P)$ in the Grothendieck
group of $\Sph^+_M$. 
Taking the image of \eqref{e:j_*} in the Grothendieck group 
gives the equality 
$[\jmath_* (\IC_{\Bun_P}) ] = [\IC_{\wtilde\Bun_P} \star \Upsilon(\check\fu_P)]$. 
For $\mu \in \Lambda_{G,P}$, let $\Bun_M^\mu$ denote the corresponding connected component consisting of $M$-bundles of degree $\mu$. 
Let $\Bun_P^\mu := \Bun_P \xt_{\Bun_M} \Bun_M^\mu$. 
We have a Cartesian square 
\[
    \xymatrix{ \ds \Bun_P^{\mu} \xt_{\Bun_M} \eH^+_{M,X^{\theta_1}} \xt_{\Bun_M} \eH^+_{M,X^{\theta_2}} \ar[r] \ar[d]_{\id_{\Bun_P} \xt \comp} & 
\ds \wtilde \Bun_P^{\mu-\theta_1} \xt_{\Bun_M} \eH^+_{M,X^{\theta_2}} \ar[d]^{\bar\iota^{\theta_2}} \\ 
\ds \Bun_P^\mu \xt_{\Bun_M} \eH^+_{M,X^{\theta}} \ar[r]^{\iota^\theta} & 
\wtilde\Bun_P^{\mu-\theta} 
} \]
where $\theta = \theta_1 + \theta_2$. 
Therefore pulling back by $\iota^{\theta *}$ gives 
\begin{equation} \label{eqn:j_*Groth}
	[ \iota^{\theta *} \jmath_*(\IC_{\Bun_P})] = \sum_{\theta_1 + \theta_2 = \theta} 
[ \iota^{\theta_1 *}(\IC_{\wtilde\Bun_P}) \star \Upsilon(\check\fu_P)^{\theta_2} ] 
\end{equation}
in the Grothendieck group of $\Bun_P \xt_{\Bun_M} \eH^+_{M,X^\theta}$.
The following result is proved in \cite[Theorem 1.12]{BFGM} after passing to
the Grothendieck group, and it is proved in the derived category in \cite[Proposition 4.4]{BG2}:
\begin{prop} \label{prop:BFGM}
	There exists a canonical isomorphism in $D(\Bun_P \xt_{\Bun_M} \eH^{+,\theta}_M)$:
	\[ \iota^{\theta*} (\IC_{\wtilde \Bun_P}) \cong \IC_{\Bun_P} \tbt U^\vee(\check \fu_P^-)_{\Sph^+_M}^{\theta}. \]
\end{prop}
\noindent
Combining \eqref{eqn:j_*Groth} and Proposition~\ref{prop:BFGM}, we deduce that
\begin{equation} \label{eqn:tildeUpsilon}
	[\tilde\Upsilon(\check \fu_P)^\theta ] = \sum_{\theta_1+\theta_2=\theta} 
	[ U^\vee(\check \fu_P^-)_{\Sph^+_M}^{\theta_1} \star \Upsilon(\check\fu_P)^{\theta_2} ] 
\end{equation}
in the Grothendieck group of $\Sph^+_M$.
\smallskip

\begin{prop} \label{prop:geom-nu} Suppose $k=\bbF_q$.
The trace of the geometric Frobenius on $*$-stalks of $\tilde\Upsilon(\check \fu_P)^\theta$ is equal to the function 
\[ (\eF_M, \beta_M)  \in \Gr^+_{M,X^\theta}(\bbF_q) \mapsto q^{-\brac{\check\rho_P,\theta}} \prod_v \nu_{M,v}(m_v), \]
where $m_v \in \Gr^+_{M,v}(\bbF_q) \subset M(F_v)/M(\fo_v)$ is determined by $\beta_M$, and $\nu_{M,v}$ is the $K_{M,v}$-bi-invariant measure on $M(F_v)$ defined in \S\ref{def:nu}. 
\end{prop}
\begin{proof}
Let $x^\fA \in (X^\fA)_\disj(\bbF_q),\, \fA \in \cP_\theta$, denote the image of 
$(\eF_M, \beta_M)$ under $\Gr^+_{M, X^\theta} \to X^\theta$. 
Then the fiber of $x^\fA$ is isomorphic to $\prod_v \Gr^{+,\theta_v}_{M,v}$,
and the point $(\eF_M,\beta_M)$ corresponds to the collection 
$\{ m_v \in \Gr^{+,\theta_v}_{M,v}(\bbF_q) \}$. 
Since $\tilde\Upsilon(\check\fu_P)$ is a factorization algebra, it suffices to 
consider the case when $\beta_M$ is an isomorphism on $X-v$ for a fixed
closed point $v$, i.e., $x^\fA = \theta_v \cdot v$ is supported at a single point $v$. 

The trace of geometric Frobenius on $*$-stalks of a complex
only depends on the image of the complex in the Grothendieck group. Therefore
\eqref{eqn:tildeUpsilon}, the discussion in \S\ref{sss:Uprodv} and \S\ref{sss:Omegaprodv}, and the compatibility of 
geometric Satake with the classical Satake transform (Remark~\ref{rem:trSat}) 
together imply that the trace of geometric Frobenius at the $*$-stalk of 
$(\eF_M, \beta_M)$ equals 
\[ \eS_v^{-1} \left( \left( \sum_{n=0} (-1)^n [\wedge^n \check\fu_P] \right) \ot 
\left(\sum_{n=0}^\infty [\Sym^n \check\fu_P] \cdot q_v^{-n} \right) \right)(m_v). \]
Comparing with \eqref{eqn:Lprod}, we deduce the proposition.
\end{proof}

\subsection{Geometric proof}
We prove Proposition~\ref{prop:geom-nu} above using \eqref{eqn:SatLnu}, which is essentially the classical Gindikin--Karpelevich formula. 
However, the Satake transform does not appear in the statement of Proposition~\ref{prop:geom-nu}. In this subsection we give a more direct proof of Proposition~\ref{prop:geom-nu} using derived algebraic geometry.

\subsubsection{Zastava spaces} \label{sss:Zastava}
% define Zastava
Let $\eZ^{P,\theta}$ denote the Zastava space defined in \cite{BFGM} corresponding to the parabolic $P$, and let $\oo\eZ{}^{P,\theta}$ denote the open
Zastava space (called $\eZ^{P,\theta}_\mathrm{max}$ in \emph{loc. cit.}).
We have a map $\pi_\eZ : \eZ^{P,\theta} \to \Gr^+_{M,X^\theta}$. 
Let $\oo \pi_\eZ : \oo\eZ{}^{P,\theta} \to \Gr^+_{M,X^\theta}$ denote the
restriction. 

\subsubsection{} 
Let $\tilde \Omega(\check\fu_P^-)$ denote the Verdier dual of 
$\tilde\Upsilon(\check \fu_P^-)$. Then $\tilde \Omega(\check\fu_P^-)$ is 
a factorization algebra on $\Sph^+_M$ with the defining equation
\[ \iota^{\theta !}(\jmath_! (\IC_{\Bun_P})) \cong \IC_{\Bun_P} \tbt 
\tilde\Omega(\check\fu_P^-)^\theta \]
for $\theta \in \Lambda^\pos_{G,P}$.
Using the local model of \cite[\S 3]{BFGM} and the contraction principle
(\cite[Proposition 5.2]{BFGM}), one sees that there is a canonical
isomorphism
\begin{equation} \label{e:OmegaZ}
    \tilde\Omega(\check\fu_P^-)^\theta \cong (\oo\pi_\eZ)_! ( \IC_{\oo\eZ{}^{P,\theta}} ).
\end{equation}
From this equation, factorization of $\tilde\Omega(\check\fu_P^-)$
follows from the factorization property of $\oo\eZ{}^{P,\theta}$. 

\begin{lem} \label{lem:f_Omega}
Suppose $k=\bbF_q$. The trace of geometric Frobenius on $*$-stalks of 
$\tilde\Omega(\check\fu_P^-)^\theta$ is equal to the function 
\[ (\eF_M, \beta_M) \in \Gr^+_{M,X^\theta}(\bbF_q) \mapsto 
q^{-\brac{\check\rho_P,\theta}} \prod_v \mu_{M,v}(m_v), \]
where $m_v \in \Gr^+_{M,v}(\bbF_q)\subset M(F_v)/M(\fo_v)$ is determined by $\beta_M$,
 and $\mu_{M,v}$ is the $K_{M,v}$-bi-invariant measure on $M(F_v)$ defined 
in \S\ref{sss:mu}.
\end{lem}
\begin{proof}
Let $x^\fA \in (X^\fA)_\disj(\bbF_q),\, \fA \in \cP_\theta$, denote the image of 
$(\eF_M, \beta_M)$ under $\Gr^+_{M, X^\theta} \to X^\theta$. 
Then the fiber of $x^\fA$ is isomorphic to $\prod_v \Gr^{+,\theta_v}_{M,v}$,
and the point $(\eF_M,\beta_M)$ corresponds to the collection 
$\{ m_v \in \Gr^{+,\theta_v}_{M,v}(\bbF_{q}) \}$. 
Since $\tilde\Omega(\check\fu_P)$ is a factorization algebra, it suffices to 
consider the case when $\beta_M$ is an isomorphism on $X-v$ for a fixed
closed point $v$, i.e., $x^\fA = \theta_v \cdot v$.
Therefore we can restrict our attention to the central fiber
\[ \oo \fZ{}^\theta_v := \oo\eZ{}^{P,\theta} \xt_{X^\theta} \spec(k) \]
where $\spec(k) \to X^\theta$ is the point $\theta_v \cdot v$. 
Let $\Gr_{P,v}^{\theta_v}$ denote the preimage under
$\Gr_{P,v} \to \Gr_{M/[M,M],v}$ of the point corresponding to $\theta_v$. 
By \cite[Proposition 2.6]{BFGM},
there is a natural identification $\oo\fZ{}^\theta_v \cong 
\Gr_{P,v}^{\theta_v} \cap \Gr_{U^-,v}$ such that $\oo\pi_\eZ$
corresponds to the map
\[ \Gr_{P,v}^{\theta_v} \cap \Gr_{U^-,v} \into \Gr_{P,v}
\to \Gr_{M,v}. \]
In other words, the central fiber of the open Zastava space is an
intersection of semi-infinite orbits in the affine Grassmannian.
Recall from \S\ref{sss:mu} that 
$\mu_{M,v}$ is defined as the measure of certain semi-infinite orbits. 
By Grothendieck's trace formula, we deduce that the trace
of geometric Frobenius on $*$-stalks of $(\oo\pi_\eZ)_!(\wbar\bbQ_\ell)_{\oo\eZ{}^{P,\theta}}$ equals the function $m_v \mapsto \mu_{M,v}(m_v)$. 
% \mu_M invariant under transpose aka chevalley anti-involution
Since $\oo\eZ{}^{P,\theta}$ is a smooth scheme of dimension 
$\brac{2\check\rho_P,\theta}$, we have proved the lemma. 
\end{proof}

By \cite[Proposition 6.2.2]{Whatacts}, we have a Koszul duality
\[ \mathbf 1 \ot_{\tilde\Omega(\check\fu_P^-)} \mathbf 1 \cong 
\tilde\Upsilon(\check\fu_P). \]
At the level of Grothendieck groups, this tells us that 
we have an equality $[\tilde\Omega(\check\fu_P^-)] \star [\tilde\Upsilon(\check\fu_P)] = [\mathbf 1]$.
Therefore the Grothendieck function of $\tilde\Omega(\check\fu_P^-)$
is the convolution inverse of the Grothendieck function
of $\tilde\Upsilon(\check\fu_P)$. 
Since $\nu_{M,v}$ is defined to be the convolution inverse of $\mu_{M,v}$, 
Lemma~\ref{lem:f_Omega} implies Proposition~\ref{prop:geom-nu}.

\section{The Drinfeld--Lafforgue--Vinberg compactification} \label{sect:DLV}
Let $k$ be a perfect base field. In this section we review the
definition and properties of the stack $\VinBun_G$ introduced in \cite{Schieder:gen}. 
In \S\ref{ss:VinBuntrue}, we mention an alternate definition of 
$\VinBun_G$ in the general case when $[G,G]$ is not simply connected.
For $k=\bbF_q$, we use results from \emph{loc.~cit.}~to compute the 
trace of the geometric Frobenius acting on the $*$-stalks of 
the $*$-pushforward of the constant sheaf $\wbar\bbQ_\ell$ under the diagonal
morphism $\Delta$ of $\Bun_G$ (see Theorem~\ref{thm:b}). This is done by
using a certain compactification of $\Delta$, which we construct 
in \S\ref{ss:compactify}.

\subsection{The Deconcini--Procesi--Vinberg semigroup} 
Set $T_\adj := T/Z(G)$, where $Z(G)$ is the center of $G$. 
The simple roots identify $T_\adj$ with $\bbG_m^{\abs{\Gamma_G}}$.  

Let $\wbar{G_\enh}$ denote the Vinberg semigroup of $G$, which admits
a homomorphism 
\[ \bar \pi : \wbar{G_\enh} \to \wbar{T_\adj} \] 
where $\wbar{T_\adj} := (\bbA^1)^{\abs{\Gamma_G}}$.
Any representation of $G_\enh$ decomposes into ones of the form $V \ot k_{\check\lambda}$
where $V \in \Rep(G)$ and $\check\lambda \in \check\Lambda$ is a weight of $T$,
such that for all weights $\check\mu$ of $V$, the difference $\check\lambda - \check\mu$ belongs to
the root lattice. By definition, $V \ot k_{\check\lambda} \in \Rep(\wbar{G_\enh})$
if and only if $\check\lambda-\check\mu \in \check\Lambda^\pos_G$ for all weights $\check\mu$
of $V$.

\subsubsection{} We recall some facts whose proofs can be found in 
\cite[\S 8]{Vinberg} in the characteristic zero case and in \cite{Ritt} in general. 
Let $V(\check\lambda)$ denote the irreducible $G$-module of highest weight $\check\lambda \in \check\Lambda^+_G$. 

Following \cite[Lemma D.4.2, Definition D.4.3]{DG:CT}, we 
let $\oo{\wbar{G_\enh}}$ denote the non-degenerate locus of $\wbar{G_\enh}$. By definition, this is the open subscheme of $\wbar{G_\enh}$ whose 
$\bar k$-points are the elements $g \in \wbar{G_\enh}(\bar k)$ with
nonzero action on $V(\check\lambda) \ot \bar k_{\check\lambda}$ 
for all dominant weights $\check\lambda \in \check\Lambda^+_G$.

It is known that $\overset\circ{\wbar{G_\enh}}$ is smooth over $\wbar{T_\adj}$.
The choice of Cartan subgroup $T \subset G$ defines a section $\fs : T_\adj \to G_\enh$ by $\fs(t) = (t^{-1},t)$, which extends
to a homomorphism of monoids 
\[ \bar\fs : \wbar{T_\adj} \to \wbar{G_\enh} \]
with image contained in $\overset\circ{\wbar{G_\enh}}$ (cf.~\cite[Lemma D.5.2]{DG:CT}). 
The $G \xt G$-action on $\oo{\wbar{G_\enh}}$ gives an equality 
(\cite[Corollary D.5.4]{DG:CT})
\[ \oo{\wbar{G_\enh}} = G \cdot \bar\fs(\wbar{T_\adj}) \cdot G. \]

\subsubsection{}
For a standard parabolic $P$ with Levi subgroup $M$, let  
$\mathbf{c}_P \in \wbar{T_\adj}$ be the point defined by the condition
that $\check\alpha_i(\mathbf{c}_P)=1$ for simple roots $\check\alpha_i$ contained inside $M$,
and $\check\alpha_i(\mathbf{c}_P)=0$ for all other simple roots. 

There is a canonical $T$-stable stratification of $\wbar{T_\adj}$ indexed by 
standard parabolics, with the point $\mathbf c_P$ contained in the stratum $(\wbar{T_\adj})_P$ corresponding to the parabolic $P$. 
In other words, 
\[ (\wbar{T_\adj})_P = \{ t \in \wbar{T_\adj} \mid 
\check\alpha_i(t) \ne 0,\, i\in \Gamma_M \text{ and }\check\alpha_j(t)=0,\,
j \in \Gamma_G - \Gamma_M \}. \] 
The $T$-action on $\mathbf c_P$ induces an isomorphism $T/Z(M) \cong (\wbar{T_\adj})_P$.
Define 
\[ (\wbar{T_\adj})_{\ge P} = \{ t \in \wbar{T_\adj} \mid 
\check\alpha_i(t) \ne 0,\, i\in \Gamma_M \} \]
to be the open locus of $\wbar{T_\adj}$ obtained by removing all 
strata corresponding to parabolic subgroups not containing $P$.

\begin{eg} Let $G = \on{SL}(2)$. Then $G_\enh = \on{GL}(2)$
and $\wbar{G_\enh} = \on{Mat}(2)$, the monoid of $2 \xt 2$ matrices.
In this case $T_\adj = \bbG_m,\, \wbar{T_\adj} = \bbA^1$, and
$\bar\pi : \on{Mat}(2) \to \bbA^1$ is the determinant map.
The non-degenerate locus $\oo{\wbar{G_\enh}} = \on{Mat}(2) - \{0\}$
is the open subset of nonzero matrices. 
Let $B$ equal the Borel of upper triangular matrices and identify 
$T$ with the subgroup of diagonal matrices in $G$. 
The section $\bar\fs$ corresponding to $(B,T)$ 
is the map $\bbA^1 \to \on{Mat}(2)$ sending $x \mapsto \left( \begin{smallmatrix} 1 & 0 \\ 0 & x \end{smallmatrix} \right)$.
The idempotent $\mathbf c_G$ equals $1 \in \bbA^1$, and 
$\mathbf c_B$ equals $0$.
\end{eg}

\subsubsection{}  \label{sss:Tsplit}
Note that the projection $T_\adj = T/Z(G) \to T/Z(M) = \bbG_m^{\abs{\Gamma_M}}$ has a natural splitting
$T/Z(M) \into T/Z(G)$ corresponding to the inclusion $\bbG_m^{\abs{\Gamma_M}} \xt \{1\} \into \bbG_m^{\abs{\Gamma_G}}$.
Set $\wbar{T/Z(M)} := (\bbA^1)^{\abs{\Gamma_M}}$. 
We have a decomposition $T_\adj = (T/Z(M)) \xt (Z(M)/Z(G))$, which 
extends to a decomposition 
\[ \wbar{T_\adj} = \wbar{T/Z(M)} \xt \wbar{Z(M)/Z(G)}, \] 
where $\wbar{Z(M)/Z(G)} \cong (\bbA^1)^{\abs{\Gamma_G} - \abs{\Gamma_M}}$
is the closure of $Z(M)/Z(G) \subset T_\adj$ in $\wbar{T_\adj}$.
Under the above decomposition, the point $(1,0)$ corresponds to $\mathbf c_P$,
the stratum $(\wbar{T_\adj})_P$ corresponds to $T/Z(M) \xt \{0\} = \bbG_m^{\abs{\Gamma_M}} \xt \{0\}$, and
$(\wbar{T_\adj})_{\ge P} = T/Z(M) \xt \wbar{Z(M)/Z(G)} = \bbG_m^{\abs{\Gamma_M}} \xt (\bbA^1)^{\abs{\Gamma_G} - \abs{\Gamma_M}}$. 

\subsection{The stack $\VinBun_G$}
Following \cite{Schieder:gen}, the stack 
\[ \VinBun_G  \subset \Maps(X, G \bs \wbar{G_\enh} / G) \]
is the open substack\footnote{The definition of $\VinBun_G$ is deceptively similar to that of $\tilde\eH_M$ in \S\ref{ss:tH_M}. We note the differences: in the definition of $\VinBun_G$, we consider the quotient stack by $G \xt G$ and not $G_\enh \xt G_\enh$. More importantly, the non-degenerate locus $\oo{\wbar{G_\enh}}$ is larger than the subgroup $G_\enh$.} of maps generically landing in the non-degenerate locus
$\oo{\wbar{G_\enh}}$.
For a test scheme $S$, an $S$-point of $\VinBun_G$ is a datum of
$(\eF_G^1, \eF_G^2, \beta)$, where 
$\eF_G^1, \eF_G^2$ are $G$-bundles on $X \xt S$ and 
\[ \beta : X \xt S \to \wbar{G_\enh} \xt^{G \xt G} (\eF_G^1 \xt_{X \xt S} \eF_G^2) \] 
is a section over $X \xt S$ that generically lands 
in $\oo{\wbar{G_\enh}}$ over every geometric point of $S$.
We call such a section $\beta$ a $\wbar{G_\enh}$-morphism $\eF_G^2 \to \eF_G^1$.

The map $\bar\pi$ induces a map 
\[ \bar \pi_{\Bun} : \VinBun_G \to \Maps(X,\wbar{T_\adj}) = \wbar{T_\adj},\] 
where
the last equality holds because $X$ is proper and geometrically connected.
Let $\VinBun_{G,P}$ (resp.~$\VinBun_{G,\ge P}$) denote the preimage of $(\wbar{T_\adj})_P$
(resp.~$(\wbar{T_\adj})_{\ge P}$) under $\bar\pi_{\Bun}$.
Note that $\VinBun_{G,\ge P}$ contains the open stratum $\VinBun_{G,G}$.

\subsection{The $T_\adj$-action on $\VinBun_G$} \label{sect:Tact}
In what follows, we will define a canonical action of $T_\adj$ on $\VinBun_G$ which is equivariant with respect to $\bar \pi_{\Bun}$ and the identity action on $T_\adj$. 

\subsubsection{}  \label{sss:stackquot}
Suppose we have an exact sequence of algebraic groups 
\[ 1 \to H' \to H \to H'' \to 1 \]
and an action of $H$ on a $k$-scheme $Y$. Then the stack $H' \bs Y$
is an $H''$-torsor over the stack $H \bs Y$: indeed, the morphism 
$H' \bs Y \to H \bs Y$ is obtained by base change from the $H''$-torsor $H' \bs \on{pt} \to H \bs \on{pt}$, where $\on{pt} = \spec(k)$. 

\smallskip

In particular, $H''$ acts on $H'\bs Y$ over $H \bs Y$. One can think of
this action as follows. An $S$-point of $H \bs Y$ is an 
$H$-torsor $\eF_H \to S$ equipped with an $H$-equivariant morphism $\eF_H \to Y$. 
Lifting these data to a morphism $S \to H' \bs Y$ is the same as specifying
an $H'$-structure on $\eF_H$, which is the same as specifying an $H$-equivariant morphism $\eF_H \to H''$. The set of all such morphisms $\eF_H \to H''$
is equipped with an action of $H''(S)$ (by right translations). 

\subsubsection{} 
Applying the discussion above to $Y = \wbar{G_\enh}$ and the exact sequence 
\[ 1 \to G \xt G \to G_\enh \xt G_\enh \to T_\adj \xt T_\adj \to 1 \]
(so $H = G_\enh \xt G_\enh$, $H' = G \xt G$, and $H'' = T_\adj \xt T_\adj$), 
one gets a canonical action of $T_\adj \xt T_\adj$ on $G \bs \wbar{G_\enh} / G$. We will be considering only the action of $T_\adj = T_\adj \xt \{1\} 
\subset T_\adj \xt T_\adj$ (which comes from the action of $G_\enh$ on $\wbar{G_\enh}$ by left translations). 

\subsubsection{}
The $T_\adj$-action on $G \bs \wbar{G_\enh} / G$ preserves 
$G \bs \oo{\wbar{G_\enh}} / G$, so it induces a $T_\adj$-action 
on $\VinBun_G \subset \Maps(X, G \bs \wbar{G_\enh} / G)$.

\medskip

This action can be described explicitly as follows.
A $G$-bundle $\eF_G$ on $X\xt S$ is equivalent to a $G_\enh$-bundle $\eF_{G_\enh}$ on $X \xt S$ together with a trivialization of the induced $T_\adj$-bundle $\pi(\eF_{G_\enh})$. 
The group $T_\adj(S)$ acts on the space of such trivializations, so
$T_\adj$ acts on $\VinBun_G$ by leaving the $G_\enh$-bundle $\eF^1_{G_\enh}$ induced by 
$\eF^1_G$ fixed and changing the trivialization of $\pi(\eF^1_{G_\enh})$.

\subsection{Fiber bundles} \label{sect:fiberbundles}
Fix a standard parabolic subgroup $P$, and consider the open locus
$\VinBun_{G,\ge P}$ lying over 
\[ (\wbar{T_\adj})_{\ge P} = T/Z(M) \xt \wbar{Z(M)/Z(G)}. \]
Let $\VinBun_{G,\ge P,\strict} = \bar\pi_{\Bun}^{-1}(\{1\}\xt\wbar{Z(M)/Z(G)})$.
Since we have the splitting $T/Z(M) \into T_\adj$ (see \S\ref{sss:Tsplit}), 
the $T_\adj$-action on $\VinBun_G$ defined in \S\ref{sect:Tact} restricts to a 
$T/Z(M)$-action on $\VinBun_{G,\ge P}$. This action 
induces an isomorphism
\begin{equation} \label{eqn:>=Pfiber}
    \VinBun_{G,\ge P} \cong \VinBun_{G,\ge P,\strict} \xt  (T/Z(M)),
\end{equation}
i.e., $\VinBun_{G,\ge P}$ is a trivial fiber bundle over the
projection to $T/Z(M)$. 

Note that $\mathbf c_P$ is the zero element in $\wbar{Z(M)/Z(G)}$.
Let $(\wbar{G_\enh})_{\mathbf c_P}$ denote the fiber of $\bar \pi$ over $\mathbf c_P$. 
Then $\VinBun_{G,\mathbf c_P} := \bar\pi_{\Bun}^{-1}(\mathbf c_P)$ 
is equal to the stack
\[ \Maps^\circ(X, G\bs (\wbar{G_\enh})_{\mathbf c_P} / G) \]
where the superscript $^\circ$ denotes the open substack of maps generically landing in the
non-degenerate locus.
Intersecting \eqref{eqn:>=Pfiber} with the $P$-locus gives 
\begin{equation} \label{eqn:Pfiber}
	\VinBun_{G,P} \cong \VinBun_{G,\mathbf c_P} \xt (T/Z(M)). 
\end{equation}
In the case $P=G$, the $G$-locus $\VinBun_{G,G}$ is isomorphic to $\Bun_G \xt T_\adj$.

\subsubsection{}
It is known (cf.~\cite[Appendix C]{DG:CT}) that the $G \xt G$-action on $\bar\fs(\mathbf{c}_P) \in \oo{\wbar{G_\enh}}$ 
induces an isomorphism 
\[ \bbX_P := (G \xt G)/(P \xt_M P^-) \cong (\oo{\wbar{G_\enh}})_{\mathbf c_P}.\] 
We learned of the following lemma from \cite[Lemma 2.1.11]{Schieder:gen}.

\begin{lem}  \label{lem:X_PVin}
	The variety $(\wbar{G_\enh})_{\mathbf c_P}$ is isomorphic to $\wbar\bbX_P$.
\end{lem}
\begin{proof}
	By \cite[Theorem 7]{Ritt}, the irreducible affine variety $(\wbar{G_\enh})_{\mathbf c_P}$
	is normal. 
	Since $(\wbar{G_\enh})_{\mathbf c_P}$ contains 
	the non-degenerate locus $(\oo{\wbar{G_\enh}})_{\mathbf c_P} \cong \bbX_P$ as a dense open subscheme, 
the extension property of regular functions on normal varieties implies that
it suffices to show
that the degenerate locus in $(\wbar{G_\enh})_{\mathbf c_P}$ has codimension at least $2$.
This can be checked by considering the combinatorial description
of $G\xt G$-orbits in $\wbar{G_\enh}$ from \cite[Theorem 6]{Ritt}. 
In characteristic $0$, this is the same combinatorial description as in \cite{Vinberg}.
\end{proof}

Lemma~\ref{lem:X_PVin} implies that $\VinBun_{G,\mathbf c_P}$ is isomorphic
to $\Maps^\circ( X, G \bs \wbar \bbX_P / G)$, where the superscript $^\circ$ denotes
the open locus of maps generically landing in $G \bs \bbX_P / G$. 

\subsection{Defect stratification} \label{sect:defect}
Define the closed embedding $M \into \bbX_P$ as the composition 
of the closed embeddings $M \into G/U : m \mapsto mU$ and $G/U \into \bbX_P : g \mapsto (g,1)$. 
From \cite[Corollary 4.1.5]{Wa} we know that $M \into \bbX_P$ extends 
to a closed embedding $\wbar M \into \wbar \bbX_P$.
This induces a map 
\[ 
	\Maps^\circ( X, P \bs \wbar M / P^-) \to \Maps^\circ(X, G \bs \wbar \bbX_P / G)  = \VinBun_{G,\mathbf c_P}, 
\]
where $\Maps^\circ(X, P \bs \wbar M / P^-)$ is the stack of 
maps $X \xt S \to P \bs \wbar M / P^-$ that generically land in 
$P \bs M / P^-$ over every geometric point of $S$.

\subsubsection{}
Let $\eH^+_M = \Maps^\circ(X, M \bs \wbar M / M)$ denote the stack introduced in \S\ref{sect:H_M}. 
Recall that there are two maps $\overset\leftarrow h, \overset\rightarrow h : \eH^+_M \to \Bun_M$.
Observe that there is a canonical isomorphism
\[ \Maps^\circ( X, P \bs \wbar M / P^- ) \cong \Bun_P \xt_{\Bun_M, \overset\leftarrow h} \eH_M^+ \xt_{\overset\rightarrow h, \Bun_M} 
\Bun_{P^-}. \]
Thus we have a map of stacks 
\begin{equation} \label{eqn:strat}
 \Bun_P \xt_{\Bun_M} \eH_M^+ \xt_{\Bun_M} \Bun_{P^-} \to \VinBun_{G,\mathbf c_P}.
\end{equation}

\subsubsection{}

Let $\Lambda_{G,P} = \pi_1(M)$ denote the quotient of $\Lambda$ by the 
subgroup generated by the coroots of $M$. 
Recall that there is a bijection $\pi_0(\Bun_M) \cong \Lambda_{G,P}$. 
Let $\Bun_M^\mu,\, \mu \in \Lambda_{G,P}$ denote the corresponding
connected component consisting of $M$-bundles of degree $\mu$.  

Let $\Bun_P^\mu$ (resp.~$\Bun_{P^-}^\lambda$) denote the preimage of $\Bun_M^\mu$ 
(resp.~$\Bun_M^\lambda$) under the projection $\Bun_P \to \Bun_M$ 
(resp.~$\Bun_{P^-} \to \Bun_M$).

\subsubsection{} 
For $\mu,\lambda \in \Lambda_{G,P}$, let 
$\eH_M^{+,\mu,\lambda}$ denote the preimage of $\Bun_M^\mu \xt \Bun_M^\lambda$
under $(\overset\leftarrow h, \overset\rightarrow h)$. 
One can check using Remark~\ref{rem:heckeHN} that $\eH_M^{+,\mu,\lambda}$ is nonempty if and only if $\mu-\lambda$ lies in the image of $\Lambda^{\pos,\bbQ}_{U}\cap \Lambda$ under the projection $\Lambda \to \Lambda_{G,P}$.

\begin{prop}[{\cite[Proposition 3.2.2]{Schieder:gen}}] 
(i) 
For $\mu,\lambda \in \Lambda_{G,P}$, the restriction of \eqref{eqn:strat} to the corresponding substack 
\begin{equation} \label{eqn:defect} 
	\Bun_P^\mu \xt_{\Bun^\mu_M} \eH_M^{+,\mu,\lambda} \xt_{\Bun^\lambda_M} \Bun_{P^-}^\lambda
	\to \VinBun_{G,\mathbf c_P}
\end{equation}
is a locally closed embedding. 
Let $\VinBun_{G,\mathbf c_P}^{\mu,\lambda}$ denote the corresponding locally
closed substack. 

(ii) The locally closed substacks $\VinBun_{G,\mathbf c_P}^{\mu,\lambda}$
form a stratification of $\VinBun_{G,\mathbf c_P}$. In particular,
the map \eqref{eqn:strat} is a bijection at the level of $k$-points. 
\end{prop}

We call the stratification above the \emph{defect stratification} of 
$\VinBun_{G,\mathbf c_P}$. % and hence of $\VinBun_{G,P}$. 
By \eqref{eqn:Pfiber}, we have an identical stratification
of $\VinBun_{G,P}$ with strata 
\[ \VinBun_{G,P}^{\mu,\lambda} \cong 
\VinBun_{G,\mathbf c_P}^{\mu,\lambda} \xt (T/Z(M)).\]

\subsection{Remarks on the case when $[G,G]$ is not simply connected} 
\label{ss:VinBuntrue}
Note that we have not assumed that $[G,G]$ is simply connected in this Appendix (unlike in Appendices \ref{s:appendixHecke}--\ref{s:factorization}). 
Without this assumption, it is possible for the image of $\Lambda^{\pos,\bbQ}_{U} \cap \Lambda$ in $\Lambda_{G,P}$ to be larger than $\Lambda^\pos_{G,P}$. 
%%%
Then Lemma~\ref{lem:Gstratclosure} below shows that the $G$-stratum $\VinBun_{G,G}$ is not necessarily dense in $\VinBun_G$ when $[G,G]$ is not simply connected. 
\medskip

We define a slightly different stack $\VinBun^\true_G$.
We suggest that $\VinBun^\true_G$ is the more ``philosophically correct'' definition for $\VinBun_G$ when $[G,G]$ is not simply connected.

\begin{rem}
We will continue using the original stack $\VinBun_G$ in the rest of this article (for arbitrary $G$) as it suffices for our purposes, but it is also possible to work directly with $\VinBun^\true_G$ everywhere. 
\end{rem}

\subsubsection{} The idea is to replace the Vinberg semigroup $\wbar{G_\enh}$, which is an algebraic monoid, by its stacky version, which is an algebraic monoidal stack\footnote{An algebraic monoidal stack $\eY$ is an 
algebraic stack $\eY$ with a coherently associative composition law $\eY \xt \eY \to \eY$ and a morphism $\spec(k) \to \eY$ such that for any scheme $S$, the composition $S \to \spec(k) \to \eY$ is a unit object of $\eY(S)$.}. 

Let $G^\ssc$ denote the universal cover of $[G,G]$ (as algebraic groups).
Then $Z(G^\ssc)$ is a finite group scheme containing $\ker(G^\ssc \to G)$.
Since $G = Z(G) \cdot [G,G]$, we have an isomorphism
\begin{equation} \label{e:Gssc-twist}
    G \cong (G^\ssc \xt Z(G)) / Z(G^\ssc),
\end{equation}
where $Z(G^\ssc)$ is embedded in $G^\ssc \xt Z(G)$ anti-diagonally.
Furthermore, the following is well-known (cf.~\cite{Vinberg, Ritt98}):
\begin{lem}\label{lem:VinGIT}
The isomorphism \eqref{e:Gssc-twist} extends to an isomorphism of the Vinberg semigroup $\wbar{G_\enh}$ with the the GIT quotient of $\wbar{G^\ssc_\enh} \xt Z(G)$ by the anti-diagonal action of $Z(G^\ssc)$. 
\end{lem}

Lemma~\ref{lem:VinGIT} motivates us to 
define the stacky version of the Vinberg semigroup as the quotient stack
\begin{equation}
    (\wbar{G_\enh})^\true := (\wbar{G^\ssc_\enh} \xt Z(G)) / Z(G^\ssc). 
\end{equation}
This is an algebraic monoidal stack. By Lemma~\ref{lem:VinGIT}, we have a 
canonical map 
\begin{equation}\label{e:truemap}
    (\wbar{G_\enh})^\true \to \wbar{G_\enh}
\end{equation}
from the stack quotient to the GIT quotient, and we see that $\wbar{G_\enh}$ is the coarse moduli space of $(\wbar{G_\enh})^\true$.

By \eqref{e:Gssc-twist}, we see that the homomorphism \eqref{e:truemap} restricts to an isomorphism of groups 
$G^\ssc_\enh \xt^{Z(G^\ssc)} Z(G) \cong G_\enh$, 
so we have an open embedding 
\[ G_\enh \into (\wbar{G_\enh})^\true. \]
Moreover, $Z(G^\ssc_\enh)$ acts freely\footnote{The quotient $\oo{\wbar{G_\enh}}/Z(G_\enh)$ is the \emph{wonderful compactification} of $G/Z(G)$.} 
on the non-degenerate locus $\oo{\wbar{G^\ssc_\enh}}$. Thus the open substack 
\[ (\oo{\wbar{G_\enh}})^\true := (\oo{\wbar{G^\ssc_\enh}} \xt Z(G))/Z(G^\ssc) \]
 is representable by a scheme, and Lemma~\ref{lem:VinGIT} implies that 
\eqref{e:truemap} restricts 
to an isomorphism $(\oo{\wbar{G_\enh}})^\true \cong \oo{\wbar{G_\enh}}$ on non-degenerate loci.

\begin{lem} \label{lem:truess}
The map \eqref{e:truemap} is an isomorphism if and only if $[G,G]$ is simply connected.
\end{lem}
\begin{proof}
Recall that $\wbar{G^\ssc_\enh}$ is an algebraic monoid with zero. 
Thus the action of $\ker(G^\ssc \to [G,G])$ on $\wbar{G^\ssc_\enh}$ is free if and only if the kernel is trivial, i.e., $[G,G]$ is simply connected. 
Therefore if $[G,G]$ is not simply connected, the stack quotient $(\wbar{G_\enh})^\true$ cannot be representable by a scheme.

In the other direction, suppose that $[G,G]$ is simply connected.
Then $Z(G^\ssc) \subset Z(G)$ acts freely on $\wbar{G^\ssc_\enh} \xt Z(G)$, so \eqref{e:truemap} is an isomorphism by Lemma~\ref{lem:VinGIT}.
\end{proof}

\subsubsection{Definition of $\VinBun^\true_G$}
As explained in \S\ref{sss:stackquot}, we have an action of $G \xt G$ on the 
stack $(\wbar{G_\enh})^\true$, where we are using the identification \eqref{e:Gssc-twist}. 
Define 
\begin{equation} \label{e:VinBun-canon}
    \VinBun^\true_G = \Maps^\circ (X, G \bs (\wbar{G_\enh})^\true /G ), 
\end{equation}
where the superscript $^\circ$ denotes the locus of maps that
generically land in $G \bs \oo{\wbar{G_\enh}} / G$ over every geometric point of a test scheme $S$.

\subsubsection{}
The map \eqref{e:truemap} induces a canonical map of stacks 
\begin{equation}\label{e:Vintruemap} 
    \VinBun^\true_G \to \VinBun_G 
\end{equation}
over $\Bun_G \xt \Bun_G$.

The open embedding $G_\enh \into (\wbar{G_\enh})^\true$ induces
an open embedding 
\[ \VinBun_{G,G} \into \VinBun^\true_G \]
over $\VinBun_G$.

\subsubsection{}
The projection $\wbar{G^\ssc_\enh} \xt Z(G) \to Z(G)$ induces a 
homomorphism of monoidal stacks $(\wbar{G_\enh})^\true \to Z(G)/Z(G^\ssc)$. 
Note that $G^\abst := Z(G)/Z(G^\ssc)$ is the stacky abelianization of $G$ defined in \cite{StAb}. 
By \eqref{e:Gssc-twist}, we also get a homomorphism
 of group stacks 
$G \to G^\abst$. The $G \xt G$-action on $(\wbar{G_\enh})^\true$ 
is compatible with these maps to $G^\abst$.

Consider the map $\spec(k) \to G^\abst$ corresponding to $1\in Z(G)$. 
We have Cartesian squares
\[ \xymatrix{ \wbar{G^\ssc_\enh} \ar[r] \ar[d] & (\wbar{G_\enh})^\true \ar[d] \\ \spec(k) \ar[r] & G^\abst } \quad\quad 
\xymatrix{ \ds G \xt_{G^{\abst}} G \ar[r] \ar[d] & G \xt G \ar[d] \\
\spec(k) \ar[r] & G^\abst } \]
where $G \xt G$ maps to $G^\abst$ by $(g_1,g_2) \mapsto g_1 g_2^{-1}$. 
Note that $G \xt_{G^\abst} G$ is isomorphic to a semidirect product of $G$ and $G^\ssc$.
Since $G^{\abst}$ is a group stack, it follows formally that 
\[
 (\wbar{G_\enh})^\true / (G \xt G) \cong \wbar{G^\ssc_\enh} / ( G \xt_{G^\abst} G). 
\]

We can also repeat the above discussion at the coarse level: by Lemma~\ref{lem:VinGIT}, we have $\wbar{G_\enh} = (\wbar{G^\ssc_\enh} \xt Z(G))/\!\!/ Z(G^\ssc)$, where $/\!\!/$ denotes the GIT quotient. 
Thus the projection to the second factor induces a homomorphism of monoids
$\wbar{G_\enh} \to Z(G)/\!\!/ Z(G^\ssc) = Z(G)/Z([G,G]) = G/[G,G] =: G^\ab$ 
such that the $G \xt G$-action on $\wbar{G_\enh}$ lies over $G^\ab$.
Let $(\wbar{G_\enh})_1$ denote the fiber
$\wbar{G_\enh} \xt_{G^\ab} \spec(k)$ over $1\in G^\ab(k)$. 
Then it again follows formally that 
\[ \wbar{G_\enh} / (G \xt G) \cong (\wbar{G_\enh})_1 / G \xt_{G^\ab} G. \]

We have a group homomorphism $G \xt_{G^\abst} G \to G \xt_{G^\ab} G$ whose 
kernel is isomorphic to $\ker(G^\ssc \to [G,G])$. We also have a finite homomorphism of monoids $\wbar{G^\ssc_\enh} \to (\wbar{G_\enh})_1$. 
Thus we deduce that there is a commutative diagram
\[ \xymatrix{ \VinBun^\true_G \ar[d] \ar[rd] \\ 
\Maps^\circ(X, (\wbar{G_\enh})_1 / (G \xt_{G^\abst} G) ) \ar[r] \ar[d] &
\Bun_{G \xt_{G^\abst} G} \ar[d] \\ 
\VinBun_G \ar[r] & \Bun_{G \xt_{G^\ab} G} \ar[r] & \Bun_G \xt \Bun_G } \]
where the square is Cartesian, the superscript $^\circ$ denotes the substack of maps generically landing in $(\oo{\wbar{G_\enh}})_1 / (G \xt_{G^\abst} G)$, 
and the composition of the left vertical maps equals the map \eqref{e:Vintruemap}.  
Here we have factored the map \eqref{e:Vintruemap} into a ``change of space'' map and a ``change of group'' map. 
The following lemma is well-known.

\begin{lem} \label{lem:Buntorsor}
Let $H$ be a connected reductive group, and let $A \subset Z(H)$ be a finite central subgroup. 
Then $\Bun_H \to \Bun_{H/A}$ is a torsor by the group stack $\Bun_A$ 
over an open and closed substack of $\Bun_{H/A}$. 
More specifically, this substack is the union of the connected components
in $\pi_0(\Bun_{H/A}) = \pi_1(H/A)$ corresponding to 
$\pi_1(H) \subset \pi_1(H/A)$. 
\end{lem}
\begin{proof} Let $B$ be a Borel subgroup of $H$. 
Then $\Bun_B \to \Bun_H$ and $\Bun_{B/A} \to \Bun_{H/A}$ are surjective by \cite{DS}. Thus to prove the statement about the image of $\Bun_H$ in $\Bun_{H/A}$, it suffices to consider $\pi_0(\Bun_B) \to \pi_0(\Bun_{B/A})$. This reduces
to an analogous statement in the case where $H$ is a torus, which is straightforward. 

It is a standard fact that the action of
$\Bun_A$ on $\Bun_H$ defines an isomorphism between $\Bun_A\times\Bun_H$ and the
fiber product $\Bun_H \xt_{\Bun_{H/A}} \Bun_H$. 
Therefore to prove that the map from $\Bun_H$ to its image in $\Bun_{H/A}$ is a torsor, we must show that this map is flat. The map is flat because it is a morphism between smooth stacks of the same dimension with $0$-dimensional fibers.
\end{proof}

We now consider the ``change of space'' map 
\begin{equation} \label{e:VinBuncos}
\VinBun^\true_G = \Maps^\circ\left(X, \wbar{G^\ssc_\enh}/(G \xt_{G^\abst} G) \right) 
\to \Maps^\circ\left(X, (\wbar{G_\enh})_1/ (G \xt_{G^\abst} G)\right). 
\end{equation}

\begin{lem} \label{lem:Mapsfin}
Let $\eY_1 \to \eY_2$ be a finite schematic morphism of stacks. 
Then the induced morphism $\Maps(X,\eY_1) \to \Maps(X,\eY_2)$ is also finite schematic.
\end{lem}
\begin{proof}
Fix a test scheme $S$ and a map $X_S := X \xt S \to \eY_2$. 
Let $Y$ denote the fiber product $X_S \xt_{\eY_2} \eY_1$, which is representable by a finite scheme over $X_S$. 
Then the corresponding fiber product of $S$ and $\Maps(X,\eY_1)$ over $\Maps(X,\eY_2)$ is representable by the $S$-scheme $\Sect(X_S, Y)$ of sections 
of $Y \to X_S$. 
It is well-known that since $Y \to X_S$ is affine, $\Sect(X_S,Y) \to S$ is also affine. Therefore to show that $\Sect(X_S,Y) \to S$ is finite, it suffices to show that it is proper.
We use the valuative criterion of properness:

Let $R$ be a discrete valuation ring with field of fractions $K$, and suppose that we have 
a map $\spec(R) \to S$ and a section $X_K := X \xt \spec(K) \to Y$ over $X_S$. 
This section and the natural map $\spec(K) \to \spec(R)$ define
a section $X_K \to Y_R := Y \xt_S \spec(R)$. 
Let $Z_K$ denote the image of $X_K \to Y_R$, and let $Z_R$ denote the scheme-theoretic closure of $Z_K$ in $Y_R$. 
Extending the section $X_K \to Y$ to a section $X_R := X \xt \spec(R) \to Y$
is equivalent to showing that the projection $Z_R \to X_R$ is an isomorphism.
Note that $Z_R$ is an integral scheme (because $Z_K$ is) and the
map $Z_R \to X_R$ is birational (because the map $Z_K\to X_K$ is an isomorphism). 
On the other hand, the map $Z_R \to X_R$ is finite since $Y_R \to X_R$ is finite. Lastly, smoothness of $X$ implies that $X_R$ is a regular scheme. Hence $X_R$ is normal, and $Z_R \to X_R$ is an isomorphism.
This checks the condition of the valuative criterion and hence proves the lemma.
\end{proof}

Let $A$ denote the finite abelian group scheme $\ker(G^\ssc \to [G,G])$. 

\begin{cor} \label{cor:VinBuncos}
The map \eqref{e:VinBuncos} is schematic and finite. 
More specifically, it is the composition of an $A$-torsor followed by a closed embedding. 
\end{cor}
\begin{proof}
The map $\wbar{G^\ssc_\enh} \to (\wbar{G_\enh})_1$ is finite, and 
the preimage of $(\oo{\wbar{G_\enh}})_1$ equals $\oo{\wbar{G^\ssc_\enh}}$.
Then Lemma~\ref{lem:Mapsfin} implies that the map \eqref{e:VinBuncos} is schematic and finite. 
Let \[ \eY\subset \Maps^\circ(X, (\wbar{G_\enh})_1 / (G \xt_{G^\abst} G))  \]
denote the (scheme-theoretic) image, which is a closed substack. 
Take $(\eP,\beta) \in \eY(S)$ for a test scheme $S$, where $\eP$ is a $G \xt_{G^\abst} G$-torsor over $X_S := X \xt S$ and 
$\beta$ is a section $X_S \to ((\wbar{G_\enh})_1)_\eP$ over $X_S$ such that $\beta|_{\spec(F) \xt S}$ lands in $((\oo{\wbar{G_\enh}})_1)_\eP$. 
Since $\oo{\wbar{G^\ssc_\enh}}\to (\oo{\wbar{G_\enh}})_1$ is an $A$-torsor, the set of sections $\tilde\beta|_{\spec(F) \xt S} : \spec(F) \xt S \to ((\oo{\wbar{G_\enh}})_1)_\eP$ lifting $\beta$ has a simply transitive 
action by $A(\spec(F)\xt S)$.
Since $A$ is finite over $k$ and $X$ is geometrically connected, we deduce that
$A(\spec(F)\xt S) = A(S)$. 
Thus the canonical $A(S)$-action on the set of sections $\tilde \beta : X_S \to (\wbar{G^\ssc_\enh})_\eP$ lifting $\beta$ is simply transitive. 
By definition of $\eY$, a lift $\tilde\beta$ exists after restricting along some fppf covering $S' \to S$. 
We conclude that $\VinBun^\true_G \to \eY$ is an $A$-torsor.
\end{proof}

\begin{prop}  \label{prop:trueproper}
The map $\VinBun^\true_G \to \VinBun_G$ is finite schematic. 
\end{prop}
\begin{proof}
We first show that the map $\VinBun^\true_G \to \VinBun_G$ is proper. 
By Lemma~\ref{lem:Buntorsor}, the map $\Bun_{G \xt_{G^\abst} G} \to \Bun_{G \xt_{G^\ab} G}$ is proper. 
Thus by base change, the map 
\[ \Maps^\circ\left(X, (\wbar{G_\enh})_1 / (G \xt_{G^\abst} G)\right) \to \Maps^\circ\left(
X, (\wbar{G_\enh})_1 / (G \xt_{G^\ab} G) \right) = \VinBun_G \] 
is proper. Composing this map with \eqref{e:VinBuncos}, which is finite by Corollary~\ref{cor:VinBuncos}, we conclude that $\VinBun^\true_G \to \VinBun_G$ is proper.

Next we prove that the map $\VinBun^\true_G \to \VinBun_G$ is schematic. 
Let $S$ be an affine scheme. A map $S \to \VinBun_G$ is the datum of a $G \xt_{G^\ab} G$-torsor $\eP$ on $X \xt S$ and a $G \xt_{G^\ab} G$-equivariant map 
$\eP \to (\wbar{G_\enh})_1$. Moreover, there is an open subset $\oo X \subset X$ 
such that $\eP|_{\oo X \xt S}$ is sent to $(\oo{\wbar{G_\enh}})_1$. 
Recall that we have a surjective homomorphism $G\xt_{G^\abst} G \to G \xt_{G^\ab} G$ 
with kernel $A:=\ker(G^\ssc \to [G,G])$. 
Then an $S'$-point of the fiber product $\eY:= S \xt_{\VinBun_G} \VinBun^\true_G$ 
parametrizes a $G \xt_{G^{\abst}} G$-torsor $\tilde\eP$ on $X \xt S'$ and a $G \xt_{G^\abst} G$-equivariant map $\tilde\beta: \tilde\eP \to \wbar{G^\ssc_\enh}$ such that 
the $G \xt_{G^\ab} G$-torsor induced by $\tilde\eP$ is isomorphic to $\eP|_{X \xt S'}$, and 
the diagram 
\[ \xymatrix{ \tilde \eP \ar[d]\ar[r]^{\tilde\beta} & \wbar{G^\ssc_\enh} \ar[d] \\ 
\eP|_{X \xt S'} \ar[r] & (\wbar{G_\enh})_1 } \]
commutes. This implies that $\tilde\eP|_{\oo X \xt S'}$ lands in $\oo{\wbar{G^\ssc_\enh}}$. 
Since $\oo{\wbar{G^\ssc_\enh}} \to (\oo{\wbar{G_\enh}})_1$ is an $A$-torsor, we 
get an isomorphism 
\[ \tilde \eP |_{\oo X \xt S'} \cong \eP|_{\oo X \xt S'} \xt_{(\oo{\wbar{G_\enh}})_1} \oo{\wbar{G^\ssc_\enh}}, \] 
where the r.h.s.~only depends on the map $S \to \VinBun_G$.
Thus $(\tilde \eP,\tilde\beta)$ are determined by their restrictions to 
the formal completion of $(X-\oo X)\xt S'$ in $X \xt S'$. 
Using twisted versions of the affine Grassmannian, we deduce that the fiber product $\eY$ is a closed subscheme of a projective ind-scheme over $S$. Since $\eY$ is of finite type, we conclude that $\eY$ is a scheme. 

We have shown that $\VinBun^\true_G \to \VinBun_G$ is a proper schematic map. One observes from Lemma~\ref{lem:Buntorsor} and Corollary~\ref{cor:VinBuncos} that $\VinBun^\true_G \to \VinBun_G$ is also quasi-finite. 
Therefore this map is finite schematic.
\end{proof}

\begin{lem} \label{lem:Gstratclosure}
The closure of the open substack $\VinBun_{G,G}$ in $\VinBun_G$ intersects the stratum $\VinBun_{G,P}^{\mu,\lambda}$ only if $\mu-\lambda \in \Lambda^\pos_{G,P}$. 
\end{lem}
\begin{proof}
The image of the finite map $\VinBun^\true_G \to \VinBun_G$ is a closed substack containing the $G$-stratum. 
Let $P$ be a standard parabolic subgroup of $G$ with Levi factor $M$.
Let $\tilde G$ denote $G^\ssc \xt Z(G)$, and 
let $\tilde M, \tilde P$ denote the preimages of $M,P$ under the isogeny $\tilde G \to G$. 
We have the corresponding boundary degeneration $\bbX_{\tilde P} \subset \wbar{\tilde G_\enh}$ and its affine closure $\wbar \bbX_{\tilde P}$.
Define the closed embedding $\tilde M \into \bbX_{\tilde P}$ as in \S\ref{sect:defect}, and let $\wbar{\tilde M}$ denote the closure of $\tilde M$ in $\wbar \bbX_{\tilde P}$. 
For a place $v$ of $X$, Remark~\ref{rem:heckeHN} implies that 
\[ \tilde M(\fo_v) \bs (\wbar{\tilde M}(\fo_v) \cap \tilde M(F_v)) / \tilde M(\fo_v) = \Lambda^{\pos,\bbQ}_{U} \cap \Lambda_{\tilde G},\] 
and $\Lambda^{\pos,\bbQ}_{U} \cap \Lambda_{\tilde G} = \Lambda^\pos_{U}$
since $[\tilde G, \tilde G]$ is simply connected. 
Thus we deduce from the construction of the defect stratification in \S\ref{sect:defect} that the image of $\VinBun^\true_G \to \VinBun_G$ 
intersects $\VinBun_{G,P}^{\mu,\lambda}$ if and only if $\mu-\lambda \in \Lambda^\pos_{G,P}$. This implies the lemma.
\end{proof}

\subsection{The function $b$}
Suppose $k=\bbF_q$. Let 
\[ \Delta : \Bun_G \to \Bun_G \xt \Bun_G \] 
denote the diagonal morphism.
Given $G$-bundles $\eF^1_G, \eF^2_G \in \Bun_G(\bbF_q)$, let 
$b(\eF^1_G, \eF^2_G)$ denote the trace of the geometric Frobenius 
acting on the $*$-stalk of the complex $\Delta_*(\wbar\bbQ_\ell)$ over the point
$(\eF^1_G, \eF^2_G) \in (\Bun_G \xt \Bun_G)(\wbar\bbF_q)$.

\subsubsection{}  \label{sss:asympVin}
Recall from \S\ref{sss:Asymp} that 
$\Asymp_P(\delta_K)$ is a $K \xt K$-invariant function in 
$C^\infty_b(\bbX_P(\bbA))$, where 
$\delta_K$ is the characteristic function of $K$ on $G(\bbA)$. 
Let \[ \beta : X \to (\wbar \bbX_P)_{\eF_G^1, \eF_G^2} \] denote a 
section that generically lands in the non-degenerate locus $(\bbX_P)_{\eF^1_G, \eF^2_G}$. 
Then for any $v \in \abs X$, 
choosing trivializations of $\eF^i_G \xt_X \spec(\fo_v)$ 
defines an isomorphism $(\wbar \bbX_P)_{\eF_G^2, \eF_G^2}(\fo_v) 
\cong \wbar\bbX_P(\fo_v)$. This defines an element $\beta_v \in \bbX_P(F_v) \cap \wbar\bbX_P(\fo_v)$, and the $K_v \xt K_v$-orbit of $\beta_v$ does not depend on the choice of trivializations. 
Non-degeneracy of $\beta$ implies that $(\beta_v) \in \bbX_P(\bbA)$. 
We define $\Asymp_P(\delta_K)(\beta)$ to be the value of $\Asymp_P(\delta_K)$
at this adelic point.

\begin{thm} \label{thm:b} 
Let $E = \wbar\bbQ_\ell$. We have an equality 
\[ b(\eF_G^1, \eF_G^2) = \sum_P (-1)^{\dim Z(M)} \sum_\beta \Asymp_P(\delta_K)(\beta), \]
where $P$ ranges over the standard parabolic subgroups of $G$, 
and $\beta$ ranges over the non-degenerate sections $\beta : X \to (\wbar \bbX_P)_{\eF_G^1, \eF_G^2}$. 
\end{thm}

The strategy for proving Theorem~\ref{thm:b} was suggested by Drinfeld, and it
consists of compactifying the diagonal morphism of $\Bun_G$. 
The geometry of the compactification then reduces to a theorem of 
\cite{Schieder:gen}, and 
the corresponding Grothendieck functions are computed using the 
facts reviewed in Appendix \ref{s:factorization}. 
The proof of Theorem~\ref{thm:b} is given at the end of \S\ref{ss:i^*j_*}.

\subsection{Compactifications of the diagonal morphism of $\Bun_G$}
\label{ss:compactify}

The diagonal morphism $\Delta$ is in general not proper, and one would like to compactify 
it (e.g., to compute $*$-restrictions of $\Delta_*$). 
We first review the definition of the stack $\wbar\Bun'_G$ (denoted by $\wbar\Bun_G$ in \cite{Schieder:gen}), which is a compactification of 
the morphism $\Bun_G \xt \bbB Z(G) \to \Bun_G \xt \Bun_G$, which $\Delta$ factors through. 
For the purposes of this paper, we define a slightly different stack $\wbar\Bun_G$, which is a compactification of $\Delta$ when $G$ is semisimple. 
When $G$ is not semisimple, $\wbar\Bun_G$ is not quite a compactification of $\Delta$, but it is equally good for our purposes.

\subsubsection{$\wbar\Bun'_G$}  \label{sss:barBun'G}
The action of $Z(G_\enh) = T$ on $\wbar{G_\enh}$ induces a $T$-action on $\VinBun_G$. Define $\wbar \Bun'_G = \VinBun_G / T$. 
There is an open embedding $\VinBun_{G,G} / T = \Bun_G \xt \bbB Z(G)
\into \wbar\Bun'_G$.
Observe that $\Delta$ factors as
\[ \Bun_G \to \Bun_G \xt \bbB Z(G) \into \wbar\Bun'_G \to \Bun_G \xt \Bun_G, \]
where $\bbB Z(G)$ is the classifying stack of $Z(G)$-bundles.
The following lemma is well-known: 

\begin{lem} \label{lem:barBunproper}
The map $\wbar \Bun'_G \to \Bun_G \xt \Bun_G$ is schematic and projective\footnote{The proof shows that there exists a coherent sheaf $\eF$ on the stack $\Bun_G \xt \Bun_G$ and a closed embedding $\wbar\Bun'_G \into \bbP(\eF)$.}.
\end{lem}
\begin{proof} 
Let $\eF_G^1, \eF_G^2 \in \Bun_G(S)$ for a test scheme $S$. 
Let $\Sect(X_S, (\wbar{G_\enh})_{\eF_G^1, \eF_G^2})$ denote the 
$S$-scheme of sections for the fiber bundle $(\wbar{G_\enh})_{\eF_G^1, \eF_G^2} \to X_S := X \xt S$. 
Then 
\begin{equation}\label{e:Sect} \wbar\Bun'_G \xt_{\Bun_G \xt \Bun_G} S \cong 
\Sect^\circ(X_S, (\wbar{G_\enh})_{\eF_G^1, \eF_G^2}) / T, 
\end{equation}
where the superscript $^\circ$ denotes the open locus of sections generically
landing in $(\oo{\wbar{G_\enh}})_{\eF^1_G,\eF^2_G}$. 
We wish to show \eqref{e:Sect} is a projective scheme over $S$.

Let $\Delta(\check\lambda)$ denote the Weyl $G$-module of highest weight $\check\lambda \in \check\Lambda^+_G$. It is known from the general theory of reductive monoids (cf.~proof of \cite[Proposition 2.3.2]{Wa}) that there exists a finite\footnote{In the particular case of the Vinberg semigroup, one can deduce that this map is a closed embedding using \cite[Exercises 6.1E, 6.2E]{BrK}.} map 
\begin{equation}\label{e:VinPluck}
    \wbar{G_\enh} \to \prod \End(\Delta(\check\lambda) \ot k_{\check\lambda}) \xt \wbar{T_\adj}, 
\end{equation}
where the product ranges over any \emph{finite} set of generators for the monoid $\check\Lambda^+_G$. The image of \eqref{e:VinPluck} satisfies the Pl\"ucker relations (cf.~\cite{BG,BFGM}). Therefore by considering the $G_\enh$-modules $\det(\Delta(\check\lambda) \ot k_{\check\lambda}) = \det(\Delta(\check\lambda)) \ot k_{\dim(\Delta(\check\lambda)) \check\lambda}$, we see that the composition of
\eqref{e:VinPluck} with the projection to $\prod \End(\Delta(\check\lambda) \ot k_{\check\lambda})$ is a finite map.
By Lemma~\ref{lem:Mapsfin}, we can reduce to showing that 
\begin{equation} \label{e:SectV}
    \Sect^\circ(X_S, \prod \End(\Delta(\check\lambda)\ot k_{\check\lambda})_{\eF_G^1,\eF_G^2}) /T
\end{equation}
is representable by a projective scheme over $S$. 
Here the superscript $^\circ$ denotes the locus of maps generically landing in $\prod (\End(\Delta(\check\lambda) \ot k_{\check\lambda})-\{0\})_{\eF^1_G,\eF^2_G}$.

For a test scheme $S'\to S$, an $S'$-point of the stack \eqref{e:SectV} is the data of $((\eF_T)_{S'}, \beta_{\check\lambda})$ where $(\eF_T)_{S'}$ is a $T$-bundle over $S'$ and $\beta_{\check\lambda}$ is an $\eO_{X \xt S'}$-module map 
\[ \Delta(\check\lambda)_{\eF_G^2} \ot_{\eO_S} \eL_{\check\lambda} \to \Delta(\check\lambda)_{\eF_G^1} \ot_{\eO_S} \eO_{S'} \] 
that is generically nonzero over all geometric points of $S'$,
where $\eL_{\check\lambda} = (k_{\check\lambda})_{(\eF_T)_{S'}}$ is the corresponding line bundle on $S'$. 
Observe that $\beta_{\check\lambda}$ is equivalent to an $S'$-fiberwise nonzero map 
\begin{equation} \label{e:Quot}
    \pi'^*(\eL_{\check\lambda}) \to (\Delta(\check\lambda)^*_{\eF^2_G} \ot \Delta(\check\lambda)_{\eF^1_G}) \ot_{\eO_S} \eO_{S'}, 
\end{equation}
where $\pi'$ is the projection $X \xt S' \to S'$. 
Set $\eE = \Delta(\check\lambda)^*_{\eF^2_G} \ot \Delta(\check\lambda)_{\eF^1_G}$, which is a locally free $\eO_{X \xt S}$-module. 
Let $\pi$ denote the projection $X \xt S \to S$, and observe that $\pi_*(\eE)$ 
is a perfect complex that commutes with base change (here and elsewhere, $\pi_*$ denotes the derived direct image functor).  
Then by adjunction, \eqref{e:Quot} is equivalent to a map in the derived category of coherent sheaves on $S'$ 
\[ \eL_{\check\lambda} \to \pi_*(\eE) \ot_{\eO_S} \eO_{S'} \] 
that is nonzero on every fiber of $S'$ (here the tensor product is derived). 
Applying derived $\eHom(?, \eO_{S'})$, this map is equivalent to a fiberwise nonzero map
\begin{equation} \label{e:projspace}
    \eHom(\pi_*(\eE), \eO_S) \ot_{\eO_S} \eO_{S'} \to \eL_{\check\lambda}^{-1}.
\end{equation}
Since $\pi_*(\eE) \ot_{\eO_S} k_s$ lives in cohomological degrees $0,1$ for any point $\spec(k_s) \to S$, we deduce that $\pi_*(\eE)$ is locally quasi-isomorphic to a complex of locally free $\eO_S$-modules living in degrees $0,1$. Therefore $\eHom(\pi_*(\eE),\eO_S)$ lives in cohomological degrees $-1,0$.
We conclude that \eqref{e:projspace} is equivalent to a surjection of $\eO_{S'}$-modules
\[ H^0\eHom(\pi_*(\eE),\eO_S) \ot_{\eO_S} \eO_{S'} \to \eL_{\check\lambda}^{-1}, \]
and $H^0\eHom(\pi_*(\eE),\eO_S)$ commutes with base change. 
We have shown that for fixed $\check\lambda$, the data $(\eL_{\check\lambda}^{-1}, \beta_{\check\lambda})$ defines an $S'$-point of the
projective $S$-scheme $\mathbf{Proj}_S \Sym_{\eO_S}( H^0\eHom(\pi_*\eE, \eO_S))$. 
By the Pl\"ucker relations, we conclude that the stack \eqref{e:SectV} is representable by a closed subscheme of a projective $S$-scheme.
\end{proof}

\subsubsection{}
Let $Z_0(G)$ denote the neutral connected component of the center of $G$. 
Then we have a finite map $T/Z_0(G) \to T/Z(G) = T_\adj$.
The character lattice of $T_\adj$ corresponds to the root lattice in $\check\Lambda$. 
If the Langlands dual group $\check G$ does not have a simply connected derived group, then the root lattice is not 
saturated in $\check\Lambda$, so $T/Z_0(G) \ne T/Z(G)$ in general. 

Recall that $k[\wbar{T_\adj}]$ is the semigroup algebra of 
$\check\Lambda^\pos_G$. 
Define $\wbar{T/Z_0(G)}$ so that $k[\wbar{T/Z_0(G)}]$
is the semigroup algebra of $\check\Lambda^{\pos,\bbQ}_G \cap \check\Lambda$.
There is a natural finite map 
\[ \wbar{T/Z_0(G)} \to \wbar{T_\adj}  \]
extending the map $T/Z_0(G) \to T_\adj$.

\subsubsection{$\wbar\Bun_G$} \label{sss:barBunG}
Consider the base change 
\[ (\VinBun_G)_{\wbar{T/Z_0(G)}}:= \VinBun_G \xt_{\wbar{T_\adj}} \wbar{T/Z_0(G)}\] over $\wbar{T/Z_0(G)}$. 
Then $T$ acts diagonally on $(\VinBun_G)_{\wbar{T/Z_0(G)}}$,
and we define 
\[ \wbar\Bun_G = (\VinBun_G)_{\wbar{T/Z_0(G)}} / T. \]
Since $\bar\pi_{\Bun}^{-1}(T_\adj) = \VinBun_{G,G} = \Bun_G \xt T_\adj$, we see that
there is an open embedding $\Bun_G \xt \bbB Z_0(G) \into \wbar\Bun_G$. 
There is a natural finite map $\wbar\Bun_G \to \wbar\Bun'_G$, and we have the
commutative diagram
\[ \xymatrix{ \Bun_G \ar[r] & \Bun_G \xt \bbB Z_0(G) \ar@{^{(}->}[r] \ar[d] & \wbar\Bun_G \ar[d]  \\ 
& \Bun_G \xt \bbB Z(G) \ar@{^{(}->}[r] & \wbar\Bun'_G \ar[r] & \Bun_G \xt \Bun_G } \]
factoring the diagonal $\Delta$. 
Then the composite map $\wbar\Delta : \wbar\Bun_G \to \Bun_G \xt \Bun_G$ 
is also proper, so $\wbar\Bun_G$ is a ``compactification'' of $\Delta$.
This is the compactification that we will use to prove Theorem~\ref{thm:b}.

\begin{eg}
Let $G = \on{SL}(2)$. Then $Z_0(G)=\{1\}$, and the map $\wbar{T/Z_0(G)} \to \wbar{T_\adj}$ corresponds to the map $\bbA^1 \to \bbA^1 : \epsilon \mapsto \epsilon^2$. 
An $S$-point of $\wbar\Bun_G$ is 
a collection $(\eL_1, \eL_2, l, \beta,\epsilon)$, where 

(a) $\eL_1,\eL_2$ are rank $2$ vector bundles on $X \xt S$ with trivializations of their determinants,

(b) $l$ is a line bundle on $S$,

(c) $\beta \in\Hom (\eL_2, \eL_1) \ot l$ is not equal to $0$ on $X \xt s$ for every geometric point $s \to S$, 

(d) $\epsilon\in l$, and 

(e) the equation $\det \beta=\epsilon^2$ holds. 

\noindent In comparison, an $S$-point of $\wbar \Bun'_G$ is a collection $(\eL_1,\eL_2, l, \beta)$ satisfying (a)--(c) above.
\end{eg}

\subsubsection{}
We have a Cartesian square
\begin{equation} \label{e:VinbarBun}
\xymatrix{ \Bun_G \xt T \ar[r] \ar[d] & \Bun_G \ar[d] \\ 
\Bun_G \xt (T/Z_0(G)) \ar@{^{(}->}[d]_\jmath \ar[r] & \Bun_G \xt \bbB Z_0(G)\ar@{^{(}->}[d] \\ 
(\VinBun_G)_{\wbar{T/Z_0(G)}} \ar[r] & \wbar\Bun_G } 
\end{equation}
where the horizontal maps are $T$-torsors, and the lower vertical maps are open embeddings.
Since $Z_0(G)$ is connected, $T \to T/Z_0(G)$ is a trivial $Z_0(G)$-bundle.
Therefore the pushforward of the constant sheaf $(\wbar\bbQ_\ell)_T$
to $T/Z_0(G)$ equals $(\wbar\bbQ_\ell)_{T/Z_0(G)} \ot H^*(Z_0(G), \wbar\bbQ_\ell)$. 
Thus by smooth base change, to compute the function $b$ it suffices to compute the trace of the geometric Frobenius acting on the $*$-stalks of
$\jmath_* \wbar\bbQ_\ell$.

\subsection{The $*$-extension of the constant sheaf} \label{ss:i^*j_*}
Let 
\[ \jmath : \Bun_G \xt (T/Z_0(G)) \into (\VinBun_G)_{\wbar{T/Z_0(G)}}  \]
denote the open embedding.
We want to compute the $*$-restriction of $\jmath_* \wbar\bbQ_\ell$ to 
the strata $\VinBun_{G,P}^{\mu,\lambda} \xt_{\wbar{T_\adj}} \wbar{T/Z_0(G)}$
for $\mu,\lambda \in \Lambda_{G,P}$.

Recall that the $P$-locus $(\wbar{T_\adj})_P$ is isomorphic to 
$T/Z(M)$. 
\begin{lem}  \label{lem:rad0}
    The reduced part of $(\wbar{T_\adj})_P \xt_{\wbar{T_\adj}} \wbar{T/Z_0(G)}$ 
is isomorphic to $T/Z_0(M)$. 
\end{lem}
\begin{proof}
The locally closed embedding $(\wbar{T_\adj})_P = \bbG_m^{\abs{\Gamma_M}} \xt \{0\} \into \wbar{T_\adj} = (\bbA^1)^{\abs{\Gamma_G}}$ 
identifies with the spectrum of the algebra map 
\begin{equation} \label{e:pr_M}
k[\check\alpha_j, j \in \Gamma_G] \to k[\check\alpha_i^{\pm 1}, i \in \Gamma_M] 
\end{equation}
sending $\check\alpha_i \mapsto \check\alpha_i$ for $i \in \Gamma_M$ and 
$\check\alpha_j \mapsto 0$ for $j \in \Gamma_G - \Gamma_M$ (here we consider $\check\alpha_i$ as a character and use the multiplicative notation).
Note that $k[\check\alpha_j, j\in \Gamma_G]$ is the semigroup algebra 
of $\check\Lambda^\pos_G$. 

The projection $\wbar{T/Z_0(G)} \to \wbar{T/Z(G)}$ is the spectrum of the inclusion of semigroup algebras
\[ k[\check\Lambda^\pos_G] \into k[\check\Lambda^{\pos,\bbQ}_G \cap \check\Lambda]. \]
Therefore $(\wbar{T_\adj})_P \xt_{\wbar{T_\adj}} \wbar{T/Z_0(G)}$ is the spectrum of the algebra
\begin{equation} \label{e:semialg}
    k[\check\alpha_i^{\pm 1}, i\in \Gamma_M] \ot_{k[\check\Lambda^\pos_G]} 
k[\check\Lambda^{\pos,\bbQ}_G \cap \check\Lambda]. 
\end{equation}
Since the map \eqref{e:pr_M} sends $\check\alpha_j \mapsto 0$ for $j \in \Gamma_G - \Gamma_M$, the reduced algebra of \eqref{e:semialg}
equals 
\begin{equation} \label{e:semialgr}
    k[\check\alpha_i^{\pm 1}, i\in \Gamma_M] \ot_{k[\check\Lambda^\pos_M]} k[\check\Lambda^{\pos,\bbQ}_M \cap \check\Lambda].
\end{equation}
Since the non-negative integral span of $\check\Lambda^{\pos,\bbQ}_M \cap \check\Lambda$ and $-\check\alpha_i,\, i \in \Gamma_M$ is equal to the lattice
$\check\Lambda_{T/Z_0(M)} \subset \check\Lambda$, 
the algebra \eqref{e:semialgr}  equals $k[T/Z_0(M)]$.
\end{proof}

Recall from \eqref{eqn:Pfiber} that we have 
an isomorphism $\VinBun_{G,P} \cong \VinBun_{G,\mathbf c_P} \xt (T/Z(M))$. 
Lemma~\ref{lem:rad0} implies that we have embeddings 
\begin{equation} \label{e:strata0}
    \iota_P : \VinBun_{G,\mathbf c_P} \xt (T/Z_0(M)) \into (\VinBun_G)_{\wbar{T/Z_0(G)}} 
\end{equation}
that form a stratification as $P$ ranges over all standard parabolic subgroups. 

For $\mu,\lambda \in \Lambda_{G,P}$, let 
\[ \iota^{\mu,\lambda}_P : \VinBun_{G,\mathbf c_P}^{\mu,\lambda} \xt (T/Z_0(M)) \into (\VinBun_G)_{\wbar{T/Z_0(G)}} \] 
denote the locally closed embedding defined by \eqref{e:strata0} 
and the defect stratification from \S\ref{sect:defect}. 

The following is proved in \cite[Theorem B]{Schieder:gen} in the case $P=B$. We give a proof using Proposition~\ref{prop:geom-nu} at the end of this Appendix. 
\begin{thm} \label{thm:i^*j_*}
Suppose $k=\bbF_q$. 
    The trace of geometric Frobenius on $*$-stalks of $\iota^*_P (\jmath_* (\wbar\bbQ_\ell))$ sends
\begin{equation} \label{eqn:Asympf}
	(\eF_G^1, \eF_G^2, \beta, t) \in (\VinBun_{G,\mathbf c_P} \xt T/Z_0(M))(\bbF_q) \mapsto 
    (1-q)^{\abs{\Gamma_G}-\abs{\Gamma_M}} \Asymp_P(\delta_K)(\beta). 
\end{equation}
\end{thm}
\noindent Here we use Lemma~\ref{lem:X_PVin} to identify $\beta$ with a 
section $X \to (\wbar\bbX_P)_{\eF_G^1, \eF_G^2}$, and $\Asymp_P(\delta_K)(\beta)$ is defined in \S\ref{sss:asympVin}. 

\medskip

Assuming Theorem~\ref{thm:i^*j_*}, we prove Theorem~\ref{thm:b}.

\begin{proof}[Proof of Theorem~\ref{thm:b}]
Let $\wbar\Bun_G$ denote the compactification of 
$\Delta$ defined in \S\ref{sss:barBunG}. 
Factor $\Delta$ into $\bar\jmath: \Bun_G \to \wbar\Bun_G$ 
and the proper map $\wbar\Delta : \wbar\Bun_G \to \Bun_G \xt \Bun_G$. 
For a sheaf $\eF$, we will use $f_\eF$ to denote 
its Grothendieck function, i.e., the trace of the geometric Frobenius
acting on the $*$-stalks over the $\bbF_q$-points. 
By the Grothendieck-Lefschetz trace formula, 
$b(\eF^1_G, \eF^2_G)$ equals the sum of the values of 
$f_{\bar\jmath_* \wbar\bbQ_\ell}$ at the points of $\wbar\Bun_G(\bbF_q)$ lying over $(\eF^1_G,\eF^2_G)\in (\Bun_G \xt \Bun_G)(\bbF_q)$.

We have the $T$-torsor $(\VinBun_G)_{\wbar{T/Z_0(G)}} \to \wbar\Bun_G$. 
Let 
\[ \jmath : \Bun_G \xt T/Z_0(G) \into (\VinBun_G)_{\wbar{T/Z_0(G)}} \] denote
the open embedding. Recall that $T \to T/Z_0(G)$ is a trivial $Z_0(G)$-bundle, so the pushforward of $(\wbar\bbQ_\ell)_T$ to 
$T/Z_0(G)$ equals $(\wbar\bbQ_\ell)_{T/Z_0(G)} \ot H^*(Z_0(G), \wbar\bbQ_\ell)$. 
The trace of the geometric Frobenius acting on $H^*(Z_0(G),\wbar\bbQ_\ell)$
equals $(1-q)^{\dim(Z_0(G))}$. 
Since any $T$-torsor over $\bbF_q$ is trivial, we deduce
from the Cartesian square \eqref{e:VinbarBun} and smooth base change that 
\[ b(\eF^1_G, \eF^2_G) = (-1)^{\dim T} (1-q)^{-\abs{\Gamma_G}} 
    \sum_{\tilde\beta} f_{\jmath_* \wbar\bbQ_\ell}(\tilde\beta), 
\]
where the sum is over $\tilde\beta \in (\VinBun_G)_{\wbar{T/Z_0(G)}}(\bbF_q)$ 
mapping to $(\eF_G^1, \eF_G^2)$.
From \eqref{e:strata0}, we have a stratification of 
$(\VinBun_G)_{\wbar{T/Z_0(G)}}$ by 
$\VinBun_{G,\mathbf c_P} \xt (T/Z_0(M))$, 
where $P$ ranges over all standard parabolic subgroups. 
Theorem~\ref{thm:i^*j_*} implies that
\[ f_{\iota_P^* \jmath_* \wbar\bbQ_\ell} (\eF^1_G, \eF^2_G, \beta, t) = 
(1-q)^{\abs{\Gamma_G} - \abs{\Gamma_M}} \Asymp_P(\delta_K)(\beta) \]
for $\beta : X \to (\wbar\bbX_P)_{\eF^1_G, \eF^2_G}$ a non-degenerate section
and $t \in (T/Z_0(M))(\bbF_q)$. 
Putting it all together, we prove the theorem.
\end{proof}

The remainder of this Appendix works towards setting up the proof of Theorem~\ref{thm:i^*j_*}, which is given at the end.

\subsection{Reduction to the Hecke stack} 
The isomorphism 
\[ \Bun_P^\mu \xt_{\Bun_M^\mu} \eH_M^{+,\mu,\lambda} \xt_{\Bun_M^\lambda} \Bun_{P^-}^\lambda \cong \VinBun_{G,\mathbf c_P}^{\mu,\lambda} \]
induced by \eqref{eqn:defect} allows us to define the projection map
\[ \pr^{\mu,\lambda}_M : \VinBun_{G,\mathbf c_P}^{\mu,\lambda} \to \eH_M^{+,\mu,\lambda}, \]
which is smooth with equidimensional fibers.

For $(\eF_G^1, \eF_G^2, \beta) \in \VinBun_{G,\mathbf c_P}^{\mu,\lambda}(\bbF_q)$, 
let $(\eF_M^1, \eF_M^2, \beta_M) := \pr^{\mu,\lambda}_M(\beta) \in \eH_M^{+,\mu,\lambda}(\bbF_q)$. 
Choosing trivializations $\eF_M^1,\eF_M^2$ over $\spec(\fo_v)$, the $\wbar M$-morphism 
$\beta_M$ defines an element $(m_v)$ in the restricted product
$\prod_v (\wbar M(\fo_v) \cap M(F_v))$ with respect to the open subgroups $M(\fo_v) \subset M(F_v)$. 
The $M(\fo_v)\xt M(\fo_v)$-orbit of $m_v$ does not depend on the choice of trivializations. 
One deduces from \eqref{eqn:Asymp0} that 
\begin{equation} \label{e:Asympnu}
    \Asymp_{P,v}(\delta_{K_v})(\beta_v) = \nu_{M,v}(m_v),
\end{equation}
where $\nu_{M,v}$ is the $K_{M,v}$-bi-invariant measure on 
$M(F_v)$ defined in \S\ref{def:nu}. 
Thus the function \eqref{eqn:Asympf} reduces to a function on $\eH^{+,\mu,\lambda}_M(\bbF_q)$.

On the other hand, we have a similar reduction 
of $\iota^{\mu,\lambda*}_P (\jmath_* (\wbar\bbQ_\ell))$ 
by a modified version of \cite[Theorem 4.3.1]{Schieder:gen}:

Recall from Lemma~\ref{lem:Gstratclosure} that 
$\iota^{\mu,\lambda*}_P(\jmath_* (\wbar\bbQ_\ell)) = 0$ unless 
$\mu-\lambda \in \Lambda^\pos_{G,P}$. 
Assuming that $\mu-\lambda \in \Lambda^\pos_{G,P}$, the relevant definitions and results from \S\ref{ss:tildeU} still hold, without the assumption that $[G,G]$ is simply connected.

\begin{thm} \label{thm:Schieder}
Let $\mu,\lambda \in \Lambda_{G,P}$ with $\mu-\lambda \in \Lambda^\pos_{G,P}$.
The $*$-restriction $\iota^{\mu,\lambda *}_P (\jmath_* (\wbar\bbQ_\ell))$ to the
stratum $\VinBun_{G,\mathbf c_P}^{\mu,\lambda} \xt (T/Z_0(M))$ is equal to
\[ \pr^{\mu,\lambda *}_M \left( \wbar\bbQ_\ell \tbt \tilde\Upsilon(\check\fu_P)^{\mu-\lambda} \right) \bt (\wbar\bbQ_\ell({\ts \frac 1 2})[1])^{-\brac{2\check\rho_P,\mu-\lambda}}
\ot H^*(Z_0(M)/Z_0(G),\wbar\bbQ_\ell). 
\]
\end{thm}
\noindent 
Here $\tilde\Upsilon(\check\fu_P)^{\mu-\lambda}$ is the factorization algebra 
on $\Gr^+_{M,X^{\mu-\lambda}}$ defined in \S\ref{ss:tildeU}, and we 
can form the sheaf $(\wbar\bbQ_\ell)_{\Bun_M} \tbt \tilde\Upsilon(\check\fu_P)^{\mu-\lambda} \in D(\eH^+_{M,X^{\mu-\lambda}})$ using \eqref{e:HeGr+}. 

The proof of Theorem~\ref{thm:Schieder} follows the same reasoning 
as the proof of \cite[Theorem 4.3.1]{Schieder:gen}, using the local models
defined in \emph{loc.~cit}, which we now review. 

\subsection{Local models} 
Recall that $\bar\pi : \wbar{G_\enh} \to \wbar{T_\adj}$ denotes the projection
and $\bar \fs: \wbar{T_\adj} \to \wbar{G_\enh}$ is a section. 
Let $(\wbar{G_\enh})_{\ge P} := \bar\pi^{-1}( (\wbar{T_\adj})_{\ge P})$ 
denote the open submonoid. 

Define $\eY^P$ to be the scheme representing the substack 
\[ \Maps^\circ(X, U^- \bs (\wbar{G_\enh})_{\ge P} / P)
\subset \Bun_{U^-} \xt_{\Bun_G} \VinBun_{G,\ge P} \xt_{\Bun_G} \Bun_P\] 
of maps generically landing
in $U^- \cdot \bar\fs( (\wbar{T_\adj})_{\ge P} ) \cdot P$. 
Then $\bar \pi$ induces a map 
\[ \eY^P \to (\wbar{T_\adj})_{\ge P}.\] 

For $\theta \in \Lambda_{G,P}$, let $\eY^{P,\theta}$ denote the
preimage of $\Bun_M^{-\theta}$ under the projection $\Bun_P \to \Bun_M$.

\subsubsection{}
One can define both left and right actions of $T_\adj$ on 
$\eY^P$ in a similar way as in \S\ref{sect:Tact}. 
Then the action of $T/Z(M) \into T_\adj$ defines an isomorphism 
\[ \eY^P \cong \eY^P_\strict \xt (T/Z(M)),\] 
where $\eY^P_\strict := \eY^P \xt_{\wbar{T_\adj}} \wbar{Z(M)/Z(G)}$ is 
the local model for $\VinBun_{G,\ge P,\strict}$ considered in \cite[\S 6.1]{Schieder:gen}. 

\subsubsection{}
Let $e_P = \bar\fs(\mathbf c_P)$ and  
$(\wbar{G_\enh})_{\ge P,\strict} = \bar\pi^{-1}(\{1\}\xt\wbar{Z(M)/Z(G)})$. 
By \cite[Theorem 4.2.10]{Wa}, we have 
\[ e_P \cdot (\wbar{G_\enh})_{\ge P,\strict} \cdot e_P = 
e_P \cdot (\wbar{G_\enh})_{\mathbf c_P} \cdot e_P = \wbar M.\] 
The map $(\wbar{G_\enh})_{\ge P,\strict} \to \wbar M$ factors through
$U^- \bs (\wbar{G_\enh})_{\ge P,\strict} / U$. 
Therefore we have a map 
\[ \pi_\eY : \eY^{P,\theta} \to \Gr^+_{M,X^\theta} \xt (T/Z(M)). \]
The embedding $\wbar M = e_P \cdot (\wbar{G_\enh})_{\mathbf c_P} \cdot e_P 
\into (\wbar{G_\enh})_{\mathbf c_P}$ induces a section 
\[ \sigma_\eY : \Gr^+_{M,X^\theta} \xt (T/Z(M)) \to \eY^{P,\theta} \]
of $\pi_\eY$. Both $\pi_\eY$ and $\sigma_\eY$ are 
compatible with the projection \[ \eY^P \to (\wbar{T_\adj})_{\ge P} 
= \bbG_m^{\abs{\Gamma_M}} \xt (\bbA^1)^{\abs{\Gamma_G}-\abs{\Gamma_M}} \to 
\bbG_m^{\abs{\Gamma_M}} = T/Z(M).\]

\subsubsection{} The scheme $\eY^{P,\theta}$ is a local model 
for $\VinBun_{G,\ge P}$. We will need to consider
\[ \tilde\eY^{P,\theta} := \eY^{P,\theta} \xt_{\wbar{T_\adj}} \wbar{T/Z_0(G)}, \] 
which is a local model for $\VinBun_{G, \ge P} \xt_{\wbar{T_\adj}} \wbar{T/Z_0(G)}$. Let $(\wbar{T/Z_0(G)})_{\ge P}$ and $(\wbar{T/Z_0(G)})_P$ 
denote the base changes of the corresponding loci in $\wbar{T_\adj}$,
so that $\tilde\eY^{P,\theta}$ lies over $(\wbar{T/Z_0(G)})_{\ge P}$.
By base change, we get maps 
\[ \Gr^+_{M,X^\theta} \xt (\wbar{T/Z_0(G)})_P \overset{\tilde\sigma_\eY}\longrightarrow  \tilde\eY^{P,\theta} \overset{\tilde\pi_\eY}\longrightarrow \Gr^+_{M,X^\theta} \xt (\wbar{T/Z_0(G)})_P. \]
At the level of reduced schemes, Lemma~\ref{lem:rad0} implies that we 
have maps 
\[ (\Gr^+_{M,X^\theta})_\red \xt (T/Z_0(M)) \overset{\tilde\sigma_\eY}\longrightarrow \tilde\eY^{P,\theta}_\red \overset{\tilde\pi_\eY}\longrightarrow (\Gr^+_{M,X^\theta})_\red \xt (T/Z_0(M)). \]

\subsection{Contracting action on $\tilde\eY^{P,\theta}_\red$.}
\label{sss:contract}
Fix a cocharacter $\nu_M : \bbG_M \to Z_0(M) \subset T$ which contracts
$U^-$ to the element $1 \in U^-$. Then $\nu_M$
defines a $\bbG_m$-action on $\eY^{P,\theta}$ that contracts $\eY^{P,\theta}$
onto the section $\sigma_\eY$ by \cite[Lemma 6.5.6]{Schieder:gen}. 

By definition, the induced 
$\bbG_m$-action on $(\wbar{T_\adj})_{\ge P}$ is via the composition
$\bbG_m \overset{-2\nu_M}\longrightarrow T$ and the usual $T$-action on $\wbar{T_\adj}$. 
Therefore if we consider the $\bbG_m$-action on $\wbar{T/Z_0(G)}$ 
via the composition $\bbG_m \overset{-2\nu_M}\longrightarrow T$ and 
the usual $T$-action on $\wbar{T/Z_0(G)}$, we 
get a $\bbG_m$-action on $\tilde\eY^{P,\theta}$, and hence on $\tilde\eY^{P,\theta}_\red$.
Moreover from Lemma~\ref{lem:rad0} we deduce that
this $\bbG_m$-action contracts $\tilde\eY^{P,\theta}_\red$ onto 
the section $(\Gr^+_{M,X^\theta})_\red \xt (T/Z_0(M))$.

\subsubsection{} We are now ready to prove Theorem~\ref{thm:Schieder}.

\begin{proof}[Proof of Theorem~\ref{thm:Schieder}] 
Since the open locus $\VinBun_{G,\ge P}$ contains the $P$- and $G$-loci, 
to compute $\iota^{\mu,\lambda*}_P\jmath_* \wbar\bbQ_\ell$, we may restrict to 
$\VinBun_{G,\ge P} \xt_{\wbar{T_\adj}} \wbar{T/Z_0(G)}$. 

Let $\eY^{P,\theta}_G \subset \eY^{P,\theta}$ denote
the $G$-locus: the preimage of $T_\adj$ under the projection to $\wbar{T_\adj}$.
 Set $\tilde\eY^{P,\theta}_G = \eY^{P,\theta}_G \xt_{T_\adj} (T/Z_0(G))$. 
The $T_\adj$-action on $\eY^{P,\theta}$ induces an isomorphism 
\[ \eY^{P,\theta}_G \cong \oo\eZ{}^{P,\theta} \xt T_\adj, \]
where $\oo\eZ{}^{P,\theta}$ is the open Zastava space (see \S\ref{sss:Zastava}).
Hence there is an isomorphism 
\[ \tilde \eY^{P,\theta}_G \cong \oo\eZ{}^{P,\theta} \xt (T/Z_0(G)).\]

Let $\jmath_G : \tilde\eY^{P,\theta}_G \into \tilde\eY^{P,\theta}$ denote
the open embedding. Then the assertion of the theorem
reduces, as explained in \cite[\S 3, \S 8]{BFGM}, to proving that 
\begin{equation} \label{e:locali*j*}
    \tilde\sigma_\eY^* (\jmath_{G*}( \wbar\bbQ_\ell)) \cong 
    \tilde\Upsilon(\check\fu_P)^\theta \bt 
    (\wbar\bbQ_\ell({\ts\frac 1 2})[1])^{-\brac{2\check\rho_P, \theta}} \ot 
H^*(Z_0(M)/Z_0(G), \wbar\bbQ_\ell). 
\end{equation}
Since we are working with \'etale sheaves, we can work 
at the level of reduced schemes. Then we can apply the
contraction principle (cf.~\cite[\S 5]{BFGM}, \cite[Lemma 7.2.1]{Schieder:gen}) to the $\bbG_m$-action on $\tilde\eY^{P,\theta}_\red$ 
defined in \S\ref{sss:contract}. 
This gives an isomorphism $\tilde\sigma_\eY^*(\jmath_{G*}(\wbar\bbQ_\ell)) \cong 
\tilde\pi_{\eY *}(\jmath_{G*}(\wbar\bbQ_\ell))$.
At the level of reduced schemes, 
\[ \tilde\pi_\eY \circ \jmath_G : \eY^{P,\theta}_{G,\red} \cong 
    \oo\eZ{}^{P,\theta} \xt (T/Z_0(G)) \to (\Gr^+_{M,X^\theta})_\red \xt (T/Z_0(M)) \]
is the product of $\oo \pi_\eZ$ and the natural projection 
$T/Z_0(G) \to T/Z_0(M)$. 
Recall from \eqref{e:OmegaZ} that there is a canonical isomorphism
$\tilde\Upsilon(\check\fu_P)^\theta \cong (\oo \pi_\eZ)_*(\IC_{\oo \eZ{}^{P,\theta}})$. Noting that $\oo\eZ{}^{P,\theta}$ is smooth of dimension $\brac{2\check\rho_P,\theta}$, we get the identification 
\[ (\oo \pi_\eZ)_*(\wbar\bbQ_\ell) \cong \tilde\Upsilon(\check\fu_P)^\theta \ot (\wbar\bbQ_\ell({\ts\frac 1 2})[1])^{-\brac{2\check\rho_P,\theta}}.\] 
Since $Z_0(M)/Z_0(G)$ is a torus, we observe that 
$T/Z_0(G) \to T/Z_0(M)$ is a trivial $Z_0(M)/Z_0(G)$-bundle. 
Equation \eqref{e:locali*j*}, and hence the theorem, now follows.
\end{proof}

\begin{proof}[Proof of Theorem~\ref{thm:i^*j_*}] 
The trace of the geometric Frobenius on $H^*(\bbG_m, \wbar\bbQ_\ell)$ equals $1-q$,
and $Z_0(M)/Z_0(G)$ is a product of $\abs{\Gamma_G} - \abs{\Gamma_M}$ copies of $\bbG_m$. 
Therefore the trace of geometric Frobenius on $H^*(Z_0(M)/Z_0(G),\wbar\bbQ_\ell)$ equals
$(1-q)^{\abs{\Gamma_G}-\abs{\Gamma_M}}$. Theorem~\ref{thm:i^*j_*} now follows 
from 
Theorem~\ref{thm:Schieder}, Proposition~\ref{prop:geom-nu}, and 
\eqref{e:Asympnu}.
\end{proof}

\bibliographystyle{amsplain}
\bibliography{asymptotics}

\end{document}